\documentclass[reqno,11pt,]{amsart}
\usepackage{amsmath,amssymb,latexsym}
\usepackage[all]{xy}
\usepackage[pdftex]{graphicx}
\usepackage{amsthm, amsfonts}
\usepackage[OT2,T1]{fontenc}
\usepackage{tikz-cd}
\usepackage{mathtools}
\usepackage{mathrsfs}
\usepackage[normalem]{ulem}
\usepackage{srcltx}
\usepackage[square,numbers]{natbib}
\usepackage[top=2.5cm, bottom=2.5cm, outer=2cm, inner=2cm]{geometry}
\usepackage[pagebackref=true, colorlinks, linkcolor=blue, citecolor=red]{hyperref}
\DeclareSymbolFont{cyrletters}{OT2}{wncyr}{m}{n}
\DeclareMathSymbol{\Sha}{\mathalpha}{cyrletters}{"58}
\numberwithin{equation}{section}
\theoremstyle{plain}
\newtheorem{theorem}{Theorem}[section]
\newtheorem{lemma}[theorem]{Lemma}
\newtheorem{corollary}[theorem]{Corollary}
\newtheorem{proposition}[theorem]{Proposition}
\theoremstyle{definition}
\newtheorem{definition}[theorem]{Definition}
\newtheorem{example}[theorem]{Example}
\newtheorem{remark}[theorem]{Remark}

\newtheorem{thmalph}{Theorem}

\newcommand{\C}{\color{black}}
\newcommand{\col}{\color{black}}
\newcommand{\colo}{\color{black}} %%blue
\newcommand{\cv}{\color{black}} %%violet

\newcommand{\DD}{\mathbb{D}}
\newcommand{\Z}{\mathbb{Z}}
\newcommand{\cO}{\mathcal{O}}

\newcommand{\Gal}{\operatorname{Gal}}
\newcommand{\Ker}{\operatorname{Ker}}

\newcommand{\Fr}{\operatorname{Fr}}
\newcommand{\Exp}{\operatorname{Exp}}

\newcommand{\End}{\operatorname{End}}
\newcommand{\card}{\operatorname{card}}
\newcommand{\trace}{\operatorname{trace}}
\newcommand{\trac}{\operatorname{tr}}
\newcommand{\Tot}{\operatorname{Tot}}
\newcommand{\co}{\operatorname{coker}}

\newcommand{\Res}{\operatorname{Res}}

\newcommand{\Aut}{\operatorname{Aut}}

\newcommand{\Keywords}[1]{\par\noindent
	{\small{Keywords and phrases}: #1}}
\newcommand{\AMS}[1]{\par\noindent
	{\small{AMS Subject Classification}: #1}}

\author{Chandrakant Aribam, Neha Kwatra \textsuperscript{1}}
\address{Indian Institute of Science Education and Research (IISER) Mohali, Knowledge City, Sector 81, Manauli, SAS Nagar, Punjab, 140306, India.}
\email{aribam@iisermohali.ac.in}
\address{Indian Institute of Science Education and Research (IISER) Mohali, Knowledge City, Sector 81, Manauli, SAS Nagar, Punjab, 140306, India.}
\email{nehakwatra0@gmail.com; nehakwatra@iisermohali.ac.in}
\date{}

\begin{document}
	\title{Galois Cohomology for Lubin-Tate $(\varphi_q,\Gamma_{LT})$-modules over coefficient rings}
\begin{abstract}
The classification of the local Galois representations using $(\varphi,\Gamma)$-modules by Fontaine has been generalized by Kisin and Ren over the Lubin-Tate extensions of local fields using the theory of $(\varphi_q,\Gamma_{LT})$-modules. In this paper, we extend the work of (Fontaine) Herr by introducing a complex which allows us to compute cohomology over the Lubin-Tate extensions and compare it with the Galois cohomology groups. We further extend that complex to include certain non-abelian extensions. We then deduce some relations of this cohomology with those arising from $(\psi_q,\Gamma_{LT})$-modules. We also compute the Iwasawa cohomology over the Lubin-Tate extensions in terms of {\C the} $\psi_q$-operator acting on the \'{e}tale $(\varphi_q,\Gamma_{LT})$-module attached to the local Galois representation. Moreover, we generalize the notion of $(\varphi_q,\Gamma_{LT})$-modules over the coefficient ring $R$ and show that the equivalence given by Kisin and Ren extends to the Galois representations over $R$. This equivalence allows us to generalize our results to the case of coefficient rings.		
\end{abstract}
\maketitle
\footnotetext[1]{The author is supported by the Ph.D. fellowship from IISER Mohali, and the results of this paper are part of the author's doctoral dissertation.}
\let\thefootnote\relax\footnotetext{
	\AMS{11F80, 11F85, 11S25, 11S31, 14F30.}
	}	
\section{Introduction}\label{sec1}
{\C Let $K$ be a finite extension of $\mathbb{Q}_p$ with residue field $k$, where $k$ is finite and of characteristic $p$. Fix an algebraic closure $\bar{K}$ of $K$, then }%%%%Let $K$ be a field complete with respect to a discrete valuation whose residue field $k$ is finite and of characteristic $p$, where $p$ is a fixed prime. In other words, $K$ is a local field, and 
we denote by $G_K = \Gal(\bar{K}/K)$ the {\C absolute} %%local 
Galois group. Recall that a $\mathbb{Z}_p$-adic representation of $G_K$ is a %%{\C free} 
$\mathbb{Z}_p$-module of finite {\C type together} %%rank 
with a continuous and linear action of $G_K$. For the Witt ring $W(k)$, let $\mathcal{O}_{\mathcal{E}}$ be the $p$-adic completion of $W(k)((u))$ with the field of fractions $\mathcal{E}$. Let {\C $K_{cyc}=\cup_{n\geq1}K(\mu_{p^n})$ %%$K_{cyc}$ be the cyclotomic $\mathbb{Z}_p$-extension
denote the %%the totally ramified 
abelian extension} of $K$ in $\bar{K}$ obtained by adjoining the $p^n$-th roots of unity to $K$, $H = \Gal(\bar{K}/K_{cyc})$ and $\Gamma =G_K/H= \Gal(K_{cyc}/K)$. Then there is a natural action of $\Gamma$ and a Frobenius $\varphi$ on $\mathcal{O}_{\mathcal{E}}$.

 In \cite{Fon}, Fontaine introduced a new technique to understanding the category of $\mathbb{Z}_p$-adic representations of $G_K$ in terms of algebraic objects, namely, the $(\varphi,\Gamma)$-modules. In the equal characteristic case ($(p,p)$ case), he constructed a category of \'{e}tale $\varphi$-modules over $\mathcal{O}_{\mathcal{E}}$ and proved that this category is equivalent to the category of $\mathbb{Z}_p$-adic representations of $G_K$. {\C Recall that a $\varphi$-module over $\mathcal{O}_{\mathcal{E}}$ is an $\mathcal{O}_{\mathcal{E}}$-module $M$ together with a semi-linear map $\varphi_M: M\rightarrow M$. We say that $M$ is \'{e}tale over $\mathcal{O}_{\mathcal{E}}$ if $M$ is an $\mathcal{O}_{\mathcal{E}}$-module of finite type and $\Phi_M^{lin}: M_{\varphi}\rightarrow M$ is an isomorphism, where $M_\varphi$ is the base change of $M$ by $\mathcal{O}_{\mathcal{E}}$ via $\varphi$. } %%%Recall that an \'{e}tale $\varphi$-module over $\mathcal{O}_{\mathcal{E}}$ is a finite rank $\mathcal{O}_{\mathcal{E}}$-module with a bijective semi-linear operator $\varphi$. 
 An \'{e}tale $(\varphi,\Gamma)$-module over $\mathcal{O}_{\mathcal{E}}$ is an \'{e}tale $\varphi$-module over $\mathcal{O}_{\mathcal{E}}$ together with a continuous and semi-linear action of $\Gamma$ commuting with the action of $\varphi$. Then using the theory of the field of norms due to Fontaine and Wintenberger \cite{Win}, he deduced the mixed characteristic case ($(0,p)$ case) from the equal characteristic case. In this case, he decomposed the Galois group $G_K$ along the %%a totally ramified $\mathbb{Z}_p$-
{\colo extension $K_{cyc}$ of $K$, and }%%. Then he 
showed that the category of $\mathbb{Z}_p$-adic representations of $G_K$ is equivalent to the category of \'{e}tale $(\varphi,\Gamma)$-modules over $\mathcal{O}_{\mathcal{E}}$. 
 
  This equivalence is a deep result that allows the computation of the Galois cohomology. In \cite{LH1}, Herr devised a technique to calculate the Galois cohomology by introducing a complex, namely, the \emph{Herr complex}. The Herr complex is defined on the category of \'{e}tale $(\varphi,\Gamma)$-modules and the cohomology groups of this complex turn out to match with the Galois cohomology groups on the category of $\mathbb{Z}_p$-adic representations of $G_K$. The results of Fontaine, along with this complex, play a crucial role in all the works pertaining to the computation of the Galois cohomology. In \cite{Flo}, {\C Tavares Ribeiro} %%%Floric
  further extended the Herr complex to the False-Tate type curve extensions to include certain non-abelian extensions over cyclotomic $\mathbb{Z}_p$-extension.
  
 {\C In the case when $K=\mathbb{Q}_p$, note} %%Note 
 that the extension $K_{cyc}$ is obtained by adjoining the $p^n$-th roots of unity to $K$, and they are the $p^n$-torsion points of the multiplicative Lubin-Tate formal group $\mathbb{G}_m$ on $\mathbb{Q}_p$ with respect to the uniformizer $p$. Thus  %%%%%Since the $p^n$th roots of unity are the $p^n$-torsion points of the multiplicative Lubin-Tate formal group $\mathbb{G}_m$ on $\mathbb{Q}_p$, with respect to the uniformizer $p$, therefore the cyclotomic $\mathbb{Z}_p$-extension 
{\C the extension $K_{cyc}$} is the same as the extension associated with the multiplicative Lubin-Tate formal group. {\C In this context, the computation of Galois cohomology using Fontaine's theory has found deep applications in Iwasawa theory over cyclotomic extensions as shown in \cite{CC}, as well as in the study of a $p$-adic Local Langlands correspondence over $\mathbb{Q}_p$.} It is natural to try to carry out this theory for \emph{arbitrary} Lubin-Tate formal group over $K$. In this direction, there has been a lot of activity in recent years to develop the Fontaine theory for Lubin-Tate formal groups \cite{KR}, \cite{Berger2013}, \cite{Four-Xie}, \cite{Berger2016}, \cite{SV}, \cite{Ber-Four}, and \cite{Ber-Sch-Xie} defined over %%%where the base field is 
a finite extension $L$ of $\mathbb{Q}_p$ with ring of integers $\mathcal{O}_L$ and uniformizer $\pi_L$. 
 
 In \cite{KR}, Kisin {\C and} Ren classified the local Galois representations using the extensions arising from division (torsion) points of {\C a} %%the 
 Lubin-Tate formal group defined over {\col $L$ for a finite extension $L/\mathbb{Q}_p$ contained in }$K$. More precisely, consider a Lubin-Tate formal group $\mathcal{G}$ over $L$, %% a finite extension $L/\mathbb{Q}_p$ contained in $K$, 
 and for $n\geq 1$, let {\C $K_n\subset \bar{K}$} be the subfield generated by the $\pi_L^n$-torsion points of $\mathcal{G}$, where $\pi_L$ is a uniformizer of $\mathcal{O}_L$. %%These fields $K_n$'s are usually referred to as \emph{Lubin-Tate extensions} of $K$. 
 Define $K_{\infty}:= \cup_{n\geq1} K_n$. {\C The field extension $K_\infty$ is usually referred to as a \emph{Lubin-Tate extension} of $K$. Let} $\Gamma_{LT}:= \Gal(K_\infty/K)$. Then they obtained a classification of $G_K$-representations on finite $\mathcal{O}_L$-modules via \'{e}tale $(\varphi_q,\Gamma_{LT})$-modules, where \'{e}tale $(\varphi_q,\Gamma_{LT})$-modules are analogues of \'{e}tale $(\varphi,\Gamma)$-modules (\cite[Theorem 1.6]{KR}). More details are given in section \ref{2.2}.  
   
 This paper depends heavily on the classification of $G_K$-representations given by Kisin and Ren. {\C We have also been influenced by the work \cite{SV} for applications to Iwasawa theory.} We show that the theorem of Kisin and Ren allows us to compute the Galois cohomology of representations defined over $\mathcal{O}_L$ (Theorem \ref{lattices}). For this first, we observe that the theorem of Kisin and Ren \cite[Theorem 1.6]{KR} holds for ${\bf Rep}_{\mathcal{O}_L-tor}^{dis}(G_K)$ the category of discrete $\pi_L$-primary {\C $\mathcal{O}_L$-modules} %%abelian groups 
 with a continuous and linear action of $G_K$. It is crucial to work with this category as this category has enough injectives, and this category is equivalent to the category of injective limits of $\pi_L$-power torsion objects in the category of \'{e}tale $(\varphi_q,\Gamma_{LT})$-modules over $\mathcal{O}_{\mathcal{E}}$ (Corollary \ref{KR discrete}). {\col We emphasize that we can make the constructions only through the category of $\pi_L$-primary discrete representations.} Next, we generalize the Herr complex to the Lubin-Tate extensions, and we call it the \emph{Lubin-Tate Herr complex} (see Definition \ref{LTHC}). %%%{\colo In the case when $K$ is a finite extension of $L$ contained in $K_\infty$ such that $\Gal(K_\infty/K)$ is $p$-torsion free,} 
{\colo In the case when $K$ is a fixed finite extension of $L$ contained in $K_\infty$ such that $\Gal(K_\infty/K)$ is isomorphic to $\mu_{q-1}\oplus \mathbb{Z}_p^d$, where $\mu_{q-1}$ is the group of $(q-1)$th roots of unity,} we have the following result. {\colo For more details see section \ref{sec3}.}
	\begin{thmalph}[=Theorem \ref{G_K-cohomo}]\label{M1}
For a discrete $\pi_L$-primary {\C $\mathcal{O}_L$-module} %%abelian group 
		$V$ with a continuous and linear action of $G_K$, we have a natural isomorphism 
		\begin{equation*}
		H^i(G_K,V)\cong \mathcal{H}^i(\Phi\Gamma_{LT}^{\bullet}(\mathbb{D}_{LT}(V)))\quad \text{for} \ i\geq 0.
		\end{equation*}
%%%{where $L\subseteq K$ is such that $\Gamma_{LT}$ has no $p$-torsion.}
The cohomology groups on the right side are computed using the Lubin-Tate Herr complex defined for $(\varphi_q,\Gamma_{LT})$-module corresponding to $V$, while the left hand side denotes the usual Galois cohomology groups of the representation $V$.
	\end{thmalph}
Moreover, we show that both the cohomology functors commute with the inverse limits and deduce the above theorem for the case when $V$ is a representation defined over $\mathcal{O}_L$ (Theorem \ref{lattices}). We further extend the equivalence of categories of Kisin and Ren to include certain non-abelian extensions over the Lubin-Tate extension (Theorem \ref{False-Tate equivalence}) and show that the construction of the Lubin-Tate Herr complex for $(\varphi_q,\Gamma_{LT})$-modules can be generalized to $(\varphi_q,\Gamma_{LT,FT})$-modules over non-abelian extensions, and we call it the \emph{False-Tate type Herr complex} (Definition \ref{FTHC}). In this case, we establish the following theorem.
	\begin{thmalph}[= Theorem \ref{Main4}]\label{M2}
		For any $V \in {\bf Rep}_{\mathcal{O}_L-tor}^{dis}(G_K)$, we have a natural isomorphism
		\begin{equation*}
		H^i(G_K,V)\cong \mathcal{H}^i(\Phi\Gamma_{LT,FT}^{\bullet}(\mathbb{D}_{LT,FT}(V)))\quad \text{for}\ i\geq 0.
		\end{equation*}
		In other words, the False-Tate type Herr complex $\Phi\Gamma_{LT,FT}^{\bullet}(\mathbb{D}_{LT,FT}(V))$ computes the Galois cohomology of $G_K$ with coefficients in $V$. 
	\end{thmalph}
Next, {\C we consider the operator $\psi_q$ acting on \'{e}tale $(\varphi_q,\Gamma_{LT})$-modules, as defined in \cite{SV}. Then for any $V \in {\bf Rep}_{\mathcal{O}_L-tor}^{dis}(G_K)$, we have a well-defined homomorphism 
	\begin{equation*}
	\mathcal{H}^i(\Phi\Gamma_{LT}^{\bullet}(\mathbb{D}_{LT}(V)))\rightarrow \mathcal{H}^i(\Psi\Gamma_{LT}^{\bullet}(\mathbb{D}_{LT}(V)))\quad \text{for} \ i\geq0.
	\end{equation*}
In the cyclotomic case, it is a deep theorem of Colmez and Herr that the analogous map is an isomorphism. At present, in the Lubin-Tate case, we can only say that $\mathcal{H}^0(\Phi\Gamma_{LT}^{\bullet}(\mathbb{D}_{LT}(V)))\rightarrow \mathcal{H}^0(\Psi\Gamma_{LT}^{\bullet}(\mathbb{D}_{LT}(V)))$ is injective. Moreover, we prove similar result in the case of False-Tate type extensions (Theorem \ref{Theorem False Tate})}
 Next, we recall the computation of the Iwasawa cohomology in terms of the complex associated with $\psi_q$ (Theorem \ref{Iwasawa cohomology}), {\C which is a result of \cite[Theorem 5.13]{SV}}. 
\iffalse{
Next, we describe the Iwasawa cohomology in terms of the complex associated with $\psi_q$. We prove the following theorem.
	\begin{thmalph}[=Theorem \ref{Iwasawa cohomology}]\label{M7}
For any $V \in {\bf Rep}_{\mathcal{O}_K-tor}^{dis}(G_K)$, the complex 
	\begin{equation*}
	\underline{\Psi}^{\bullet}(\mathbb{D}_{LT}(V(\chi_{cyc}^{-1}\chi_{LT}))): 0\rightarrow \mathbb{D}_{LT}(V(\chi_{cyc}^{-1}\chi_{LT}))\xrightarrow{\psi-id}\mathbb{D}_{LT}(V(\chi_{cyc}^{-1}\chi_{LT}))\rightarrow 0,
	\end{equation*} 
	where $\psi= \psi_{\mathbb{D}_{LT}(V(\chi_{cyc}^{-1}\chi_{LT}))}$, computes $H^i_{Iw}({K_{\infty}/K},V)_{i\geq 1}$ the Iwasawa cohomology groups.
	\end{thmalph}
}\fi

{\C A Lubin-Tate analogue of Perrin-Riou's exponential map has been given in \cite[section 3.5]{Ber-Four} in the context of analytic representations. Hoping for an exponential map for families of Galois representations that give rise to such analytic representations, in the second part, we begin by generalizing our results to the case of coefficient rings.} %%%In the second part, we use our techniques to deal with the case of coefficient rings. 
Recall that a coefficient ring is a complete local Noetherian ring with finite residue field. 

In \cite{Dee}, Dee generalized Fontaine theory to the case of a general complete Noetherian local ring $R$, whose residue field is a finite extension of $\mathbb{F}_p$.  He extended Fontaine's \cite{Fon} results to the category of $R$-modules of finite type with a continuous $R$-linear action of $G_K$. He constructed a category of \'{e}tale $\varphi$-modules (resp., \'{e}tale $(\varphi,\Gamma)$-modules) over $K$ parameterized by $R$ and proved that this category is equivalent to the category of $R$-linear representations of $G_K$ in the equal characteristic case (resp., mixed characteristic case) (\cite[Theorem 2.1.27 and Theorem 2.3.1]{Dee}). The category of \'{e}tale $\varphi$-modules (resp., \'{e}tale $(\varphi,\Gamma)$-modules) is defined to be a module of finite type over the completed tensor product $\mathcal{O}_{\mathcal{E}}\hat{\otimes}_{\mathbb{Z}_p}R$ with  an action of $\varphi$ (resp., $\varphi$ and $\Gamma$) as in the case of Fontaine. The core point of the proof is Lemma 2.1.5. and Lemma 2.1.6. in \cite{Dee}. In the proof of the equivalence of categories stated above, he used the results of Fontaine \cite{Fon} crucially for the case when the representation $V$ has finite length. Then the general case was deduced by taking the inverse limits. 

We also extend a result of Kisin and Ren (\cite[Theorem $1.6$]{KR}) to give a classification of the category of $R$-representations of $G_K$. We consider a category of \'{e}tale $(\varphi_q,\Gamma_{LT})$-modules over the completed tensor product $\mathcal{O}_R:=\mathcal{O}_{\mathcal{E}}\hat{\otimes}_{\mathcal{O}_L}R$, where the ring $\mathcal{O}_{\mathcal{E}}$ is constructed using the periods of Tate-module of $\mathcal{G}$. Then we prove that this category is equivalent to the category of $R$-representations of $G_K$. In the equal characteristic case, we show the following result.

	\begin{thmalph}[=Theorem \ref{Main6.1}]\label{M3}
		The functor $V\mapsto \mathbb{D}_R(V)$ is an exact equivalence of categories between ${\bf Rep}_R(G_K)$ the category of $R$-representations of $G_K$ and ${\bf Mod}_{/\mathcal{O}_R}^{\varphi_q,\acute{e}t}$ the category of \'{e}tale $\varphi_q$-modules over $\mathcal{O}_R$ with quasi-inverse functor $\mathbb{V}_R$.
	\end{thmalph}
	The construction of these functors is explained in section \ref{sub6.1}. In the case of mixed characteristic, we have following theorem which gives a classification of $R$-representations of the local Galois group in terms of \'{e}tale $(\varphi_q,\Gamma_{LT})$-modules over $\mathcal{O}_R$. 
	\begin{thmalph}[=Theorem \ref{Main6.2}]\label{M4}
		The functor $\mathbb{D}_R$ is an equivalence of categories between ${\bf Rep}_R(G_K)$ the category of $R$-linear representations of $G_K$ and ${\bf Mod}_{/\mathcal{O}_R}^{\varphi_q,\Gamma_{LT},\acute{e}t}$ the category of \'{e}tale $(\varphi_q,\Gamma_{LT})$-modules over $\mathcal{O}_R$. The functor $\mathbb{V}_R$ is a quasi-inverse of the functor $\mathbb{D}_R$. 	
	\end{thmalph}
	We also have a generalization of Theorem \ref{M1} and Theorem \ref{Iwasawa cohomology} to the case of the coefficient ring. The generalization of Theorem \ref{Iwasawa cohomology} to the case of coefficient rings allows us to generalize the dual exponential map
	\begin{equation*}
	\text{Exp}^*:H^1_{Iw}({K_{\infty}/K},\mathcal{O}_L(\chi_{cyc}\chi_{LT}^{-1}))\xrightarrow{\sim} \mathbb{D}_{LT}(\mathcal{O}_L)^{\psi_{\mathbb{D}_{LT}(\mathcal{O}_L)}=id}
	\end{equation*}
	 defined in \cite{SV} over coefficient rings (see Corollary \ref{dual exp.}). It is possible that this leads to the construction of Coates-Wiles homomorphisms for the Galois representations defined over $R$. %%%%%After getting the results of this paper, we found that there are some similar results in \cite{PZ}, though in a different perspective. Our treatment of Lubin-Tate extensions is new, and also include certain non-abelian extensions.
	 
{\C We would like to mention the work of Kupferer and Venjakob \cite{KV} whose results can be compared with some of our results here. We hope that their results also carry over to the case of coefficient rings as we have done here.}

	\subsection*{Organization of the paper} In section \ref{sec2}, we recall some necessary background that will be used in subsequent sections. In section \ref{sec3}, we define the Lubin-Tate Herr complex as a generalization of Herr complex over Lubin-Tate extensions and compute the Galois cohomology groups of representations defined over $\mathcal{O}_L$. In the next section, we extend the Lubin-Tate Herr complex to include certain non-abelian extensions and show results in the computation of Galois cohomology. In section \ref{section psi}, we define an operator $\psi_q$ acting on the category of \'{e}tale $(\varphi_q,\Gamma_{LT})$-modules and prove some results, which give a relationship between the cohomology groups of the Lubin-Tate Herr complex for $\varphi_q$ and $\psi_q$. We also present our results, which give a relationship between the False-Tate type Herr complex for $\varphi_q$ and $\psi_q$. In section \ref{Iwasawa}, we briefly recall the computation of the Iwasawa cohomology due to Schneider and Venjakob for Lubin-Tate extensions in terms of the complex associated with $\psi_q$. Then in section \ref{sec6}, we record some significant results on coefficient rings, and 	we generalize a theorem of Kisin and Ren over coefficient rings, which in turn allows us to extend our results over coefficient rings, and these results appear in section \ref{sec7}.
	%%%% This paper consists of nine sections. Section \ref{sec2} introduces some standard results of Lubin-Tate extensions. In section \ref{sec3}, we define the Lubin-Tate Herr complex as a generalization of the Herr complex for Lubin-Tate extensions and compute Galois cohomology groups in Theorem \ref{G_K-cohomo}. In the next section, we extend the complex defined in section \ref{sec3} to include certain non-abelian extensions and show results in the computation of Galois cohomology. In section \ref{section psi}, we define a linear operator $\psi_q$, and prove Theorem \ref{Main5}, which gives a relation between the Lubin-Tate Herr complex for $\varphi_q$ and $\psi_q$. We also prove Theorem \ref{Theorem False Tate}, which gives a relation between the False-Tate type Herr complex for $\varphi_q$ and $\psi_q$. In section \ref{Iwasawa}, we compute Iwasawa cohomology in terms of $\psi_q$ complex. Then in section \ref{sec5}, we record some significant results of coefficient rings. In section \ref{sec6}, we generalize Kisin and Ren's Theorem \cite[Theorem 1.6]{KR} for coefficient rings, which in turn allows us to extend Theorem \ref{lattices}, Theorem \ref{Main5} and Theorem \ref{Iwasawa cohomology} over the coefficient rings to Theorem \ref{Main7}, Theorem \ref{Main7.2} and Theorem \ref{Main7.3}, respectively, in section \ref{sec7}.  
\subsection*{Acknowledgements}
We would like to thank Laurent Berger for going through an earlier draft of the article and suggesting valuable comments regarding the article. We would like to sincerely thank all the anonymous referees for all the comments and suggestions and also pointing out several imprecisions.
\section{Lubin-Tate Extensions and Galois Representations}\label{sec2}
%%%%This section is divided into two subsections. In section \ref{sec2.1}, we provide an introduction to the theory of Lubin-Tate formal groups. In section \ref{2.2}, we recall the construction of the equivalence of categories proved by Kisin and Ren. 
\subsection{Background on Lubin-Tate modules}\label{sec2.1}
In this section, we recall some basic results of Lubin-Tate modules. {\C We closely follow the exposition given in \cite{Neu}.} {\col Let $L$ be a local field of characteristic $0$ with the ring of integers $\mathcal{O}_L$, maximal ideal $\mathfrak{m}_L$, and residue field $k_L$ of characteristic $p>0$. Let $\bar{L}$ be a fixed algebraic closure of $L$ with the ring of integers $\mathcal{O}_{\bar{L}}$ and maximal ideal $\mathfrak{m}_{\bar{L}}$.} %%%%% Fix} %%We fix a local field $K$ of characteristic $0$ with the ring of integers $\mathcal{O}_K$, maximal ideal $\mathfrak{m}_K$, and residue field $k$ of characteristic $p>0$. \sout{Let $\pi$ be a prime element of $\mathcal{O}_K$, $\card(k)=q$ and $q=p^r$ for some fixed $r$.} Let $\bar{K}$ be a fixed algebraic closure of $K$ with the ring of integers $\mathcal{O}_{\bar{K}}$ and maximal ideal $\mathfrak{m}_{\bar{K}}$. {\col Let $L$ be a finite extension of $\mathbb{Q}_p$ contained in $K$. Let $\mathcal{O}_L$} %%%%For a local field $K$, let $\mathcal{O}_K$ be the ring of integers of $K$ with the maximal ideal $\mathfrak{m}_K$. Let $\pi$ be a prime element of $K$ and $k = \mathcal{O}_K/\mathfrak{m}_K$ be its residue field with characteristic $p$. Assume that $\#k=q$, where $q$ is a power of $p$. Let $\bar{K}$ be the algebraic closure of $K$ with the ring of integers $\mathcal{O}_{\bar{K}}$ and maximal ideal $\mathfrak{m}_{\bar{K}}$. \par 

{\C A one-dimensional \emph{formal group} $\mathcal{G}$ over $\mathcal{O}_L$ is a formal power series $\mathcal{G}(X,Y)\in \mathcal{O}_L[[X,Y]]$ in two variables with coefficients in $\mathcal{O}_L$ such that 
\begin{enumerate}
\item $\mathcal{G}(X,0) = X$ and $\mathcal{G}(0,Y) = Y$, i.e., $\mathcal{G}(X,Y) \equiv X+Y\,\text{mod deg}\,2$,
\item $\mathcal{G}(X,\mathcal{G}(Y,Z)) = \mathcal{G}(\mathcal{G}(X,Y),Z)$. 
\end{enumerate}
Moreover, if $\mathcal{G}$ satisfies
%%\begin{equation*}
$\mathcal{G}(X,Y)= \mathcal{G}(Y,X)$ 
%%\end{equation*}
then $\mathcal{G}$ is said to be a \emph{commutative formal group}. In this article, a formal group always means a one-dimensional commutative formal group.

A \emph{formal $\mathcal{O}_L$-module} is a formal group $\mathcal{G}$ over $\mathcal{O}_L$ together with a ring homomorphism
\begin{align*}
& \mathcal{O}_L\rightarrow \End_{\mathcal{O}_L}(\mathcal{G})\\ &a \mapsto [a]_{\mathcal{G}}(Z)
\end{align*}
such that $[a]_{\mathcal{G}}(Z) \equiv aZ\,\text{mod deg}\,2$.} 

{\col Let $q= |k_L|$.} A \emph{Lubin-Tate module} over $\mathcal{O}_L$, for a prime element $\pi_L$ of $\mathcal{O}_L$, is a formal $\mathcal{O}_L$-module $\mathcal{G}$ such that $[\pi_L]_{\mathcal{G}}(X)\equiv X^q\mod\pi_L$. Then the set $\mathfrak{m}_{\bar{L}}$ together with the operations
	\begin{equation*}	
	x\underset{\mathcal{G}}+y:=\mathcal{G}(x,y) \quad \text{and} \quad  a.x:=[a]_{\mathcal{G}}(x) \quad \text{for}\  x,y\in\mathfrak{m}_{\bar{L}} \ \text{and} \ a\in\mathcal{O}_L
	\end{equation*}
gives rise to an $\mathcal{O}_L$-module in the usual sense, which we denote by {\cv$(\mathfrak{m}_{\bar{L}})_{\mathcal{G}}$}. Now consider 
	\begin{equation*}
	\mathcal{G}(n):=\{\lambda\in{\cv (\mathfrak{m}_{\bar{L}})_{\mathcal{G}}}\vert\pi_L^n.\lambda=0\}
	=\{\lambda\in{\cv (\mathfrak{m}_{\bar{L}})_{\mathcal{G}}}\;\vert[\pi_L^n]_{\mathcal{G}}(\lambda)=0\}=\Ker([\pi_L^n]_{\mathcal{G}}) 
	\end{equation*}
	the group of $\pi_L^n$-division points. Then $\mathcal{G}(n)$ is a free $\mathcal{O}_L/\pi_L^n\mathcal{O}_L$-module of rank $1$ \cite[Chapter III, Proposition 7.2]{Neu}.  
	
Let $L_n:=L(\mathcal{G}(n))$. Since $\mathcal{G}(n)\subseteq \mathcal{G}(n+1)$, we have a chain of fields $L\subseteq L_1\subseteq L_2\subseteq\ldots\subseteq L_{\infty}= \bigcup_{n=1}^{\infty} L_n$. {\C The field {\cv $L_\infty=\cup_{n\geq1}L_n$} is referred to as {\col a} \emph{Lubin-Tate extension} of $L$.} %%These field extensions are called \emph{Lubin-Tate extensions}. 
The extension $L_n/L$ is totally ramified abelian extension of degree $q^{n-1}(q-1)$ with Galois group $\Gal(L_n/L)\cong \Aut_{\mathcal{O}_L}(\mathcal{G}(n))\cong \mathcal{O}_L^{\times}/\mathcal{O}^{\times(n)}_L$ \cite[Chapter III, Theorem 7.4]{Neu}. Moreover, this isomorphism fits into the following commutative diagram
	\begin{center}
		\begin{tikzcd}
		\Gal(L_{n+1}/L) \arrow{r}{\cong} \arrow[swap]{d}{restriction} & \mathcal{O}_L^{\times}/\mathcal{O}^{\times(n+1)}_L \arrow{d}{projection} \\
		\Gal(L_n/L) \arrow{r}[swap]{\cong}& \mathcal{O}_L^{\times}/\mathcal{O}^{\times(n)}_L.
		\end{tikzcd}
	\end{center}
Now by taking the projective limits, we obtain the isomorphism
	\begin{equation}\label{Lubin Tate character}
	\Gal(L_\infty/L)\cong \mathcal{O}_L^\times.
	\end{equation}
	\subsection{Kisin and Ren's equivalence}\label{2.2} 
%%Recall that 
Let $K$ be a local field of characteristic $0$, i.e., $K$ is a finite extension of $\mathbb{Q}_p$ and it is complete with respect to a discrete valuation with finite residue field $k$ of characteristic $p>0$. Let $\bar{K}$ be a fixed algebraic closure of $K$ with ring of integers $\mathcal{O}_{\bar{K}}$. We write $G_K:=\Gal(\bar{K}/K)$ for the absolute Galois group of $K$. %%%Assume that $p$ is an odd prime.

In this section, we recall the construction of equivalence of categories of Kisin and Ren \cite[Theorem $1.6$]{KR}. For this, let $W = W(k)$ be the ring of Witt vectors over $k$ and  $K_0 = W[\frac{1}{p}]$ be the field of fractions of $W$. Then $K_0$ is maximal unramified extension of $\mathbb{Q}_p$ contained in $K$. %%%%% Let $\mathcal{O}_K$ be the ring of integers of $K$ and $\pi$ be its uniformizer and $\#k= q = p^r$. Fix an algebraic closure $\bar{K}$ of $K$ with the ring of integers $\mathcal{O}_{\bar{K}}$ and set $G_K =\Gal(\bar{K}/K)$. For an $\mathcal{O}_{K_0}$-algebra $A$, we write $A_K = A\otimes_{\mathcal{O}_{K_0}} \mathcal{O}_K$.}
{\col Let $L$ be a finite extension of $\mathbb{Q}_p$ contained in $K$. Let $\mathcal{O}_L$ be the ring of integers of $L$ with the residue field %%a prime element $\pi_L$. Let 
	$k_L\subset k$. %% be the residue field of $\mathcal{O}_L$. 
	Let $\mathcal{O}_{L_0}= W(k_L)$, $L_0= \mathcal{O}_{L_0}[\frac{1}{p}]$, and $\card(k_L)=q=p^r$. For an $\mathcal{O}_{L_0}$-algebra $A$, we write $A_L = A\otimes_{\mathcal{O}_{L_0}}\mathcal{O}_L$.}
 
 Let $\mathcal{G}$ be the Lubin-Tate group over $L$ corresponding to a uniformizer $\pi_L$ of $\cO_L$. As in \cite{KR}, we fix a local co-ordinate $X$ on $\mathcal{G}$ such that the Hopf algebra $\mathcal{O}_{\mathcal{G}}$ may be identified with $\mathcal{O}_L[[X]]$. For any $a \in \mathcal{O}_L$, write $[a]_{\mathcal{G}}\in \mathcal{O}_L[[X]]= \mathcal{O}_{\mathcal{G}}$ the power series giving the endomorphism $a$ of $\mathcal{G}$. Let $K_\infty$ be a Lubin-Tate extension of $K$ {\cv obtained by adjoining all the $\pi$-power torsion points of $\mathcal{G}$ to $K$}. Let $H_K= \Gal(\bar{K}/K_{\infty})$ and $\Gamma_{LT}= G_K/H_K= \Gal(K_\infty/K)$. Let $\mathcal{TG}$ be the $p$-adic Tate-module of $\mathcal{G}$. Then $\mathcal{TG}$ is a free $\mathcal{O}_L$-module of rank $1$. The action of $G_K$ on $\mathcal{TG}$ factors through $\Gamma_{LT}$ and induces an isomorphism $\chi_{LT} : \Gamma_{LT} \rightarrow \mathcal{O}_L^\times$. 
 
 Let $\mathcal{R} = \varprojlim \mathcal{O}_{\bar{K}}/p\mathcal{O}_{\bar{K}}$, where the transition maps are given by the Frobenius $\varphi$. The ring $\mathcal{R}$ can also be identified with $\varprojlim \mathcal{O}_{\bar{K}}/\pi_L\mathcal{O}_{\bar{K}}$, and the transition maps being given by the $q$-Frobenius $\varphi_q = \varphi^r$. The ring $\mathcal{R}$ is a complete valuation ring, and it is perfect of characteristic $p$. The fraction field $\Fr(\mathcal{R})$ of $\mathcal{R}$ is a complete, algebraically closed non-archimedean {(\cv with respect to the valuation induced from $\mathcal{R}$)} perfect field of characteristic $p$. {\C For these properties of the ring $\mathcal{R}$, details are given in \cite[Chapter 4]{Fon}}. Then we have  a map $\iota : \mathcal{TG}\rightarrow \mathcal{R}$, which is induced by the evaluation of $X$ at $\pi_L$-torsion points. Let $v = (v_n)_{n \geq 0}\in \mathcal{TG}$ with $v_n \in \mathcal{G}(n)$ and $\pi_L.v_{n+1} = v_n,$ then $\iota(v) = (v^*_n(X)+\pi_L\mathcal{O}_{\bar{K}})_{n\geq 0}$, {\C where $v_n^*(X)= [\pi_L]_{\mathcal{G}}(X)|_{X=v_n}$. Recall that $W(\mathcal{R})_L= W(\mathcal{R})\otimes_{\mathcal{O}_{L_0}} \mathcal{O}_L$. Then } %%Moreover, 
 we have the following lemma, which follows from \cite[Lemma 9.3]{Co1}. More details are given in \cite[\S2.1]{Sch}.
\begin{lemma}\!\textup{\cite[Lemma 1.2]{KR}}\label{embedding}
There is a unique map $\{\} : \mathcal{R}\rightarrow W(\mathcal{R})_L$ such that $\{x\}$ is a lifting of $x$ and $\varphi_q(\{x\}) = [\pi_L]_{\mathcal{G}}(x).$ Moreover, $\{\}$ respects the action of $G_K$. In particular, if $v \in \mathcal{TG}$ is an $\mathcal{O}_L$-generator, there is an embedding $W_L[[X]]\hookrightarrow W(\mathcal{R})_L$ sending $X$ to $\{\iota(v)\}$ which identifies $W_L[[X]]$ with a $G_K$-stable, $\varphi_q$-stable subring of $W(\mathcal{R})_L$ such that $\{\iota(\mathcal{TG})\}$ lies in the image of $W_L[[X]]$. 
\end{lemma}
The $G_K$-action on $W_L[[X]]$ factors through $\Gamma_{LT}$, and we have $\varphi_q(X) = [\pi_L]_{\mathcal{G}}(X)$ and $\sigma_a(X)= [a]_{\mathcal{G}}(X)$, where $\sigma_a = \chi_{LT}^{-1}(a)$ for any $a \in \mathcal{O}_L^\times$. We fix an $\mathcal{O}_L$-generator $v \in \mathcal{TG}$ and identify $W_L[[X]]$ with a subring of $W(\mathcal{R})_L$ by sending $X$ to $\{\iota(v)\}$ by using Lemma \ref{embedding}.

Let $\mathcal{O}_{\mathcal{E}}$ be the $p$-adic completion of $W_L[[X]][\frac{1}{X}]$. Then $\mathcal{O}_{\mathcal{E}}$ is a complete discrete valuation ring with uniformizer $\pi_L$ and the residue field $k((X))$. Since $W(\mathcal{R})$ is $p$-adically complete, we may view
	\begin{equation*}
{\C	\mathcal{O}_{\mathcal{E}}%%%\subset W(\mathcal{R})_K 
	\subset  W(\Fr(\mathcal{R}))_L. }
	\end{equation*}
Let $\mathcal{O}_{\mathcal{E}^{ur}} \subset W(\Fr(\mathcal{R}))_L$ denote the maximal integral unramified extension of $\mathcal{O}_{\mathcal{E}}$ and $\mathcal{O}_{\widehat{\mathcal{E}^{ur}}}$ the $p$-adic completion of  $\mathcal{O}_{\mathcal{E}^{ur}}$, which is again a subring of $W(\Fr(\mathcal{R}))_L$. Let $\mathcal{E}, \mathcal{E}^{ur}$ and $\widehat{\mathcal{E}^{ur}}$ denote the field of fractions of $\mathcal{O}_{\mathcal{E}}, \mathcal{O}_{\mathcal{E}^{ur}}$ and $\mathcal{O}_{\widehat{\mathcal{E}^{ur}}}$, respectively. These rings are all stable under the action of $\varphi_q$ and $G_K$. Moreover, the $G_K$-action {\C on $\mathcal{O}_{\mathcal{E}}$} factors through $\Gamma_{LT}$.
	\begin{lemma}\!\textup{\cite[Lemma 1.4]{KR}}\label{Galois group isomorphism}
The residue field of $\mathcal{O}_{\widehat{\mathcal{E}^{ur}}}$ is a separable closure of $k((X))$, and there is a natural isomorphism
		\begin{equation*}
		\Gal(\mathcal{E}^{ur}/\mathcal{E}) \xrightarrow{\sim}\Gal(\bar{K}/K_\infty).
		\end{equation*}
	\end{lemma}
Let $E:= k((X))$, which is the residue field of $\mathcal{E}$; then it follows from Lemma \ref{Galois group isomorphism} that $E^{sep}$ is the residue field of $\widehat{\mathcal{E}^{ur}}$. The following lemma is an easy consequence of the definition of $\mathcal{O}_{\widehat{\mathcal{E}^{ur}}}$.
	\begin{lemma}\label{Exact}
		$(\mathcal{O}_{\widehat{\mathcal{E}^{ur}}})^{\varphi_q=id} = \mathcal{O}_L$.
	\end{lemma}  
	\begin{proof}
		Consider the exact sequence
		\begin{equation*}
		0\rightarrow k_L \rightarrow E^{sep}\xrightarrow[x\mapsto x^q-x]{\varphi_q-id} E^{sep}\rightarrow 0.
		\end{equation*}
		By d\'{e}vissage, we deduce the exact sequence
		\begin{equation*}
		0\rightarrow \mathcal{O}_L/\pi_L^n\mathcal{O}_L \rightarrow \mathcal{O}_{\widehat{\mathcal{E}^{ur}}}/\pi_L^n \mathcal{O}_{\widehat{\mathcal{E}^{ur}}}\xrightarrow{\varphi_q-id} \mathcal{O}_{\widehat{\mathcal{E}^{ur}}}/\pi_L^n \mathcal{O}_{\widehat{\mathcal{E}^{ur}}}\rightarrow0, \:\forall\: n\geq 1.
		\end{equation*}
		Here the projective system $\{\mathcal{O}_L/\pi_L^n\mathcal{O}_L\}_{n\geq 1}$ has surjective transition maps, therefore passing to the projective limit is exact and gives us an exact sequence
		\begin{equation*}
		0\rightarrow \mathcal{O}_L\rightarrow \mathcal{O}_{\widehat{\mathcal{E}^{ur}}}\xrightarrow{\varphi_q-id}\mathcal{O}_{\widehat{\mathcal{E}^{ur}}}\rightarrow 0.
		\end{equation*}
		Hence, $(\mathcal{O}_{\widehat{\mathcal{E}^{ur}}})^{\varphi_q=id} = \mathcal{O}_L$.
	\end{proof} 
	
The subring $\mathcal{O}_{\mathcal{E}}\subset W(\Fr(\mathcal{R}))_L$, which is constructed using the periods of $\mathcal{TG}$, is naturally a {\C $\pi_L$-}Cohen ring for $X_K(K)$, where $X_K(K)$ is a field of characteristic $p$ constructed using the field of norms. More details can be found in \cite[\S1]{KR}. The Galois group  $G_E = \Gal(E^{sep}/E)$ can be identified with $\Gal(X_K(\bar{K})/X_K(K))$. Then by Lemma \ref{Galois group isomorphism}, we have 
	\begin{equation*}
	H_K\xrightarrow{\sim}G_E.
	\end{equation*} 
	The $G_K$-action on $\mathcal{R}$ induces a $G_K$-action on $W(\Fr(\mathcal{R}))_L$, and the rings $\mathcal{O}_{\mathcal{E}}, \mathcal{O}_{\mathcal{E}^{ur}}$ and $\mathcal{O}_{\widehat{\mathcal{E}^{ur}}}$ are stable under the action of $G_K$. On the other hand, $G_E$ acts on $\mathcal{O}_{\widehat{\mathcal{E}^{ur}}}$ by continuity and functoriality, and these actions are compatible with the identification of Galois groups $H_K\xrightarrow{\sim}G_E$. \par Let $V$ be an $\mathcal{O}_L$-module of finite rank with a continuous and linear action of $G_K$. Consider the $\varphi_q$-module: 
	\begin{equation*}
	\mathbb{D}_{LT}(V): =( \mathcal{O}_{\widehat{\mathcal{E}^{ur}}}\otimes_{\mathcal{O}_L} V)^{H_K} = (\mathcal{O}_{\widehat{\mathcal{E}^{ur}}}\otimes_{\mathcal{O}_L}V)^{G_E}.
	\end{equation*}
{\C Note that the second equality holds as the Galois group $H_K$ is identified with the Galois group $G_E$.} The action of $G_K$ on $\mathcal{O}_{\widehat{\mathcal{E}^{ur}}}\otimes_{\mathcal{O}_L}V$ induces a semi-linear action of $G_K/H_K = \Gamma_{LT} = \Gal(K_\infty/K)$ on $\mathbb{D}_{LT}(V)$. {\C A $\varphi_q$-module over $\mathcal{O}_{\mathcal{E}}$ is an $\mathcal{O}_{\mathcal{E}}$-module $M$ together with a map $\varphi_M: M\rightarrow M$, which is semi-linear with respect to $q$-Frobenius. We say that $M$ is \'{e}tale over $\mathcal{O}_{\mathcal{E}}$ if $M$ is an $\mathcal{O}_{\mathcal{E}}$-module of finite type and $\Phi_M^{lin}: M_{\varphi_q}\rightarrow M$ is an isomorphism, where $M_{\varphi_q}$ is the base change of $M$ by $\mathcal{O}_{\mathcal{E}}$ via $\varphi_q$.} A $(\varphi_q,\Gamma_{LT})$-module $M$ over $\mathcal{O}_{\mathcal{E}}$ is $\varphi_q$-module over $\mathcal{O}_{\mathcal{E}}$ together with a semi-linear action of $\Gamma_{LT}$ and this action commutes with the endomorphism $\varphi_M$ of $M$. We say that $M$ %%%% a $(\varphi_q,\Gamma_{LT})$-module 
	is \'{e}tale if it is \'{e}tale as a $\varphi_q$-module. 
	
	We write ${\bf Mod}_{/\mathcal{O}_{\mathcal{E}}}^{\varphi_q,\Gamma_{LT},\acute{e}t}$ (resp., ${\bf Mod}_{/\mathcal{O}_{\mathcal{E}}}^{\varphi_q,\Gamma_{LT},\acute{e}t,tor}$) for the category of finite free (resp., finite torsion) \'{e}tale $(\varphi_q,\Gamma_{LT})$-modules over $\mathcal{O}_{\mathcal{E}}$ and ${\bf Rep}_{\mathcal{O}_L}(G_K)$ (resp., ${\bf Rep}_{\mathcal{O}_L-tor}(G_K)$) for the category of finite free (resp., finite torsion) $\mathcal{O}_L$-modules with a continuous linear action of $G_K$. Then $\mathbb{D}_{LT}$ is a functor from ${\bf Rep}_{\mathcal{O}_L}(G_K)$ (resp., ${\bf Rep}_{\mathcal{O}_L-tor}(G_K)$) to ${\bf Mod}_{/\mathcal{O}_{\mathcal{E}}}^{\varphi_q,\Gamma_{LT},\acute{e}t}$ (resp., ${\bf Mod}_{/\mathcal{O}_{\mathcal{E}}}^{\varphi_q,\Gamma_{LT},\acute{e}t,tor}$).

Let $M$ be an \'{e}tale $(\varphi_q,\Gamma_{LT})$-module over $\mathcal{O}_{\mathcal{E}}$. Then consider the $G_K$-representation:
	\begin{equation*}
	\mathbb{V}_{LT}(M):= (\mathcal{O}_{\widehat{\mathcal{E}^{ur}}}\otimes_{\mathcal{O}_{\mathcal{E}}}M)^{\varphi_q\otimes \varphi_M=id}.
	\end{equation*}
	Here $G_K$ acts on $\mathcal{O}_{\widehat{\mathcal{E}^{ur}}}$ as before and acts via $\Gamma_{LT}$ on $M$. The diagonal action of $G_K$ on $\mathcal{O}_{\widehat{\mathcal{E}^{ur}}}\otimes_{\mathcal{O}_{\mathcal{E}}}M$ is $\varphi_q\otimes\varphi_M$-equivariant, which induces a $G_K$-action on $\mathbb{V}_{LT}(M)$. Then using {\C a similar proof as in }%%%the similar proof as that in 
	\cite[A1, Proposition 1.2.4 and 1.2.6]{Fon}, we have that the functors $\mathbb{D}_{LT}$ and $\mathbb{V}_{LT}$ are exact functors and the  natural maps
	\begin{align*}
	&\mathcal{O}_{\widehat{\mathcal{E}^{ur}}}\otimes_{\mathcal{O}_{\mathcal{E}}}\mathbb{D}_{LT}(V)\rightarrow \mathcal{O}_{\widehat{\mathcal{E}^{ur}}}\otimes_{\mathcal{O}_L}V,\\&
	\mathcal{O}_{\widehat{\mathcal{E}^{ur}}}\otimes_{\mathcal{O}_L}\mathbb{V}_{LT}(M) \rightarrow \mathcal{O}_{\widehat{\mathcal{E}^{ur}}}\otimes_{\mathcal{O}_{\mathcal{E}}}M
	\end{align*}
	are isomorphisms. In particular, we have the following result, which is established in \cite[Theorem 1.6]{KR}.
	\begin{theorem}\!\textup{\cite[Theorem 1.6]{KR}}\label{Kisin Ren}
The functors
\begin{equation*}
V\mapsto \mathbb{D}_{LT}(V)= (\mathcal{O}_{\widehat{\mathcal{E}^{ur}}}\otimes_{\mathcal{O}_L}V)^{H_K} \qquad \text{and} \qquad M\mapsto \mathbb{V}_{LT}(M)= (\mathcal{O}_{\widehat{\mathcal{E}^{ur}}}\otimes_{\mathcal{O}_{\mathcal{E}}}M)^{\varphi_q\otimes\varphi_M=id}
		\end{equation*}
are exact {\C and quasi-inverse and induce an} %%%quasi-inverse 
equivalence of categories between %% the category 
${\bf Rep}_{\mathcal{O}_L}(G_K)\ (\text{resp.,}\, {\bf Rep}_{\mathcal{O}_L-tor}(G_K) )$ and ${\bf Mod}^{\varphi_q,\Gamma_{LT},\acute{e}t}_{/\mathcal{O}_{\mathcal{E}}}$  $(\text{resp.,}\,{\bf Mod}^{\varphi_q,\Gamma_{LT},\acute{e}t,tor}_{/\mathcal{O}_{\mathcal{E}}})$.
	\end{theorem}
\iffalse{
	\begin{remark}\label{kr-F}
\sout{Instead of $K$, if we choose any finite extension, say $F$, of $K$, then we have above equivalence of categories for $\mathcal{O}_F$-modules.}
	\end{remark}
}\fi
	\section{Galois Cohomology over the Lubin-Tate Extensions} \label{sec3}
 The category of finitely generated $\mathcal{O}_L$-modules with a continuous and linear action of $G_K$ does not have injectives, so the category ${\bf Rep}_{\mathcal{O}_L}(G_K)$ does not have enough injectives. Therefore we extend the functor $\mathbb{D}_{LT}$ to a category that has enough injectives as we are going to use injective objects to compute cohomology groups.

Let ${\bf Rep}_{\mathcal{O}_L-tor}^{dis}(G_K)$ be the category of discrete $\pi_L$-primary abelian groups with a continuous action of $G_K$. Then any object in this category is the filtered direct limit of $\pi_L$-power torsion objects in ${\bf Rep}_{\mathcal{O}_L-tor}(G_K)$. Note that the category ${\bf Rep}_{\mathcal{O}_L-tor}^{dis}(G_K)$ has enough injectives. First, we extend the functor $\mathbb{D}_{LT}$ to the category ${\bf Rep}_{\mathcal{O}_L-tor}^{dis}(G_K)$. For any $V \in {\bf Rep}_{\mathcal{O}_L-tor}^{dis}(G_K)$, define
	\begin{equation*}
	\mathbb{D}_{LT}(V):= (\mathcal{O}_{\widehat{\mathcal{E}^{ur}}}\otimes_{\mathcal{O}_L}V)^{H_K}.
	\end{equation*}
Since $V$ is the filtered direct limit of $\pi_L$-power torsion objects in ${\bf Rep}_{\mathcal{O}_L-tor}(G_K)$, and both the tensor product and taking $H_K$-invariant commute with the filtered direct limits, so the functor $\mathbb{D}_{LT}$ commutes with the filtered direct limits. Therefore $\mathbb{D}_{LT}$ is an exact functor into the category $\varinjlim {\bf Mod}^{\varphi_q,\Gamma_{LT},\acute{e}t,tor}_{/\mathcal{O}_{\mathcal{E}}}$ of injective limits of $\pi_L$-power torsion objects in ${\bf Mod}^{\varphi_q,\Gamma_{LT},\acute{e}t,tor}_{/\mathcal{O}_{\mathcal{E}}}$. Now for any object $M \in \varinjlim {\bf Mod}^{\varphi_q,\Gamma_{LT},\acute{e}t,tor}_{/\mathcal{O}_{\mathcal{E}}}$, define
	\begin{equation*}
	\mathbb{V}_{LT}(M):= (\mathcal{O}_{\widehat{\mathcal{E}^{ur}}}\otimes_{\mathcal{O}_\mathcal{E}}M)^{\varphi_q\otimes\varphi_M=id}.
	\end{equation*}
	The functor $\mathbb{V}_{LT}$ also commutes with the direct limits. Then we have the following proposition, which shows that the equivalence of Theorem \ref{Kisin Ren} extends to the category of discrete $\pi_L$-primary representations of $G_K$, and this is an important step towards our main theorem.
	\begin{proposition} \label{KR discrete}
		The functor $\mathbb{D}_{LT}$ is a quasi-inverse equivalence of categories between the category ${\bf Rep}_{\mathcal{O}_L-tor}^{dis}(G_K)$ and  $\varinjlim {\bf Mod}^{\varphi_q,\Gamma_{LT},\acute{e}t,tor}_{/\mathcal{O}_{\mathcal{E}}}$ with quasi-inverse $\mathbb{V}_{LT}$.
	\end{proposition}
\begin{proof}
Since the {\cv functors} $\mathbb{D}_{LT}$ and $\mathbb{V}_{LT}$ commute with the direct limits, the proposition follows from Theorem \ref{Kisin Ren} by taking direct limits.
\end{proof}
Note that throughout this paper, each complex has the first term in degree $-1$, unless stated otherwise. 
	
%%%Let $p$ be an odd prime. 
Define $D^{sep}:= \mathcal{O}_{\widehat{\mathcal{E}^{ur}}}\otimes_{\mathcal{O}_L}V$. As $\mathcal{O}_{\widehat{\mathcal{E}^{ur}}}\otimes_{\mathcal{O}_L}V  \cong \mathcal{O}_{\widehat{\mathcal{E}^{ur}}}\otimes_{\mathcal{O}_\mathcal{E}}\mathbb{D}_{LT}(V)$, we have $D^{sep} \cong \mathcal{O}_{\widehat{\mathcal{E}^{ur}}}\otimes_{\mathcal{O}_\mathcal{E}}\mathbb{D}_{LT}(V)$. Now define the co-chain complex $\Phi^{\bullet}(D^{sep})$ as follows:
	\begin{equation*}
	\Phi^{\bullet}(D^{sep}): 0\rightarrow D^{sep}\xrightarrow{\varphi_q\otimes\varphi_{\mathbb{D}_{LT}(V)}-id}D^{sep}\rightarrow 0. 
	\end{equation*}
	\begin{lemma} \label{augmentation}
	For any discrete $\pi_L$-primary representation $V$ of $G_K$, let $V[0]$ be the complex with $V$ in degree $0$ and $0$ everywhere else. Then the augmentation map $V[0]\rightarrow \Phi^{\bullet}(D^{sep})$ is a quasi-isomorphism of co-chain complexes.
	\end{lemma} 
	\begin{proof}
By Lemma \ref{Exact}, we know that the complex $\Phi^{\bullet}(E^{sep})$ is acyclic in non-zero degrees with $0$-th cohomology equal to $k_L$, the augmentation map $$k_L[0]\rightarrow\Phi^{\bullet}(E^{sep})$$ is a quasi-isomorphism. By dévissage, the augmentation map
		\begin{equation} \label{A}
		\mathcal{O}_L/\pi_L^n[0]\rightarrow \Phi^{\bullet}(\mathcal{O}_{\widehat{\mathcal{E}^{ur}}}/\pi^n)
		\end{equation} 
		is also a quasi-isomorphism as each term in both complexes is a flat $\mathcal{O}_L/\pi_L^n$-module. If $V$ is finite abelian $\pi_L$-group then it is killed by some power of $\pi_L$, and we have $\Phi^{\bullet}(D^{sep})= \Phi^{\bullet}(\mathcal{O}_{\widehat{\mathcal{E}^{ur}}}/\pi_L^n)\otimes_{\mathcal{O}_L/\pi_L^n}V$. Since $V$ is free $\mathcal{O}_L/\pi_L^n$-module, tensoring with $V$ is an exact functor. Now tensoring (\ref{A}) with $V$, we get
		\begin{equation*}
		V[0]\rightarrow \Phi^{\bullet}(D^{sep})
		\end{equation*}  
		is a quasi-isomorphism. Since the direct limit functor is an exact functor, the general case follows by taking direct limits. 
	\end{proof}
	\begin{lemma}\label{trivial cohom}
	$H^i(H_K,\mathcal{O}_{\widehat{\mathcal{E}^{ur}}}/\pi_L^n )=0$ for all $n\geq 1$ and $i\geq 1$.
	\end{lemma}
	\begin{proof}
		By d\'{e}vissage, we are reduced to the case $n=1$, i.e., we only need to prove that $H^i(H_K, E^{sep})=0$ for all $i\geq 1$. But this is a standard fact of Galois cohomology \cite[Proposition $6.1.1$]{NSW}.
	\end{proof} 
	\begin{proposition}\label{H_K-cohomology}
	For any $V \in {\bf Rep}_{\mathcal{O}_L-tor}^{dis}(G_K)$, we have $\mathcal{H}^i(\Phi^{\bullet}(\mathbb{D}_{LT}(V)))\cong H^i(H_K, V)$ as $\Gamma_{LT}$-modules. In other words, the complex $\Phi^{\bullet}(\mathbb{D}_{LT}(V))$ computes the $H_K$-cohomology of $V$.  
	\end{proposition}
	\begin{proof}
	Assume that $V$ is finite. Then by definition,
		\begin{equation*}
		\mathbb{D}_{LT}(V)= (\mathcal{O}_{\widehat{\mathcal{E}^{ur}}}\otimes_{\mathcal{O}_L}V)^{H_K}= (D^{sep})^{H_K}.
		\end{equation*}
		So the complex $\Phi^{\bullet}(\mathbb{D}_{LT}(V))$ is the $H_K$-invariant part of $\Phi^{\bullet}(D^{sep})$. Since $V$ is finite, the terms of $\Phi^{\bullet}(D^{sep})$ are of the form $D^{sep}= E^{sep}\otimes_E \mathbb{D}_{LT}(V)$ and are acyclic objects for the $H_K$-cohomology by using Lemma \ref{trivial cohom}. Then it follows from Lemma \ref{augmentation} that $\mathcal{H}^i(\Phi^{\bullet}(\mathbb{D}_{LT}(V)))\cong H^i(H_K, V)$ as $\Gamma_{LT}$-modules. Since both the functors $\mathcal{H}^i(\Phi^{\bullet}(\mathbb{D}_{LT}(-)))$ and $H^i(H_K,-)$ commute with the filtered direct limits, the general case follows by taking the direct limits.
	\end{proof} 
{\col Note that by the structure of a unit group in a local field, we have $\mathcal{O}_L^\times\cong \mu_{q-1}\bigoplus \mathbb{Z}/p^a\mathbb{Z}\bigoplus \mathbb{Z}_p^d$, where $a\geq 0$ and $d$ is the degree of $L$ over $\mathbb{Q}_p$ (\cite[Chapter II, Proposition 5.7]{Neu1}). {\colo Note that the summand $\mathbb{Z}_p^d$ comes from the free part of the unit group $U^1:= 1+<\pi>$.}
	
From now on, in this section, we consider $K$ to be field fixed by $\mu_{q-1}\bigoplus \mathbb{Z}_p^d$, so that $\Gamma_{LT}=\mu_{q-1}\bigoplus\mathbb{Z}_p^d$. Let $\Delta=\mu_{q-1}$ and $H_K^*$ the kernel of the quotient map $G_K\twoheadrightarrow\Gamma_{LT}\twoheadrightarrow\Gamma_{LT}^*:=\Gamma_{LT}/\Delta$. Here $p\nmid q-1$, so the order of $\Delta$ is prime to $p$.} 

	\begin{proposition}\label{H_K^*-cohomo}
		The complex $\Phi^{\bullet}(\mathbb{D}_{LT}(V)^{\Delta})$ computes the $H_K^*$-cohomology of $V$.
	\end{proposition}
	\begin{proof}
Since %%{\C $p$ is odd. Then} 
the order of $\Delta$ is prime to $p$, the $p$-cohomological dimension of $\Delta$ is zero. Moreover, the isomorphism $H_K^*/H_K\cong \Delta$ gives the following short exact sequence
		\begin{equation*}
		0\rightarrow H_K\rightarrow H_K^*\rightarrow \Delta\rightarrow 0.
		\end{equation*}
		Now the result follows from the Hochschild-Serre spectral sequence together with Proposition \ref{H_K-cohomology}.
	\end{proof} 
Note that %%\sout{$\Gamma_{LT}^*$ is torsion-free {\C and we have}} %%%. Assume that 
$\Gamma_{LT}^*\cong\bigoplus_{i=1}^d\mathbb{Z}_p$ as an abelian group, %%% $\mathbb{Z}_p$-module, 
where $d$ is the degree of $L$ over $\mathbb{Q}_p$. Let $\Gamma_{LT}^*$ be topologically generated by the set $\mathfrak{X}:=\{\gamma_1,\gamma_2,\ldots,\gamma_d\}$. {\C Let $A$ be an $\mathcal{O}_L$-module with a continuous and linear action of $\Gamma_{LT}^*$.} Then consider the co-chain complex
	\begin{equation*}
	\Gamma_{LT}^{\bullet}(A): 0\rightarrow A \rightarrow\bigoplus_{i_1\in \mathfrak{X}}A\rightarrow\cdots\rightarrow\bigoplus_{\{i_1,\ldots,i_r\}\in \binom{\mathfrak{X}}{r}}A\rightarrow\cdots \rightarrow A\rightarrow 0,
	\end{equation*}
	where $\binom{\mathfrak{X}}{r}$ denotes choosing $r$-indices at a time from the set $\mathfrak{X}$, and for all $0\leq r\leq \lvert\mathfrak{X}\rvert-1$, the map $d_{i_1,\ldots,i_r}^{j_1,\ldots, j_{r+1}}:A\rightarrow A$ from the component in the $r$-th term corresponding to $\{i_1,\ldots,i_r\}$ to the component corresponding to the $(r+1)$-tuple $\{j_1,\ldots,j_{r+1}\}$ is given by
	\begin{equation*}
	d_{i_1,\ldots,i_r}^{j_1,\ldots, j_{r+1}} =
\left\{
	\begin{array}{ll}
	0  & \mbox{if } \{i_1,\ldots,i_r\}\nsubseteq\{j_1,\ldots,j_{r+1}\}, \\
	(-1)^{s_j}(\gamma_j-id) & \mbox{if } \{j_1,\ldots,j_{r+1}\}= \{i_1,\ldots,i_r\}\cup\{j\},
	\end{array}
	\right.
	\end{equation*}
and $s_j$ is the number of elements in the set $\{i_1,\ldots,i_r\}$, which are smaller than $j$. This is the Koszul complex of $A$
as a module over $\cO_L[[\Gamma_{LT}^\ast]]$, and the differentials above are the differentials of the Koszul complex with respect to the ordered sequence $\gamma_1-1,\gamma_2-1,\cdots,\gamma_d-1$.
	\begin{example}
		Let $d=2$. Then the complex $\Gamma_{LT}^\bullet(A)$ is defined as follows:
		\begin{equation*}
		\Gamma_{LT}^\bullet(A): 0\rightarrow A\xrightarrow{x\mapsto A_0x} A \oplus A\xrightarrow{x \mapsto A_1x} A\rightarrow 0,
		\end{equation*}
		where
		\[
			A_0=
			\begin{bmatrix}
			\gamma_1-id\\ \gamma_2-id 
			\end{bmatrix},
			A_1=
			\begin{bmatrix}
			-(\gamma_2-id) & \gamma_1-id 
			\end{bmatrix}. \qedhere
			\]
	\end{example}
	\begin{lemma}\label{H^0}
		The functor $A\mapsto \mathcal{H}^i(\Gamma_{LT}^{\bullet}(A))_{i\geq 0}$ is a cohomological $\delta$-functor. Moreover, if $A$ is a discrete representation of %%%% abelian group on which 
		$\Gamma_{LT}^*$ %%%%acts continuously, 
		then $\mathcal{H}^0(\Gamma_{LT}^{\bullet}(A))= A^{\Gamma_{LT}^*}$.
	\end{lemma}
	\begin{proof}
	Let 
		\begin{equation}\label{B}
		0\rightarrow A\rightarrow B\rightarrow C\rightarrow 0
		\end{equation} be a short exact sequence of representations of $\Gamma_{LT}^*$. Then we have a short exact sequence
		\begin{equation}\label{C}
		0\rightarrow \Gamma_{LT}^{\bullet}(A)\rightarrow \Gamma_{LT}^{\bullet}(B)\rightarrow \Gamma_{LT}^{\bullet}(C)\rightarrow 0
		\end{equation}
	 of co-chain complexes. Then the long exact cohomology sequence of (\ref{C}) gives maps
		\begin{equation*}
		\delta^i: \mathcal{H}^i(\Gamma_{LT}^{\bullet}(C))\rightarrow \mathcal{H}^{i+1}(\Gamma_{LT}^{\bullet}(A)),
		\end{equation*}
		which are functorial in (\ref{B}). Therefore  $A\mapsto \mathcal{H}^i(\Gamma_{LT}^{\bullet}(A))$ is a cohomological $\delta$-functor. The second part follows from the fact that the action of $\Gamma_{LT}^*$ on $A$ factors through a finite quotient. Since the classes of the elements $\gamma_i \ (i\in \mathfrak{X})$ generate finite quotients of $\Gamma_{LT}^*$, $A^{\Gamma_{LT}^*}= \cap_{i\in \mathfrak{X}}\Ker(\gamma_i-id)= \mathcal{H}^0(\Gamma_{LT}^{\bullet}(A))$.
	\end{proof}
	\begin{proposition}\label{Gamma^*-cohomo}
	Let $A$ be a discrete $\pi_L$-primary representation %%%% abelian group with a continuous action 
		of $\Gamma_{LT}^*$. Then $\mathcal{H}^i(\Gamma_{LT}^{\bullet}(A))\cong H^i(\Gamma_{LT}^*,A)$ for $i\geq0$. In other words, the complex $\Gamma_{LT}^{\bullet}(A)$ computes the $\Gamma_{LT}^*$-cohomology of $A$.
	\end{proposition}
\begin{proof}
We prove the proposition by using induction on the number of generators of $\Gamma_{LT}^*$. First, assume that $\Gamma_{LT}^*$ is topologically generated by $\langle\gamma_1,\gamma_2\rangle$. Let $\Gamma^*_{\gamma_1}$ denote the subgroup of $\Gamma_{LT}^*$ generated by $\gamma_1$ and $\Gamma_{\gamma_2}^*$ the quotient of $\Gamma_{LT}^*$ by $\Gamma_{\gamma_1}^*$. We denote by $\Gamma_{\gamma_i}^{\bullet}(A)$ the co-chain complex
	\begin{equation*}
	\Gamma_{\gamma_i}^{\bullet}(A): 0\rightarrow A\xrightarrow{\gamma_i-id} A\rightarrow 0.
	\end{equation*}
	Then the co-chain complex $\Gamma_{LT}^{\bullet}(A)$ is the total complex of the double complex $\Gamma_{\gamma_2}^{\bullet}(\Gamma_{\gamma_1}^{\bullet}(A))$, and associated to the double complex $\Gamma_{\gamma_2}^{\bullet}(\Gamma_{\gamma_1}^{\bullet}(A))$, there is a spectral sequence 
	\begin{equation}\label{SSADC}
	E_2^{mn} = \mathcal{H}^m(\Gamma_{\gamma_2}^{\bullet}(\mathcal{H}^n(\Gamma_{\gamma_1}^{\bullet}(A))))\Rightarrow \mathcal{H}^{m+n}(\Gamma_{LT}^{\bullet}(A)).
	\end{equation} 
	Moreover, associated to the group $\Gamma_{LT}^*$, we have the Hochschild-Serre spectral sequence \begin{equation}\label{HSSS}
	E_2^{mn}=H^m(\Gamma_{\gamma_2}^*,H^n(\Gamma_{\gamma_1}^*,A))\Rightarrow H^{m+n}(\Gamma_{LT}^*,A).
		\end{equation}
Now assume that $A$ is an injective object in the category of discrete $\pi_L$-primary abelian groups with a continuous action of $\Gamma_{LT}^*$. Then the complex $\Gamma_{\gamma_1}^{\bullet}(A)$ is acyclic in non-zero degrees with $0$-th  
	cohomology isomorphic to $H^0(\Gamma_{\gamma_1}^*,A)=A^{\Gamma_{\gamma_1}^*}$ \cite[Corollary 6.41]{rot}, i.e, the map $A^{\Gamma_{\gamma_1}^*}[0]\rightarrow \Gamma_{\gamma_1}^{\bullet}(A)$ is a quasi-isomorphism. But $A^{\Gamma_{\gamma_1}^*}$ is an injective object in the category of discrete $\pi$-primary abelian groups with a continuous action of $\Gamma_{\gamma_2}^*$. Now by using step 1 and step 2 of \cite[Proposition 2.1.7]{PZ},  the map $A^{\Gamma_{LT}^*}[0]\rightarrow \Gamma_{\gamma_2}^{\bullet}(A^{\Gamma_{\gamma_1}^*}) $ is a quasi-isomorphism of co-chain complexes.
	
	Note that the functor $H^i(\Gamma_{LT}^*,-)$ is a universal $\delta$-functor, and $\mathcal{H}^i(\Gamma_{LT}^{\bullet}(-))$ is a cohomological $\delta$-functor such that $H^0(\Gamma_{LT}^*,-)\cong \mathcal{H}^0(\Gamma_{LT}^{\bullet}(-))$. Therefore, we have a natural transformation $H^i(\Gamma_{LT}^*,-)\rightarrow\mathcal{H}^i(\Gamma_{LT}^{\bullet}(-))$ of $\delta$-functors. Then by using spectral sequences (\ref{SSADC}) and (\ref{HSSS}), we have 
		\begin{equation*}
		H^i(\Gamma_{LT}^*,A)\cong \mathcal{H}^i(\Gamma_{LT}^{\bullet}(A)) \quad \text{for} \ i\geq 0.
		\end{equation*} 
	Now the case for general $A$ follows from Lemma \ref{H^0} by using dimension shifting. Then by induction assume that the result is true when $\Gamma_{LT}^*$ is topologically generated by $\langle \gamma_1,\gamma_2,\ldots,\gamma_{d-1}\rangle$. Now we want to prove the proposition when $\Gamma_{LT}^* = \langle\gamma_1,\gamma_2,\ldots,\gamma_d\rangle$.
		Consider the complexes
		\begin{equation*}
		\Gamma^\bullet_{\gamma_d}(A): 0\rightarrow A\xrightarrow{\gamma_d-id}A\rightarrow 0,
		\end{equation*} and
		\begin{equation*}
		\Gamma_{LT\backslash \gamma_d}^{\bullet}(A): 0\rightarrow A \rightarrow\bigoplus_{i_1\in \mathfrak{X^\prime}}A\rightarrow\cdots\rightarrow\bigoplus_{\{i_1,\ldots,i_r\}\in \binom{\mathfrak{X^\prime}}{r}}A\rightarrow\cdots \rightarrow A\rightarrow 0,
		\end{equation*}
		where $\mathfrak{X^\prime}= \{\gamma_1,\ldots,\gamma_{d-1}\}$, and for all $0\leq r\leq \lvert\mathfrak{X^\prime}\rvert-1$, the map $d_{i_1,\ldots,i_r}^{j_1,\ldots, j_{r+1}}:A\rightarrow A$ from the component in the $r$-th term corresponding to $\{i_1,\ldots,i_r\}$ to the component corresponding to the $(r+1)$-tuple $\{j_1,\ldots,j_{r+1}\}$ is given by the following
		\begin{equation*}
			d_{i_1,\ldots,i_r}^{j_1,\ldots, j_{r+1}} =
			\left\{
			\begin{array}{ll}
			0  & \mbox{if } \{i_1,\ldots,i_r\}\nsubseteq\{j_1,\ldots,j_{r+1}\}, \\
			(-1)^{s_j}(\gamma_j-id) & \mbox{if } \{j_1,\ldots,j_{r+1}\}= \{i_1,\ldots,i_r\}\cup\{j\},
			\end{array}
			\right.
		\end{equation*}
		and $s_j$ is the number of elements in the set $\{i_1,\ldots,i_r\}$ smaller than $j$. Note that the complex $\Gamma_{LT}^\bullet(A)$ is the total complex of the double complex $\Gamma^\bullet_{\gamma_d}(\Gamma_{LT\backslash \gamma_d}^{\bullet}(A))$. %%%%%%%% where the complexes $\Gamma^\bullet_{\gamma_d}(A)$ and $\Gamma_{LT\backslash \gamma_d}^{\bullet}(A)$ are defined as the following
		Since the result is true for $\Gamma_{LT\backslash \gamma_d}^{\bullet}(A)$ by using induction hypothesis, the proof follows by using similar techniques as explained in the case when $\Gamma_{LT}^*$ is generated by $\gamma_1$ and $\gamma_2$.
\end{proof}
	\begin{remark} \label{Independence of generators}
For any representation $A$ of $\Gamma_{LT}^*$, clearly, the complex $\Gamma_{LT}^{\bullet}(A)$ depends on the choice of generators of $\Gamma_{LT}^*$. The cohomology groups of the Koszul complex $\Gamma_{LT}^{\bullet}(A)$ are independent of the choice of generators of $\Gamma_{LT}^*$.
	\end{remark}   
	\begin{proof} 
%%%{\C Since $A\otimes_{\cO_K[[\Gamma_{LT}^\ast]]}\cO_K[[\Gamma_{LT}^\ast]]\cong A$. %%% It is enough to prove that the cohomology groups of the complex $\Gamma_{LT}^\bullet(\cO_K[[\Gamma_{LT}^\ast]])$ are independent of the choice of generators of $\Gamma_{LT}^*$}It is enough to prove for $\cO_K[[\Gamma_{LT}^\ast]]$, as the general case, is obtained by tensoring with $A$.%To prove this, we use induction on the number of generators of $\Gamma_{LT}^*$. 
Assume that $\Gamma_{LT}^*=\langle\gamma_1,\gamma_2,\ldots,\gamma_n\rangle$ for generators $\gamma_i$. Let $\Gamma_{LT}^*= \langle\gamma_1^\prime,\gamma_2^\prime,\ldots,\gamma_n^\prime\rangle$ be another set of generators. Then, the ideal of $\cO_L[[\Gamma_{LT}^\ast]]$ generated by 
$\{\gamma_1-1,\gamma_2-1,\ldots,\gamma_n-1\}$ is equal to the ideal generated by $\{\gamma_1^\prime-1,\gamma_2^\prime-1,\ldots,\gamma_n^\prime-1\}$.
% Therefore, $\gamma_1'-1=r_1(\gamma_1-1)+r_2(\gamma_2-1)+\cdots+r_n(\gamma_n-1)$ where $r_i\in\cO[[\Gamma_{LT}^\ast]]$ and $r_1$ is a unit of $\cO[[\Gamma_{LT}^\ast]]$.
% as $\gamma'-1,\gamma_2-1,\cdots,\gamma_n-1$ also generate the ideal.

Setting $T_1=\gamma_1-1, T_2=\gamma_2-1,\ldots,T_n=\gamma_n-1$ and 
 $S_1=\gamma_1^\prime-1,S_2=\gamma_2^\prime-1,\ldots,S_n=\gamma_n^\prime-1$, we have an isomorphism of rings $\cO_L[[\Gamma_{LT}^*]]\cong\cO_L[[\Gamma_{LT}^*]]$ which induces an isomorphism between $\cO_L[[T_1,T_2,\ldots,T_n]]$ and $\cO_L[[S_1,S_2,\ldots,S_n]]$ via the maps $T_i\mapsto S_i$. 
 %Then, we get a map $\cO[[\Gamma_{LT}]]$-modules $T_i\mapsto\gamma_i-1\mapsto S_i\mapsto \gamma_i'-1$
 %which induces an isomorphism between $\cO[[T_1,T_2,\cdots,T_n]]$ and $\cO[[S_1,S_2,\cdots,S_n]]$. 
 Therefore the matrix of this isomorphism which has entries in $\cO_L[[\Gamma_{LT}^\ast]]$ is invertible there. Then it follows that the Koszul complexes with respect
to the sequences $\{T_i:i=1,\ldots,n\}$ and $\{S_i:i=1,\ldots,n\}$ are quasi-isomorphic (see \cite[Section 15.28]{stacks-project}).
	\end{proof}
%261
	Next, we define a complex, namely the Lubin-Tate Herr complex, which is a generalization of the Herr complex \cite{LH1}.
	\begin{definition} \label{LTHC}
		Let $M\in \varinjlim {\bf Mod}^{\varphi_q,\Gamma_{LT},\acute{e}t,tor}_{/\mathcal{O}_{\mathcal{E}}}$. Define the co-chain complex $\Phi\Gamma_{LT}^{\bullet}(M)$ as the total complex of the double complex $\Gamma_{LT}^{\bullet}(\Phi^{\bullet}(M^{\Delta}))$, and we call it the \emph{Lubin-Tate Herr complex} for $M$.
	\end{definition} 
%%%	Explicitly for the case $d=2$, the Lubin-Tate Herr complex looks like as in the following example, where we write %% Note that in the following examples 
	$M$ for $M^{\Delta}$ %%. We write $M$ only 
	for simplicity.
	\begin{example}{\label{Ex1}}
		Let $d=2$, and writing $M$ for $M^\Delta$ for simplicity, the Lubin-Tate Herr complex  $\Phi\Gamma_{LT}^{\bullet}(M)$ is defined as:
		\begin{equation*}
		0\rightarrow M\xrightarrow{x\mapsto A_{0,\varphi_q}x}M^{\oplus 3}\xrightarrow{x\mapsto A_{1,\varphi_q}x} M^{\oplus 3}\xrightarrow{x\mapsto A_{2,\varphi_q}x}M\rightarrow 0,
		\end{equation*}
		where
		\[
		A_{0,\varphi_q}=
		\begin{bmatrix}
		\varphi_M-id \\
		\gamma_1-id \\
		\gamma_2-id 
		\end{bmatrix},
		A_{1,\varphi_q} = 
		\begin{bmatrix}
		-(\gamma_1-id) & \varphi_M-id & 0  \\
		-(\gamma_2-id) & 0 & \varphi_M-id \\
		0 & -(\gamma_2-id) & \gamma_1-id 
		\end{bmatrix}, 
		\]
		\[
		A_{2,\varphi_q}=
		\begin{bmatrix}
		\gamma_2-id & -(\gamma_1-id)& \varphi_M-id  
		\end{bmatrix}.
		\]
	\end{example}
	\begin{lemma}\label{dimension-shifting}
		For any $V \in {\bf Rep}_{\mathcal{O}_L-tor}^{dis}(G_K)$, the functor $V \mapsto \mathcal{H}^i(\Phi\Gamma_{LT}^{\bullet}(\mathbb{D}_{LT}(V)))_{i\geq 0}$ is a cohomological $\delta$-functor from the category ${\bf Rep}_{\mathcal{O}_L-tor}^{dis}(G_K)$ to the category of abelian groups. Moreover, $\mathcal{H}^0(\Phi\Gamma_{LT}^{\bullet}(\mathbb{D}_{LT}(V)))\cong V^{G_K}.$
	\end{lemma}
	\begin{proof}
		Let 
		\begin{equation}\label{D}
		0\rightarrow V_1\rightarrow V_2\rightarrow V_3\rightarrow 0
		\end{equation}
		be a short exact sequence of discrete $\pi_L$-primary representations of $G_K$. The functor $\mathbb{D}_{LT}$ is exact {\cv and} gives the short exact sequence $0\rightarrow \mathbb{D}_{LT}(V_1)\rightarrow\mathbb{D}_{LT}(V_2)\rightarrow\mathbb{D}_{LT}(V_3)\rightarrow 0$ in $\varinjlim {\bf Mod}^{\varphi_q,\Gamma_{LT},\acute{e}t,tor}_{/\mathcal{O}_{\mathcal{E}}}$. By using Acyclic Assembly Lemma \cite[Lemma 2.7.3]{Wei}, we get a short exact sequence 
		\begin{equation}\label{E}
		0\rightarrow\Phi\Gamma_{LT}^{\bullet}(\mathbb{D}_{LT}(V_1))\rightarrow\Phi\Gamma_{LT}^{\bullet}(\mathbb{D}_{LT}(V_2))\rightarrow\Phi\Gamma_{LT}^{\bullet}(\mathbb{D}_{LT}(V_3))\rightarrow 0,
		\end{equation}
	of co-chain complexes. Then the long exact sequence of (\ref{E}) gives maps
		\begin{equation*}
		\delta^i: \mathcal{H}^i(\Phi\Gamma_{LT}^{\bullet}(\mathbb{D}_{LT}(V_3)))\rightarrow\mathcal{H}^{i+1}(\Phi\Gamma_{LT}^{\bullet}(\mathbb{D}_{LT}(V_1))),
		\end{equation*}
		which are functorial in (\ref{D}). Therefore $V \mapsto \mathcal{H}^i(\Phi\Gamma_{LT}^{\bullet}(\mathbb{D}_{LT}(V)))_{i\geq 0}$ is a cohomological $\delta$-functor from the category ${\bf Rep}_{\mathcal{O}_L-tor}^{dis}(G_K)$ to the category of abelian groups.
		
		For the second part,  note that $\varphi_q$ acts trivially on $V$ and it commutes with the action of $G_K$, so:
		\begin{align*}
		\mathbb{D}_{LT}(V)^{\varphi_{\mathbb{D}_{LT}(V)} = id} & = ((\mathcal{O}_{\widehat{\mathcal{E}^{ur}}}\otimes_{\mathcal{O}_L} V)^{H_K})^{\varphi_{\mathbb{D}_{LT}(V)}=id}\\ & = (\mathcal{O}_{\widehat{\mathcal{E}^{ur}}}^{\varphi_q= 1}\otimes_{\mathcal{O}_L} V)^{H_K}\\ & = (\mathcal{O}_L\otimes_{\mathcal{O}_L} V)^{H_K} \\ & \cong V^{H_K},
		\end{align*} where the third equality follows from Lemma \ref{Exact}. Therefore 
	 $$\mathbb{D}_{LT}(V)^{\varphi_{\mathbb{D}_{LT}(V)}=id,\Gamma_{LT}=id} \cong (V^{H_K})^{\Gamma_{LT}=id}=V^{G_K}.$$ On the other hand, 
	 \begin{align*}
	 \mathcal{H}^0(\Phi\Gamma_{LT}^{\bullet}(\mathbb{D}_{LT}(V)))=& (\mathbb{D}_{LT}(V)^\Delta)^{\varphi_{\mathbb{D}_{LT}(V)} = id,\Gamma_{LT}^*=id}\\=&\mathbb{D}_{LT}(V)^{\varphi_{\mathbb{D}_{LT}(V)} = id,\Gamma_{LT}=id}.
	 \end{align*}  Hence
		\begin{equation*}
		\mathcal{H}^0(\Phi\Gamma_{LT}^{\bullet}(\mathbb{D}_{LT}(V)))\cong V^{G_K}.
		\end{equation*} 
	\end{proof}
	\begin{theorem}\label{G_K-cohomo}
		Let $V \in {\bf Rep}_{\mathcal{O}_L-tor}^{dis}(G_K)$. Then $H^i(G_K,V)\cong \mathcal{H}^i(\Phi\Gamma_{LT}^{\bullet}(\mathbb{D}_{LT}(V)))$ for $i\geq 0$, i.e., the Lubin-Tate Herr complex $\Phi\Gamma_{LT}^{\bullet}(\mathbb{D}_{LT}(V))$ computes the Galois cohomology of $G_K$ with coefficients in $V$. 
	\end{theorem}
\begin{proof}
As $(H^i(G_K,-))_{i\geq 0}$ is a universal $\delta$-functor and $(\mathcal{H}^i(\Phi\Gamma_{LT}^{\bullet}(\mathbb{D}_{LT}(-))))_{i\geq 0}$ is a cohomological $\delta$-functor such that $\mathcal{H}^0(\Phi\Gamma_{LT}^{\bullet}(\mathbb{D}_{LT}(-)))\cong H^0(G_K,-)$, we have a natural transformation 
\begin{equation*}
H^i(G_K,-)\rightarrow \mathcal{H}^i(\Phi\Gamma_{LT}^{\bullet}(\mathbb{D}_{LT}(-)))
\end{equation*}
of $\delta$-functors. First, assume that $V$ is an injective object in ${\bf Rep}_{\mathcal{O}_L-tor}^{dis}(G_K)$. Then there is a spectral sequence
\begin{equation}\label{F}
E_2^{mn}= \mathcal{H}^m(\Gamma_{LT}^{\bullet}(\mathcal{H}^n(\Phi^{\bullet}(\mathbb{D}_{LT}(V)^{\Delta}))))\Rightarrow \mathcal{H}^{m+n}(\Phi\Gamma_{LT}^{\bullet}(\mathbb{D}_{LT}(V)))
\end{equation}
associated to the double complex $\Gamma_{LT}^{\bullet}(\Phi^{\bullet}(\mathbb{D}_{LT}(V)^\Delta))$, and associated to the group $G_K$, we have the Hochschild-Serre spectral sequence
	\begin{equation}\label{hsss}
	E_2^{mn}= H^m(\Gamma_{LT}^*,H^n(H_K^*,V))\Rightarrow H^{m+n}(G_K,V).
	\end{equation} 
Since $V$ is injective, it follows from Proposition \ref{H_K^*-cohomo} that the augmentation map $$V^{H_K^*}[0]\rightarrow \Phi^{\bullet}(\mathbb{D}_{LT}(V)^\Delta)$$ is a quasi-isomorphism. Also, $V^{H_K^*}$ is injective as a discrete representation of $\Gamma_{LT}^*$. Then by using Proposition \ref{Gamma^*-cohomo}, the map $$V^{G_K}[0]\rightarrow\Gamma_{LT}^{\bullet}(V^{H_K^*})$$ is a quasi-isomorphism of complexes. Now the natural transformation $H^i(G_K,-)\rightarrow \mathcal{H}^i(\Phi\Gamma_{LT}^{\bullet}(\mathbb{D}_{LT}(-)))$ and the spectral sequences (\ref{F}) and (\ref{hsss}) give the following isomorphism 
\begin{equation*}
H^i(G_K,V)\cong \mathcal{H}^i(\Phi\Gamma_{LT}^{\bullet}(\mathbb{D}_{LT}(V)))\quad \text{for}\ i\geq 0. 
\end{equation*}
As we know that the category ${\bf Rep}_{\mathcal{O}_L-tor}^{dis}(G_K)$ has enough injectives, the general case follows from Lemma \ref{dimension-shifting} by using dimension shifting.
\end{proof}
Let $V \in {\bf Rep}_{\mathcal{O}_L}(G_K)$. Then we have 
\begin{align*}
V&=\varprojlim V\otimes_{\mathcal{O}_L}\mathcal{O}_L/ \pi_L^n\mathcal{O}_L\\ &\cong\varprojlim V/\pi_L^nV,
\end{align*}
where each $V/\pi_L^nV$ is $\pi_L$-power torsion and it is also discrete as $V/\pi_L^nV$ is finite. Therefore any object in ${\bf Rep}_{\mathcal{O}_L}(G_K)$ is the inverse limit of objects in the category ${\bf Rep}_{\mathcal{O}_L-tor}^{dis}(G_K)$.
	\begin{lemma}\label{commutes inverse}
		Let $V\in {\bf Rep}_{\mathcal{O}_L}(G_K)$. Then the functor $H^i(G_K,-)$ commutes with the inverse limits, i.e., $H^i(G_K,V)\cong \varprojlim\limits_n H^i(G_K,V/\pi_L^nV)$. 
	\end{lemma}
	\begin{proof}
		As $k$ is finite, the cohomology groups $H^i(G_K,V/\pi_L^nV)$ are finite for all $n$ (\cite[Theorem $2.1$]{Ta1}). So, {\cv the Mittag-Leffler condition holds, and} the result follows from \cite[Corollary $2.2$]{Tate}.
	\end{proof}
	\begin{theorem}\label{lattices}
		Let $V\in {\bf Rep}_{\mathcal{O}_L}(G_K)$. Then we have $H^i(G_K,V)\cong \mathcal{H}^i(\Phi\Gamma_{LT}^\bullet(\mathbb{D}_{LT}(V)))$ for $i\geq0$.
	\end{theorem}
	\begin{proof}
		Firstly, we show that the functor $\mathcal{H}^i(\Phi\Gamma_{LT}^\bullet(\mathbb{D}_{LT}(-)))$ commutes with the inverse limits. Here the transition maps are surjective in the projective system $(\Phi\Gamma_{LT}^\bullet(\mathbb{D}_{LT}(V/\pi_L^nV)))_n$ of co-chain complexes of abelian groups, so the first hyper-cohomology spectral sequence degenerates at $E_2$. Moreover, it follows from Lemma \ref{commutes inverse} that $\varprojlim\limits_n{}^{1} \mathcal{H}^i(\Phi\Gamma_{LT}^\bullet(\mathbb{D}_{LT}(V/\pi_L^nV)))=0$. Therefore the second hyper-cohomology spectral sequence 
	\begin{equation*}
\varprojlim\limits_n{}^{i} \mathcal{H}^j(\Phi\Gamma_{LT}^\bullet(\mathbb{D}_{LT}(V/\pi_L^nV))) \Rightarrow \mathcal{H}^{i+j}(\Phi\Gamma_{LT}^\bullet(\mathbb{D}_{LT}(V)))
	\end{equation*} 
		also degenerates at $E_2$. Thus $\varprojlim\limits_n \mathcal{H}^i(\Phi\Gamma_{LT}^\bullet(\mathbb{D}_{LT}(V/\pi_L^nV)))= \mathcal{H}^i(\Phi\Gamma_{LT}^\bullet(\mathbb{D}_{LT}(V)))$. Now
		\begin{align*}
		H^i(G_K,V)&\cong \varprojlim\limits_n H^i(G_K,V/\pi_L^nV) \\ & \cong \varprojlim\limits_n \mathcal{H}^i(\Phi\Gamma_{LT}^\bullet(\mathbb{D}_{LT}(V/\pi_L^nV))) \\ & \cong \mathcal{H}^i(\Phi\Gamma_{LT}^\bullet(\mathbb{D}_{LT}(V))),
		\end{align*} where the first isomorphism follows from Lemma \ref{commutes inverse} and the second is induced from Theorem \ref{G_K-cohomo}.
	\end{proof}

The following corollary is not obvious from the definition of the Lubin-Tate Herr complex. 

\begin{corollary}\label{zero}
Let $V\in {\bf Rep}_{\mathcal{O}_L}(G_K)$. Then $\mathcal{H}^i(\Phi\Gamma_{LT}^\bullet(\mathbb{D}_{LT}(V)))=0$ for $i\geq 3$. %%% although it is not obvious from the definition of the Lubin-Tate Herr complex. 
	\end{corollary}
\begin{proof}
Recall the classical result that the groups $H^i(G_K,V)$ are trivial for $i\geq 3$, \cite[chapter II, Proposition 12]{Ser}. Then it follows from the above theorem that $\mathcal{H}^i(\Phi\Gamma_{LT}^\bullet(\mathbb{D}_{LT}(V)))=0$ for $i\geq3$.
\end{proof} 
\section{Galois Cohomology over the False-Tate Type Extensions}\label{sec4}
%%%%In this section, we assume that $K$ contains the $\pi$-torsion points of the Lubin-Tate group $\mathcal{G}$. 
Recall that {\col $\mathcal{G}$ is a Lubin-Tate formal group defined over $L$ with a uniformizer $\pi_L$ and} $\mathfrak{m}_L$ is the maximal ideal of $\mathcal{O}_L$. For any $x\in\mathfrak{m}_L\backslash\mathfrak{m}^2_L$, choose a system $(x_i)_{i\geq1}$ such that $[p](x_1)= x$ and $[p](x_{i+1})= x_i$ for all $i\geq 1$. Define $\widetilde{L}:= L(x_i)_{i\geq 1}$ {\col and $\widetilde{L}_\infty:= L_{\infty}\widetilde{L}$}. Then $\widetilde{L}/L$ is not a Galois {\col extension but $\widetilde{L}_\infty/L$ is Galois. Let $K$ be a finite extension of $L$ with $K\subset L_\infty$ such that $\Gal(L_\infty/K)\cong \mathbb{Z}_p^d$. Then the extension $\widetilde{K}_\infty:= K_{\infty}\widetilde{L}$ is a Galois extension with $\Gal(\widetilde K_\infty/K_\infty)\cong \mathcal{O}_L$, as $\Gal(K_\infty(x_i)/K_\infty)$ is isomorphic to $\mathcal{O}_L/p^i\mathcal{O}_L$. Let $F$ be a subfield of $\widetilde K_\infty$ such that $\Gal(F/K_\infty)$ is isomorphic to $\mathbb Z_p$}. %%%Let $\widetilde{L}_\infty:= K_{\infty}\widetilde{K}$; then it is easy to see that the extension $L/K$ is Galois. 
Then the extension $F/K$ is arithmetically pro-finite {\cv in the sense of \cite[\S 1]{Win}} as $\Gal(F/K)$ is a $p$-adic Lie group. %%%As in \cite{Win}, we consider the field of norms for this extension. The fraction field $\Fr(\mathcal{R})$ contains the field of norms $E_L:= X_K(L)$ in a natural way and $\Gal(\bar{K}/L)\cong \Gal(E^{sep}/E_L)$, {\C where the field $E$ is same as the defined in the previous section. For more details, we refer to \cite[\S 2 and \S 3]{Win}.} 
Recall that the ring $\mathcal{O}_{\widehat{\mathcal{E}^{ur}}}$ is a complete discrete valuation ring with residue field $E^{sep}$ and is stable under the action of $G_K$ and $\varphi_q$. Define $\mathcal{O}_{\mathcal{F}}:=(\mathcal{O}_{\widehat{\mathcal{E}^{ur}}})^{\Gal(\bar{K}/F)}$, {\col and $E_F:= (E^{sep})^{\Gal(\bar{K}/F)}$. Then} $\mathcal{O}_{\mathcal{F}}$ is a complete discrete valuation ring with residue field $E_F$. Moreover, {\col $\Gal(\bar{K}/F)\cong \Gal(E^{sep}/E_F)$}, {\C where the field $E$ is same as the one defined in the previous section. For more details, we refer to \cite[\S 2 and \S 3]{Win} and \cite[\S 1.4.2]{Flo}.} Further, the ring $\mathcal{O}_{\mathcal{F}}$ is stable under the action of $G_K$ and $\varphi_q$. Define $\Gamma_{LT,FT}:= \Gal(F/K)$. We summarize the above notations %%%can be summarized 
by Figure \ref{Fig 1}. %%% and Figure \ref{Fig 2}. %%% the above notations. 
 \begin{figure}[h]
 	\centering
 	\begin{tikzcd}
 		& & & \bar{K} \\
 		&{}& & \\
 		& F \arrow[dd, "{\Gamma_{LT,FT}}", no head] \arrow[rruu, "H_F", no head] &  &  \\
 		K\widetilde{L} \arrow[ru, no head] \arrow[ruu, phantom]&  & K_\infty \arrow[lu, no head] \arrow[ruuu, "H_K"', no head] &  \\
 		& K \arrow[ru, no head] \arrow[lu, no head] \arrow[rruuuu, "G_K"', no head, bend right=60] & & 
 	\end{tikzcd}
 	\caption{Field extensions of $K$} \label{Fig 1}
 \end{figure}
\iffalse{
\begin{figure}[h] %Should be Figure 1.1.1-1
%	\begin{center}
\centering
		\begin{minipage}{.50\textwidth}
			\centering
		\begin{tikzcd}
		& & & \bar{K} \\
		&{}& & \\
		& F \arrow[dd, "{\Gamma_{LT,FT}}", no head] \arrow[rruu, "H_F", no head] &  &  \\
		K\widetilde{L} \arrow[ru, no head] \arrow[ruu, phantom]&  & K_\infty \arrow[lu, no head] \arrow[ruuu, "H_K"', no head] &  \\
		& K \arrow[ru, no head] \arrow[lu, no head] \arrow[rruuuu, "G_K"', no head, bend right=60] & & 
		\end{tikzcd}
		\caption{Field extensions of $K$} \label{Fig 1}
	\end{minipage}%
\begin{minipage}{.50\textwidth}
	\centering
		\begin{tikzcd}
		&  & \Fr(\mathcal{R}) \\
		& & \\
		&  & E^{sep} \arrow[uu, no head] \\
		E_F \arrow[rru, "H_F", no head] &  &\\
		& E \arrow[lu, no head] \arrow[ruu, "H_K"', no head] &  
		\end{tikzcd}
		\caption{Field extensions of $E$}\label{Fig 2}
	\end{minipage}
\end{figure}
}\fi

 Now for any $V \in {\bf Rep}_{\mathcal{O}_L}(G_K)$, define  
	\begin{equation*}
	\mathbb{D}_{LT,FT}(V):= (\mathcal{O}_{\widehat{\mathcal{E}^{ur}}}\otimes_{\mathcal{O}_L} V)^{\Gal(\bar{K}/F)}.
	\end{equation*}
	Let ${\bf Mod}^{\varphi_q,\Gamma_{LT,FT},\acute{e}t}_{/\mathcal{O}_{\mathcal{F}}}$ be the category of finite free \'{e}tale $(\varphi_q,\Gamma_{LT,FT})$-modules over $\mathcal{O}_\mathcal{F}$. Then the modules $\mathbb{D}_{LT,FT}(V)$ and $\mathbb{D}_{LT}(V)\otimes_{\mathcal{O}_{\mathcal{E}}}\mathcal{O}_{\mathcal{F}}$ are in the category ${\bf Mod}^{\varphi_q,\Gamma_{LT,FT},\acute{e}t}_{/\mathcal{O}_{\mathcal{F}}}$, and there is a natural map $\iota: \mathbb{D}_{LT}(V)\otimes_{\mathcal{O}_{\mathcal{E}}}\mathcal{O}_{\mathcal{F}}\rightarrow \mathbb{D}_{LT,FT}(V)$.
	\begin{proposition}\label{composite functor}
		The map $\iota$ is an isomorphism of \'{e}tale $(\varphi_q,\Gamma_{LT,FT})$-modules over $\mathcal{O}_{\mathcal{F}}$.
	\end{proposition}
	\begin{proof}
		Consider %%%%%we have the following isomorphism of \'{e}tale $(\varphi_q,\Gamma_{LT})$-modules over $\mathcal{O}_{\mathcal{E}}$,
	%	\begin{equation}
		the isomorphism $\mathcal{O}_{\widehat{\mathcal{E}^{ur}}}\otimes_{\mathcal{O}_{\mathcal{E}}}\mathbb{D}_{LT}(V) \cong \mathcal{O}_{\widehat{\mathcal{E}^{ur}}}\otimes_{\mathcal{O}_L} V$ of \'{e}tale $(\varphi_q,\Gamma_{LT})$-modules over $\mathcal{O}_{\mathcal{E}}$. Then
	%	\end{equation}
		%%The above map is also an isomorphism after extending the scalars, and we have
		\begin{align*}
		\mathbb{D}_{LT}(V)\otimes_{\mathcal{O}_{\mathcal{E}}}\mathcal{O}_{\mathcal{F}} &=\mathbb{D}_{LT}(V)\otimes_{\mathcal{O}_{\mathcal{E}}} (\mathcal{O}_{\widehat{\mathcal{E}^{ur}}})^{\Gal(\bar{K}/F)} \\&= (\mathbb{D}_{LT}(V)\otimes_{\mathcal{O}_{\mathcal{E}}} \mathcal{O}_{\widehat{\mathcal{E}^{ur}}})^{\Gal(\bar{K}/F)}  \\& \cong (\mathcal{O}_{\widehat{\mathcal{E}^{ur}}}\otimes_{\mathcal{O}_L} V)^{\Gal(\bar{K}/L)}\\ &= \mathbb{D}_{LT,FT}(V),
		\end{align*}
			where the second identity follows from the fact that $\Gal(\bar{K}/F)\subseteq \Gal(\bar{K}/K_\infty)$. Thus $\iota$ is an isomorphism of \'{e}tale $(\varphi_q,\Gamma_{LT,FT})$-modules.  
	\end{proof}
Similarly, for any $M \in {\bf Mod}^{\varphi_q,\Gamma_{LT,FT},\acute{e}t}_{/\mathcal{O}_{\mathcal{F}}}$, define
\begin{equation}
\mathbb{V}_{LT,FT}(M):=(\mathcal{O}_{\widehat{\mathcal{E}^{ur}}}\otimes_{\mathcal{O}_{\mathcal{F}}} M)^{\varphi_q\otimes\varphi_M= id}.
\end{equation}
	\begin{theorem}\label{False-Tate equivalence}
The functor $\mathbb{D}_{LT,FT}$ defines an exact equivalence of categories between ${\bf Rep}_{\mathcal{O}_L}(G_K)$\\ $\ (\text{resp.,}\, {\bf Rep}_{\mathcal{O}_L-tor}(G_K))$ and ${\bf Mod}^{\varphi_q,\Gamma_{LT,FT},\acute{e}t}_{/\mathcal{O}_{\mathcal{F}}}$ $(\text{resp.,}\,{\bf Mod}^{\varphi_q,\Gamma_{LT,FT},\acute{e}t,tor}_{/\mathcal{O}_{\mathcal{F}}})$ with a quasi-inverse functor $\mathbb{V}_{LT,FT}$. %%%%%%%The functor $V\mapsto \mathbb{D}_{LT,FT}(V)$ defines an equivalence of the categories between ${\bf Rep}_{\mathcal{O}_K}(G_K)$ and ${\bf Mod}^{\varphi_q,\Gamma_{LT,FT},\acute{e}t}_{/\mathcal{O}_{\mathcal{L}}}$ with quasi-inverse $\mathbb{V}_{LT,FT}$.  
	\end{theorem}
	\begin{proof}
Since the functor $\mathbb{D}_{LT,FT}$ is composite of the functor $\mathbb{D}_{LT}$ with the scalar extension $\otimes_{\mathcal{O}_{\mathcal{E}}} \mathcal{O}_{\mathcal{F}}$, the proof follows from Proposition \ref{composite functor} and Theorem \ref{Kisin Ren}.
	\end{proof}
	\begin{remark}\label{choice} {\C Some other examples of %%The extension $\widetilde{K}$ is not the canonical one. We can also define 
	$\widetilde{K}_\infty$ are the following:} %%as follows:  
		\begin{enumerate} 
			\item Define $K_{cyc}:= K(\mu_{p^n})_{n\geq 1}$. Let $K_{cyc}\subseteq K_{\infty}$ and $\widetilde{K}:= K(\pi^{p^{-r}},r\geq 1)$, then $\widetilde{K}_\infty= K_\infty \widetilde{K}$ is a Galois extension of $K$ and $\Gal(\widetilde{K}_\infty/K_\infty)\cong \mathbb{Z}_p$. The case when $K_{cyc}= K_{\infty}$ has been considered in \cite{Flo} and \cite{LH2}. 
			\item We can also define $\widetilde{K}:=K(y_i)_{i\geq 1}$, where $(y_i)_{i\geq 1}$ is a system satisfying $[\pi_L](y_1)= y$ and $[\pi_L](y_{i+1})= y_i$ for all $i\geq 1$ and $y \in \mathfrak{m}_K\backslash \mathfrak{m}_K^2$. In this case, $\Gal(\widetilde{K}_\infty/K_\infty)$ is isomorphic to {\col a subgroup of $\mathcal{O}_L$. Indeed, we have an injective map $\Gal(K_\infty(y_i)/K_\infty)\hookrightarrow \mathcal{G}[\pi_L^i]\cong \mathcal{O}_L/\pi_L^i\mathcal{O}_L$, and $\Gal(\widetilde{K}_\infty/K_\infty)$ is the projective limit of $\Gal(K_\infty(y_i)/K_\infty)$. }
		\end{enumerate}
	\end{remark}
Then using similar methods as explained in section \ref{sec3}, we extend the functor $\mathbb{D}_{LT,FT}$ to the category of discrete $\pi_L$-primary abelian groups with a continuous action of $G_K$. Then %%%%the functor $\mathbb{D}_{LT,FT}$ is an exact equivalence of categories from the category ${\bf Rep}_{\mathcal{O}_K-tor}^{dis}(G_K)$ of discrete $\pi$-primary representations of $G_K$ to the category $\varinjlim {\bf Mod}^{\varphi_q,\Gamma_{LT,FT},\acute{e}t,tor}_{/\mathcal{O}_{\mathcal{L}}}$ of injective limits of $\pi$-power torsion objects in ${\bf Mod}^{\varphi_q,\Gamma_{LT,FT},\acute{e}t}_{/\mathcal{O}_{\mathcal{L}}}$, i.e., 
we have the following result. %%%%we extend this equivalence of categories from the category ${\bf Rep}_{\mathcal{O}_K-tor}^{dis}(G_K)$ of discrete $\pi$-primary abelian groups with a continuous action of $G_K$ to the category $\varinjlim {\bf Mod}^{\varphi_q,\Gamma_{LT,FT},\acute{e}t,tor}_{/\mathcal{O}_{\mathcal{L}}}$ of injective limits of $\pi$-torsion objects in ${\bf Mod}^{\varphi_q,\Gamma_{LT,FT},\acute{e}t}_{/\mathcal{O}_{\mathcal{L}}}$. Thus we have the following result.
	\begin{theorem}
The functors $\mathbb{D}_{LT,FT}$ and $\mathbb{V}_{LT,FT}$ are {\C exact and quasi-inverse and induce an }%%% the quasi-inverse 
equivalence of categories between the category $ {\bf Rep}_{\mathcal{O}_L-tor}^{dis}(G_K)$ and $\varinjlim {\bf Mod}^{\varphi_q,\Gamma_{LT,FT},\acute{e}t,tor}_{/\mathcal{O}_{\mathcal{F}}}$.
	\end{theorem}
\begin{proof}
Since the functor $\mathbb{D}_{LT}$ commutes with direct limits, it follows from Proposition \ref{composite functor} that the functor $\mathbb{D}_{LT,FT}$ also commutes with direct limits. Now the result follows from Theorem \ref{False-Tate equivalence} by taking direct limits and noting that the functor $\mathbb{V}_{LT,FT}$ also commutes with direct limits.
\end{proof}
Note that $H_F\cong \Gal(E^{sep}/E_F)$, then it follows that $H^i(H_F,\mathcal{O}_{\widehat{\mathcal{E}^{ur}}}/\pi_L^n )=0$ for all $n\geq 1$ and $i\geq 1$. Moreover, for any $V \in {\bf Rep}_{\mathcal{O}_L-tor}^{dis}(G_K)$, we have $\mathcal{H}^i(\Phi^{\bullet}(\mathbb{D}_{LT,FT}(V)))\cong H^i(H_F, V)$ as $\Gamma_{LT,FT}$-modules. %%%% In other words, the complex $\Phi^{\bullet}(\mathbb{D}_{LT}(V))$ computes the $H_K$-cohomology of $V$. 

%%%%Now suppose that $K$ contains the $\pi$-torsion points of the Lubin-Tate group $\mathcal{F}$. Recall that 
%%%{\C Assume that }$K$ contains the $\pi$-torsion points of the Lubin-Tate group $\mathcal{F}$ {\C defined over $K$. By this assumption, the group $\Gamma_{LT}$ is torsion free.} %%and %%%

Let $\Gamma_{LT}^*= \langle\gamma_1,\gamma_2,\ldots,\gamma_d\rangle$ as an abelian group. %% a $\mathbb{Z}_p$-module. 
Let $\widetilde{\gamma}$ be a topological generator of $\Gal(F/K_\infty)$. {\C Since $\Gal(F/K_\infty)$ is normal in $\Gal(F/K)$, we have $\Gamma_{LT,FT}=\Gamma_{LT}^*\ltimes \mathbb{Z}_p$.} We lift $\gamma_1,\gamma_2,\ldots,\gamma_d$ to the elements of $\Gal(F/\widetilde{L})$. Then $\Gamma_{LT,FT}$ %%:=\Gamma_{LT}^*\ltimes \mathbb{Z}_p$ 
is topologically generated by the set $\widetilde{\mathfrak{X}}:=\{\gamma_1,\gamma_2, \ldots,\gamma_d,\widetilde{\gamma}\}$. {\C The $\Gamma_{LT}^*$-action on $\Gal(F/K_\infty)$ by conjugation gives} the relations $\gamma_i\widetilde{\gamma}= \widetilde{\gamma}^{a_i}\gamma_i$ such that $a_i\in \mathbb{Z}_p^\times$. %% where $a_i= \chi_{LT}(\gamma_i)$ for all $i=1,\ldots,d$. 
This gives a character $\chi_{LT}'$ of $\Gamma_{LT}^*$ with $a_i= \chi_{LT}'(\gamma_i)$ for all $i=1,\ldots,d$. %%%is the Lubin-Tate character. 

Let $A$ be an arbitrary representation of the group $\Gamma_{LT,FT}$. Then consider the complex
	\begin{equation*}
	\Gamma_{LT,FT}^{\bullet}(A): 0\rightarrow A \rightarrow \bigoplus_{i_1\in \widetilde{\mathfrak{X}}} A\rightarrow\cdots\rightarrow \bigoplus_{\{i_1,\ldots,i_r\} \in \binom{\widetilde{\mathfrak{X}}}{r}} A\rightarrow\cdots \rightarrow A\rightarrow 0,
	\end{equation*}
	where for all $0\leq r\leq |\widetilde{\mathfrak{X}}|-1$, the map $d_{i_1,\ldots,i_r}^{j_1,\ldots, j_{r+1}}:A\rightarrow A$ from the component in the $r$-th term corresponding to $\{i_1,\ldots,i_r\}$ to the component corresponding to the $(r+1)$-tuple $\{j_1,\ldots,j_{r+1}\}$ is given by
	\begin{equation*}
	d_{i_1,\ldots,i_r}^{j_1,\ldots, j_{r+1}} =
	\left\{
	\begin{array}{ll}
	0  & \mbox{if } \{i_1,\ldots,i_r\}\nsubseteq\{j_1,\ldots,j_{r+1}\}, \\
	(-1)^{s_j}(\gamma_j-id) & \mbox{if } \{j_1,\ldots,j_{r+1}\}= \{i_1,\ldots,i_r\}\cup\{j\} \\ & \mbox {and} \{i_1,\ldots,i_{r}\}\ \mbox{doesn't contain}\ \widetilde{\gamma},
	\\
	(-1)^{s_j+1}\left(\gamma_j-\frac{\widetilde{\gamma}^{\chi_{LT}'(j)\chi_{LT}'(i_1)\cdots\chi_{LT}'(i_r)}-id}{\widetilde{\gamma}^{\chi_{LT}'(i_1)\cdots\chi_{LT}'(i_r)}-id}\right) & \mbox{if } \{j_1,\ldots,j_{r+1}\}= \{i_1,\ldots,i_r\}\cup\{j\} \\ & \mbox {and} \{i_1,\ldots,i_{r}\}\ \mbox{contains}\ \widetilde{\gamma},
	\\ 
	\widetilde{\gamma}^{\chi_{LT}'(i_1)\cdots\chi_{LT}'(i_r)}-id& \mbox{if } \{j_1,\ldots,j_{r+1}\}= \{i_1,\ldots,i_r\}\cup\{\widetilde{\gamma}\},
	\end{array}
	\right.
	\end{equation*}
	where $s_j$ is the number of elements in the set $\{i_1,\ldots,i_r\}$, which are smaller than $j$. 
	
%%%%	We explicitly write the complex $\Gamma_{LT,FT}^\bullet(A)$ in the case of $d=2$. 
	\begin{example}{\label{Ex2}}
		Let $d=2$, then the complex $\Gamma_{LT,FT}^\bullet(A)$ is defined as follows:
		\begin{equation*}
		\Gamma_{LT,FT}^\bullet(A): 0\rightarrow A\xrightarrow{x\mapsto A_0x} A^{\oplus3}\xrightarrow{x \mapsto A_1x} A^{\oplus3}\xrightarrow{x\mapsto A_2x}A\rightarrow 0,
		\end{equation*}
		where 
		\[
		A_0=
		\begin{bmatrix}
		\gamma_1-id\\
		\gamma_2-id \\
		\widetilde{\gamma}-id
		\end{bmatrix},
		A_1=
		\begin{bmatrix}
		-(\gamma_2-id) & \gamma_1-id & 0\\
		\widetilde{\gamma}^{a_1}-id & 0 & -\left(\gamma_1- \frac{\widetilde{\gamma}^{a_1}-id}{\widetilde{\gamma}-id}\right)\\
		0 & \widetilde{\gamma}^{a_2}-id & -\left(\gamma_2-\frac{\widetilde{\gamma}^{a_2}-id}{\widetilde{\gamma}-id}\right)
		\end{bmatrix},
		\]
		\[
		A_2=
		\begin{bmatrix}
		\widetilde{\gamma}^{a_1a_2}-id & \gamma_2-\frac{\widetilde{\gamma}^{a_1a_2}-id}{\widetilde{\gamma}^{a_1}-id}) & -\left(\gamma_1-\frac{\widetilde{\gamma}^{a_1a_2}-id}{\widetilde{\gamma}a^{a_2}-id}\right)
		\end{bmatrix}.
		\]
	\end{example}
	The functor $A\mapsto \mathcal{H}^i(\Gamma_{LT,FT}^\bullet(A))_{i\geq 0}$ forms a cohomological $\delta$-functor.  Note that the complex $\Gamma_{LT,FT}^\bullet(A)$ is the total complex of the following double complex \begin{center}
	%	\begin{tikzpicture}[baseline= (a).base]
	%	\node[scale=0.90] (a) at (0,0){
			\begin{tikzcd}[row sep=large]
			& 0  & 0  &  & 0  &  & 0  \\ \Gamma_{LT,FT \backslash \gamma_d}^{\bullet d}(A) : 0\arrow{r} & A \arrow{r}\arrow[swap]{u} & \underset{i_1\in \widetilde{\mathfrak{X}}^\prime}\bigoplus A \arrow{r}\arrow[swap]{u} & \cdots \arrow{r} & \underset{\{i_1,\ldots,i_r\}\in \binom{\widetilde{\mathfrak{X}}^\prime}{r}}\bigoplus A\arrow{r}\arrow[swap]{u} & \cdots \arrow{r}& A \arrow{r}\arrow[swap]{u} & 0\\ \Gamma_{LT,FT\backslash \gamma_d}^\bullet(A):0\arrow{r} & A \arrow{r}\arrow[swap]{u}[swap]{\gamma_d-id} & \underset{i_1\in \widetilde{\mathfrak{X}}^\prime}\bigoplus A\arrow{r}\arrow[swap]{u} & \cdots \arrow{r} & \underset{\{i_1,\ldots,i_r\}\in \binom{\widetilde{\mathfrak{X}}^\prime}{r}}\bigoplus A\arrow{r}\arrow[swap]{u} & \cdots \arrow{r}& A \arrow{r}\arrow[swap]{u} & 0\\  & 0 \arrow[swap]{u} & 0 \arrow[swap]{u} &  & 0\arrow[swap]{u} &  & 0 \arrow[swap]{u}
			\end{tikzcd}
	%	};
	%	\end{tikzpicture}
	\end{center}
	where $\widetilde{\mathfrak{X}}^{\prime} = \{\gamma_1,\ldots,\gamma_{d-1},\widetilde{\gamma}\}$ and $\Gamma_{LT,FT \backslash \gamma_d}^{\bullet d}(A)$ denotes the complex $\Gamma_{LT,FT \backslash \gamma_d}^{\bullet}(A)$ with $\widetilde{\gamma}$ replaced by $\widetilde{\gamma}^{a_d}$. The vertical maps $d_{b_1,\ldots,b_r}^{c_1,\ldots,c_r}: A\rightarrow A$ from the component in the $r$-th term corresponding to $\{b_1,\ldots,b_r\}$ to the component corresponding to $r$-th component $\{c_1,\ldots,c_r\}$ are given by the following
	\begin{equation*}
	d_{b_1,\ldots,b_r}^{c_1,\ldots,c_r}=
	\left\{
	\begin{array}{ll}
	\gamma_d-id  & \mbox{if } \{b_1,\ldots,b_r\} \mbox{ doesn't contain any term}\\& \mbox{ of the form } (\widetilde{\gamma}-id), \\
	\dfrac{(\widetilde{\gamma}^{a_d\chi_{LT}'(b_1)\cdots\chi_{LT}'(b_r)}-id)(\gamma_d-id)}{\widetilde{\gamma}^{\chi_{LT}'(b_1)\cdots\chi_{LT}'(b_r)}-id} & \mbox{if } \{b_1,\ldots,b_{r}\} \mbox{ contains a term of the form } \\ & (\widetilde{\gamma}^{\chi_{LT}'(b_1)\cdots\chi_{LT}'(b_r)}-id).
	\end{array}
	\right.
	\end{equation*}
	Then using the similar technique as in the proof of Proposition \ref{Gamma^*-cohomo}, it follows that %%%%Moreover, for a discrete $\pi$-primary abelian group $A$ with continuous action of $\Gamma_{LT,FT}$, 
	the complex $\Gamma_{LT,FT}^\bullet(A)$ computes the $\Gamma_{LT,FT}$-cohomology of $A$ and $\mathcal{H}^0(\Gamma_{LT,FT}^\bullet(A))= A^{\Gamma_{LT,FT}}$, for a discrete $\pi_L$-primary abelian group $A$ with continuous action of $\Gamma_{LT,FT}$. %%%%%% The proof is similar to as that of Proposition \ref{Gamma^*-cohomo}.
	\begin{definition} \label{FTHC}
	Let $M \in \varinjlim{\bf Mod}^{\varphi_q,\Gamma_{LT,FT},\acute{e}t,tor}_{/\mathcal{O}_{\mathcal{F}}}$. We define the co-chain complex  $\Phi\Gamma_{LT,FT}^{\bullet}(M)$ as the total complex of the double complex $\Gamma_{LT,FT}^\bullet(\Phi^\bullet(M))$ and call it the \emph{False-Tate type Herr complex} for $M$.
	\end{definition}
	\begin{example}{\label{Ex3}}
		In the case of $d=2$, the False-Tate type Herr complex is defined as follows:
		\begin{equation*}
		0\rightarrow M\xrightarrow{x\mapsto A_{0,\varphi_q}x}M^{\oplus 4}\xrightarrow{x\mapsto A_{1,\varphi_q}x} M^{\oplus 6}\xrightarrow{x\mapsto A_{2,\varphi_q}x}M^{\oplus4}\xrightarrow{x\mapsto A_{3,\varphi_q}x}M\rightarrow 0,
		\end{equation*}
		where %%%%$M=M^\Delta$, and
{\C	\setlength{\arraycolsep}{10pt}
		\renewcommand{\arraystretch}{2.5}
		\[
		A_{0,\varphi_q}=
		\begin{bmatrix}
		\varphi_M-id \\%[1em]
		\gamma_1-id\\%[1em]
		\gamma_2-id \\%[1em]
		\widetilde{\gamma}-id 
		\end{bmatrix},
		A_{1,\varphi_q}=
		\begin{bmatrix}
		-(\gamma_1-id) &  \varphi_M-id & 0 & 0\\%[1em]
		-(\gamma_2-id) & 0 &  \varphi_M-id & 0\\
		-(\widetilde{\gamma}-id) & 0 &  0 & \varphi_M-id\\
		0 & -(\gamma_2-id) & \gamma_1-id & 0\\
		0 & \widetilde{\gamma}^{a_1}-id & 0 & -\left(\gamma_1-\frac{\widetilde{\gamma}^{a_1}-id}{\widetilde{\gamma}-id}\right)\\%[1em]
		0 & 0 & \widetilde{\gamma}^{a_2}-id & -\left(\gamma_2-\frac{\widetilde{\gamma}^{a_2}-id}{\widetilde{\gamma}-id}\right)
		\end{bmatrix},
		\]
		\[
		A_{2,\varphi_q}=
		\begin{bmatrix}
		\gamma_2-id & -(\gamma_1-id) & 0 & \varphi_M-id & 0 & 0 \\
		-(\widetilde{\gamma}^{a_1}-id) & 0 & \gamma_1-\frac{\widetilde{\gamma}^{a_1}-id}{\widetilde{\gamma}-id} & 0 & \varphi_M-id & 0\\
		0 & -(\widetilde{\gamma}^{a_2}-id)  & \gamma_2-\frac{\widetilde{\gamma}^{a_2}-id}{\widetilde{\gamma}-id} & 0 & 0 & \varphi_M-id\\
		0 & 0 & 0 & \widetilde{\gamma}^{a_1a_2}-id & \gamma_2-\frac{\widetilde{\gamma}^{a_1a_2}-id}{\widetilde{\gamma}^{ a_1}-id} & -\left(\gamma_1-\frac{\widetilde{\gamma}^{a_1a_2}-id}{\widetilde{\gamma}^{ a_2}-id}\right)
		\end{bmatrix},
		\]
		\[
		A_{3,\varphi_q}=
		\begin{bmatrix}
		-(\widetilde{\gamma}^{a_1a_2}-id) & -\left(\gamma_2-\frac{\widetilde{\gamma}^{a_1a_2}-id}{\widetilde{\gamma}^{ a_1}-id}\right) & \gamma_1-\frac{\widetilde{\gamma}^{a_1a_2}-id}{\widetilde{\gamma}^{a_2}-id}  & \varphi_M-id \\
		\end{bmatrix}.
		\]
	}
	\end{example}
	Now the cohomology of the complex $\Phi\Gamma_{LT,FT}^{\bullet}(-)$ gives the cohomological functors $(\mathcal{H}^i(\Phi\Gamma_{LT,FT}^{\bullet}(-)))_{i\geq 0}$ from $\varinjlim{\bf Mod}^{\varphi_q,\Gamma_{LT,FT},\acute{e}t,tor}_{/\mathcal{O}_{\mathcal{F}}}$ to the category of abelian groups. Then we have the following theorem.
	\begin{theorem}\label{Main4}
		For any $V \in {\bf Rep}_{\mathcal{O}_L-tor}^{dis}(G_K)$, we have a natural isomorphism
		\begin{equation*}
		H^i(G_K,V)\cong \mathcal{H}^i(\Phi\Gamma_{LT,FT}^{\bullet}(\mathbb{D}_{LT,FT}(V)))\quad \text{for}\ i\geq 0.
		\end{equation*}
	\end{theorem}
	\begin{proof}
Note that the functor $V\mapsto\mathcal{H}^i(\Phi\Gamma_{LT,FT}^{\bullet}(\mathbb{D}_{LT,FT}(V)))_{i\geq 0}$ is a cohomological $\delta$-functor such that $\mathcal{H}^0(\Phi\Gamma_{LT,FT}^{\bullet}(\mathbb{D}_{LT,FT}(V)))\cong H^0(G_K,V)$. Hence the proof follows as in the proof of Theorem \ref{G_K-cohomo}.
	\end{proof} 
	\begin{corollary}
Let $V \in {\bf Rep}_{\mathcal{O}_L}(G_K)$. Then the False-Tate type Herr complex computes the Galois cohomology of $G_K$ with coefficients in $V$. %%%% we have $H^i(G_K,V)\cong \mathcal{H}^i(\Phi\Gamma_{LT,FT}^{\bullet}(\mathbb{D}_{LT,FT}(V)))$. 
	\end{corollary}
	\begin{proof}
		The proof is similar to that of Theorem \ref{lattices}.
	\end{proof}
	\section{The Operator $\psi_q$} \label{section psi}
{\C In this section, we define an integral operator $\psi_q$ following \cite{SV}, which acts linearly on the \'{e}tale $(\varphi_q,\Gamma_{LT})$-module $\mathbb{D}_{LT}(V)$.} %% Then we show that under some conditions on \'{e}tale $(\varphi_q,\Gamma_{LT})$-module, the Lubin-Tate Herr complex for $\varphi_q$ and $\psi_q$ are quasi-isomorphic (see Theorem \ref{Main5} and Remark \ref{divisible modules}). We also prove similar results for False-Tate type Herr complex. 
Recall that the residue field of $\mathcal{O}_{\mathcal{E}}$ is $E=k((X))$, which is not perfect, so $\varphi_q$ is not an automorphism but is injective. The field $\widehat{\mathcal{E}^{ur}}$, which is the fraction field of $\mathcal{O}_{\widehat{\mathcal{E}^{ur}}}$, is an extension of degree $q$ of $\varphi_{q}(\widehat{\mathcal{E}^{ur}})$. Put $\trac = \trace_{\widehat{\mathcal{E}^{ur}}/\varphi_q(\widehat{\mathcal{E}^{ur}})}$. Define
	\begin{equation*}
	\psi_q:\widehat{\mathcal{E}^{ur}} \rightarrow \widehat{\mathcal{E}^{ur}}
	\end{equation*}
such that
	\begin{equation*}
\varphi_q(\psi_q(x))=\frac{1}{\pi}(\trac(x)).
	\end{equation*}
	The map $\psi_q$ maps $\mathcal{O}_{\widehat{\mathcal{E}^{ur}}}$ to $\mathcal{O}_{\widehat{\mathcal{E}^{ur}}} $ and $\mathcal{O}_{\mathcal{E}}$ to $\mathcal{O}_{\mathcal{E}}$ \cite[Remark $3.1$]{SV}. This follows from the fact that the residue extensions $E^{sep}/\varphi_q(E^{sep})$ and $E/\varphi_q(E)$ are totally inseparable. Therefore the trace map defined by 
	\begin{equation*}
	\trac(x) = \text{trace}_{\widehat{\mathcal{E}^{ur}}/\varphi_q(\widehat{\mathcal{E}^{ur}})}(x) = \text{trace}_{\varphi_q(\widehat{\mathcal{{E}}^{ur}})}(y \mapsto xy)
	\end{equation*}
	is trivial for {\C $E^{sep}/\varphi_q(E^{sep})$ and $E/\varphi_q(E)$. }%% these extensions. 
	Hence if $x \in \mathcal{O}_{\widehat{\mathcal{E}^{ur}}}$, then $\trac(x) \in \pi\mathcal{O}_{\widehat{\mathcal{E}^{ur}}}$. 
	Moreover,
	\begin{equation*}
	\trac_{\widehat{\mathcal{E}^{ur}}/\varphi_q(\widehat{\mathcal{E}^{ur}})}(\varphi_q(x))= q \varphi_q(x)
	\end{equation*}
	implies that
	\begin{equation*}
	\psi_q(\varphi_q(x)) = \frac{q}{\pi}(x).
	\end{equation*}
	Hence 
	\begin{equation*}
	\psi_q\circ\varphi_q = \frac{q}{\pi}id.
	\end{equation*}
	We may extend this map $\psi_q$ to $\mathcal{O}_{\widehat{\mathcal{E}^{ur}}}\otimes_{\mathcal{O}_L} V$ by trivial action on $V$. As $\varphi_q$ commutes with $\Gamma_{LT}$, the module $\varphi_q(\mathcal{E})$ is stable under $\Gamma_{LT}$, so $\gamma\circ \trac\circ \gamma^{-1} = \trac$ for all $\gamma \in \Gamma_{LT}$. This ensures that $\psi_q$ commutes with $\Gamma_{LT}$ and it is also stable under the action of $\Gamma_{LT}$. This induces an operator 
	\begin{equation*}
	\psi_{\mathbb{D}_{LT}(V)}: \mathbb{D}_{LT}(V)\rightarrow \mathbb{D}_{LT}(V)
	\end{equation*} 
satisfying	
	\begin{equation}\label{psi_D}
	\psi_{\mathbb{D}_{LT}(V)}\circ\varphi_{\mathbb{D}_{LT}(V)} = \frac{q}{\pi}id_{\mathbb{D}_{LT}(V)}.
	\end{equation}
	Similarly, we have $\psi_{\mathbb{D}_{LT,FT}(V)}: \mathbb{D}_{LT,FT}(V)\rightarrow \mathbb{D}_{LT,FT}(V)$ satisfying above properties as $\psi_q$ maps $\mathcal{O}_\mathcal{F}$ to itself.
	
{\C We define a complex $\Psi^\bullet(\mathbb{D}_{LT}(V))$ as the following
	\begin{equation}\label{psi-complex}
		\Psi^{\bullet}(\mathbb{D}_{LT}(V)): 0\rightarrow \mathbb{D}_{LT}(V)\xrightarrow{ \psi_{\mathbb{D}_{LT}(V)}-\frac{q}{\pi}id}\mathbb{D}_{LT}(V)\rightarrow 0.
	\end{equation}
}
In general for any $M\in \varinjlim {\bf Mod}^{\varphi_q,\Gamma_{LT},\acute{e}t,tor}_{/\mathcal{O}_{\mathcal{E}}}$, the complex
$\Psi^\bullet(M)$ is defined as the following
\begin{equation*}
	\Psi^{\bullet}(M): 0\rightarrow M\xrightarrow{ \psi_M-\frac{q}{\pi}id}M\rightarrow 0,
\end{equation*}
where $\psi_M$ is induced operator on $M$ as explained in \eqref{psi_D}.
Before we go to next subsection, we  prove the following lemma.
	\begin{lemma}\label{phi to psi morphism}
		Let $A$ be an abelian group. Consider the following complexes $\mathscr{C}_1$, $\mathscr{C}_2$ and $\mathscr{C}_3$ 
		\begin{equation*}
		\mathscr{C}_i: 0\rightarrow A\xrightarrow{d_i}A\rightarrow 0, \qquad \text{for}\ i=1,2
		\end{equation*}
		\begin{equation*}
		\mathscr{C}_3: 0\rightarrow A \xrightarrow{d_3}\bigoplus_{i_1\in \mathfrak{y}}A\xrightarrow{d_3}\cdots\rightarrow\bigoplus_{\{i_1,\ldots,i_r\}\in \binom{\mathfrak{y}}{r}}A\xrightarrow{d_3}\cdots \xrightarrow{d_3} A\rightarrow 0,
		\end{equation*}
		where $\mathfrak{y}$ is a finite set and $\binom{\mathfrak{y}}{r}$ denotes choosing $r$-indices at a time from the set $\mathfrak{y}$. Let $\Tot(\mathscr{C}_i\mathscr{C}_j)$ be the total complex of the double complex $\mathscr{C}_i\mathscr{C}_j$. Then a morphism from the complex $\mathscr{C}_1$ to $\mathscr{C}_2$, which commutes with $d_3$, induces a natural homomorphism between the cohomology groups 
		\begin{equation*}
		\mathcal{H}^i(\Tot(\mathscr{C}_1\mathscr{C}_3))\rightarrow\mathcal{H}^i(\Tot(\mathscr{C}_2\mathscr{C}_3)).
		\end{equation*}
	\end{lemma}
	\begin{proof}
		The morphism from $\mathscr{C}_1$ to $\mathscr{C}_2$ induces the following commutative diagram
		\begin{center}
			\begin{tikzcd}
			\mathscr{C}_1 :	0\arrow{r} & A \arrow{r}{d_1} \arrow[swap]{d}{\delta_1} & A \arrow{r} \arrow[swap]{d}[swap]{\delta_2} & 0\\ \mathscr{C}_2 :	0 \arrow{r} & A \arrow{r}[swap]{d_2} & A \arrow{r} & 0.
			\end{tikzcd}
		\end{center}
		This induces a morphism between the total complex $\Tot(\mathscr{C}_1\mathscr{C}_3)$ and $\Tot(\mathscr{C}_2\mathscr{C}_3)$, as follows:
		\begin{center}
			\begin{tikzpicture}[baseline= (a).base]
			\node[scale=.90] (a) at (0,0)
			{
			\begin{tikzcd}[row sep=large]
			\Tot(\mathscr{C}_1\mathscr{C}_3):	0\arrow{r} & A \arrow{r}{(d_1,d_3)} \arrow[swap]{d}{\delta_1} &  A \oplus \smashoperator[r]{\bigoplus _{i_1\in \mathfrak{y}}}A \arrow{r} \arrow[swap]{d}{\delta_2, \smashoperator[r]{\bigoplus_{i_1\in \mathfrak{y}}}\delta_1} & \cdots\arrow{r} & \smashoperator[r]{\bigoplus_{\{i_1,\ldots,i_{d-1}\}\in \binom{\mathfrak{y}}{d-1}}} A\oplus A \arrow{r}{d_3-d_1} \arrow[swap]{d}{\smashoperator[r]{\bigoplus_{\{i_1,\ldots,i_{d-1}\}\in \binom{\mathfrak{y}}{d-1}}}\delta_2,\delta_1}  & A \arrow{r} \arrow[swap]{d}[swap]{\delta_2} & 0\\
			\Tot(\mathscr{C}_2\mathscr{C}_3):	0 \arrow{r} & A \arrow{r}[swap]{(d_2,d_3)} & A\oplus\smashoperator[r]{\bigoplus_{i_1\in \mathfrak{y}}}A \arrow{r} & \cdots\arrow{r} & \smashoperator[r]{\bigoplus_{\{i_1,\ldots,i_{d-1}\}\in \binom{\mathfrak{y}}{d-1}}}A\oplus A \arrow{r}[swap]{d_3-d_2} & A \arrow{r} & 0.
			\end{tikzcd}
		};
	\end{tikzpicture}
		\end{center}
		As the morphism from $\mathscr{C}_1$ to $\mathscr{C}_2$ commutes with $d_3$, it is easy to check that each square is commutative in the above diagram, which induces a homomorphism 
		\begin{equation*}
		\mathcal{H}^i(\Tot(\mathscr{C}_1\mathscr{C}_3))\rightarrow\mathcal{H}^i(\Tot(\mathscr{C}_2\mathscr{C}_3)).
		\end{equation*}
	\end{proof}
\iffalse{
Recall that $D^{sep}=\mathcal{O}_{\widehat{\mathcal{E}^{ur}}}\otimes_{\mathcal{O}_K}V\cong \mathcal{O}_{\widehat{\mathcal{E}^{ur}}}\otimes_{\mathcal{O}_\mathcal{E}}\mathbb{D}_{LT}(V)$. Define a complex $\Psi^\bullet(D^{sep})$ as the following
\begin{equation*}
\Psi^{\bullet}(D^{sep}): 0\rightarrow D^{sep}\xrightarrow{\psi_q\otimes \psi_{\mathbb{D}_{LT}(V)}-\frac{q}{\pi}id}D^{sep}\rightarrow 0,
\end{equation*}
{\C where $\psi_q\otimes \psi_{\mathbb{D}_{LT}(V)}$ denotes the extension of the $\psi_q$-operator as mentioned above. }

{\C Note that the map $\psi_q\otimes \psi_{\mathbb{D}_{LT}(V)}$ is not the tensor product of two maps. In fact, 
\begin{equation*}
\psi_q\otimes\psi_{\mathbb{D}_{LT}(V)}(x)=
			\left\{
			\begin{array}{ll}
				\psi_q(x)  & \mbox{if }  x\in \mathcal{O}_{\widehat{\mathcal{E}^{ur}}}, \\
			\psi_{\mathbb{D}_{LT}(V)}(x)	& \mbox{if }  x \in \mathbb{D}_{LT}(V) .
			\end{array}
	\right.
	\end{equation*} 
}}
\fi	
%%%{\C where $(\psi_q\otimes \psi_{\mathbb{D}_{LT}(V)})(a\otimes x)= \psi_q(a)\otimes \psi_{\mathbb{D}_{LT}(V)}(x)$ for $a\in \mathcal{O}_{\widehat{\mathcal{E}^{ur}}}$ and $x \in \mathbb{D}_{LT}(V)$.}
	\subsection{The case of Lubin-Tate extensions}	
\begin{definition}
		For any $M\in \varinjlim {\bf Mod}^{\varphi_q,\Gamma_{LT},\acute{e}t,tor}_{/\mathcal{O}_{\mathcal{E}}}$, the co-chain complex $\Psi\Gamma_{LT}^{\bullet}(M)$ is defined as the total complex of the double complex $\Gamma_{LT}^\bullet(\Psi^\bullet(M^{\Delta}))$. We write $\psi_M$ for the $\psi_q$-operator on $M$.
	\end{definition}
	\begin{example}{\label{Ex4}}
		Let $d=2$, then the complex $\Psi\Gamma_{LT}^{\bullet}(M)$ is defined as follows
		\begin{equation*}
		0\rightarrow M\xrightarrow{x\mapsto A_{0,\psi_q}x}M^{\oplus 3}\xrightarrow{x\mapsto A_{1,\psi_q}x} M^{\oplus 3}\xrightarrow{x\mapsto A_{2,\psi_q}x}M\rightarrow 0,
		\end{equation*}
		where $M=M^\Delta$, and
		\[
		A_{0,\psi_q}=
		\begin{bmatrix}
		\psi_M-\frac{q}{\pi}id \\
		\gamma_1-id \\
		\gamma_2-id 
		\end{bmatrix},
		A_{1,\psi_q} = 
		\begin{bmatrix}
		-(\gamma_1-id) & \psi_M-\frac{q}{\pi}id & 0  \\
		-(\gamma_2-id) & 0 & \psi_M-\frac{q}{\pi}id \\
		0 & -(\gamma_2-id) & \gamma_1-id 
		\end{bmatrix}, 
		\]
		\[
		A_{2,\psi_q}=
		\begin{bmatrix}
		\gamma_2-id & -(\gamma_1-id)& \psi_M-\frac{q}{\pi}id  
		\end{bmatrix}.
		\]
	\end{example}
Next, we have the following proposition, which is an easy consequence of Lemma \ref{phi to psi morphism}.
	\begin{proposition}\label{Prop 5.4}
Let $M\in \varinjlim {\bf Mod}^{\varphi_q,\Gamma_{LT},\acute{e}t,tor}_{/\mathcal{O}_{\mathcal{E}}}$. Then the morphism $\Phi^{\bullet}(M)\rightarrow \Psi^{\bullet}(M)$, which is given by the following
	\begin{center}
		\begin{tikzcd}[row sep=large, column sep = large]
		\Phi^\bullet(M):0\arrow{r} & M \arrow{r}{\varphi_M-id} \arrow[swap]{d}{id} & M \arrow{r} \arrow[swap]{d}[swap]{-\psi_M} & 0\\
		\Psi^\bullet(M):0 \arrow{r} & M \arrow{r}[swap]{\psi_M-\frac{q}{\pi}id} & M \arrow{r} & 0,
		\end{tikzcd}
	\end{center}
	induces a morphism 
	\begin{equation*}
	\Phi\Gamma_{LT}^{\bullet}(M) \rightarrow \Psi\Gamma_{LT}^{\bullet}(M).
	\end{equation*}
	\end{proposition}
	\begin{proof}
		Since $\psi_q$ commutes with the action of $\Gamma_{LT}$, the proof follows easily from Lemma \ref{phi to psi morphism}, by taking $\mathscr{C}_1 = \Phi^{\bullet}(M^\Delta), \mathscr{C}_2 = \Psi^{\bullet}(M^\Delta)$ and $\mathscr{C}_3 = \Gamma_{LT}^{\bullet}(M^\Delta)$.
	\end{proof}
	\begin{example}
		Let $d=2$. Then the morphism  between $\Phi\Gamma_{LT}^{\bullet}(M)$ and $\Psi\Gamma_{LT}^{\bullet}(M)$ is given by the following: 
		\begin{center}
			\begin{tikzcd}
			\Phi\Gamma_{LT}^{\bullet}(M):	0\arrow{r} & M \arrow{r}{A_{0,\varphi_q}} \arrow[swap]{d}{id} &  M \oplus M\oplus M \arrow{r}{A_{1,\varphi_q}} \arrow[swap]{d}{\mathscr{F}} &  M \oplus M\oplus M \arrow{r}{A_{2,\varphi_q}} \arrow[swap]{d}{\mathscr{F}^{\prime}}  & M \arrow{r} \arrow[swap]{d}[swap]{-\psi_M} & 0\\
			\Psi\Gamma_{LT}^{\bullet}(M):	0 \arrow{r} & M \arrow{r}[swap]{A_{0,\psi_q}} & M\oplus M \oplus M \arrow{r}[swap]{A_{1,\psi_q}} &  M \oplus M\oplus M \arrow{r}[swap]{A_{2,\psi_q}} & M \arrow{r} & 0, 
			\end{tikzcd}
		\end{center}
		where $M=M^\Delta$, and
		\begin{align*}
		\mathscr{F}(x_1,x_2,x_3) = & (-\psi_M(x_1),x_2,x_3),\\  \mathscr{F}^\prime(x_1,x_2,x_3)=& (-\psi_M(x_1), -\psi_M(x_2), x_3),
		\end{align*} 
		and the maps $A_{i,\varphi_q}$ and $A_{i,\psi_q}$ are the same as defined in Example \ref{Ex1} and Example \ref{Ex4}.
\end{example}
\begin{theorem}\label{Main5}
	Let $M \in \varinjlim {\bf Mod}^{\varphi_q,\Gamma_{LT},\acute{e}t,tor}_{/\mathcal{O}_{\mathcal{E}}}$. Then we have a well-defined homomorphism 
		\begin{equation*}
		\mathcal{H}^i(\Phi\Gamma_{LT}^{\bullet}(M))\rightarrow \mathcal{H}^i(\Psi\Gamma_{LT}^{\bullet}(M)) \quad \text{for} \ i\geq0.
		\end{equation*}
		Further, the homomorphism $\mathcal{H}^0(\Phi\Gamma_{LT}^{\bullet}(M))\rightarrow \mathcal{H}^0(\Psi\Gamma_{LT}^{\bullet}(M))$ is injective.
	\end{theorem}
\begin{proof}
Since $(-\psi_M)(\varphi_M-id) = (\psi_M-\frac{q}{\pi}id)$, and $\psi_M$ commutes with $\Gamma_{LT}$. By using Lemma \ref{phi to psi morphism} and Proposition \ref{Prop 5.4}, we have a morphism $\Phi\Gamma_{LT}^{\bullet}(M) \rightarrow \Psi\Gamma_{LT}^{\bullet}(M)$ of co-chain complexes, which induces a well-defined homomorphism 
		\begin{equation*}
		\mathcal{H}^i(\Phi\Gamma_{LT}^{\bullet}(M))\rightarrow \mathcal{H}^i(\Psi\Gamma_{LT}^{\bullet}(M)) \quad \text{for} \ i\geq0.
		\end{equation*}	
		For the second part, let $\mathcal{K}$ be the kernel and $\mathcal{C}$ be the co-kernel of the morphism $\Phi\Gamma_{LT}^{\bullet}(M) \rightarrow \Psi\Gamma_{LT}^{\bullet}(M)$. Then the complex $\mathcal{K}$ and $\mathcal{C}$ are given as the following:
		\begin{equation*}
		\mathcal{K}: 0\rightarrow 0\rightarrow \Ker\psi_M\oplus\bigoplus_{i_1\in \mathfrak{X}}0\rightarrow\cdots\rightarrow \bigoplus_{\{i_1,\ldots,i_{d-1}\}\in \binom{\mathfrak{X}}{d-1}}\Ker\psi_M\oplus 0\rightarrow \Ker\psi_M\rightarrow 0,
		\end{equation*}
		\begin{equation*}
		\mathcal{C}: 0\rightarrow 0\rightarrow \co\psi_M\oplus\bigoplus_{i_1\in \mathfrak{X}}0\rightarrow\cdots\rightarrow \bigoplus_{\{i_1,\ldots,i_{d-1}\}\in \binom{\mathfrak{X}}{d-1}}\co\psi_M\oplus 0\rightarrow \co\psi_M\rightarrow 0.
		\end{equation*} 
		The morphisms of the complex $\mathcal{K}$ are the restrictions of that of the complex $\Phi\Gamma_{LT}^{\bullet}(M)$, and the morphisms of the complex $\mathcal{C}$ are induced from the complex $\Psi\Gamma_{LT}^{\bullet}(M)$. Then we have the exact sequence 
		\begin{equation*}
		0\rightarrow \mathcal{K}\rightarrow \Phi\Gamma_{LT}^{\bullet}(M) \rightarrow \Psi\Gamma_{LT}^{\bullet}(M)\rightarrow \mathcal{C}\rightarrow 0,
		\end{equation*}
		which gives us the following short exact sequences
		\begin{equation}\label{eq(1)}
		0\rightarrow \mathcal{K}\rightarrow \Phi\Gamma_{LT}^{\bullet}(M) \rightarrow \mathbb{I}\rightarrow 0,
		\end{equation}	
		\begin{equation}\label{eq(2)}
		0\rightarrow \mathbb{I}\rightarrow \Psi\Gamma_{LT}^{\bullet}(M)\rightarrow \mathcal{C}\rightarrow 0,
		\end{equation}
		where $\mathbb{I}$ is the image of $\Phi\Gamma_{LT}^{\bullet}(M)\rightarrow \Psi\Gamma_{LT}^{\bullet}(M)$. Note that $\mathcal{H}^0(\mathcal{C})=0$. Then by taking the long exact cohomology sequence of (\ref{eq(2)}), we have $\mathcal{H}^0(\mathbb{I})\cong \mathcal{H}^0(\Psi\Gamma_{LT}^{\bullet}(M))$. Also, we have a long exact sequence
		\begin{align*}
		0 \rightarrow \mathcal{H}^0(\mathcal{K})\rightarrow \mathcal{H}^0(\Phi\Gamma_{LT}^{\bullet}(M))&\rightarrow \mathcal{H}^0(\mathbb{I}) \rightarrow \mathcal{H}^1(\mathcal{K}) \rightarrow \mathcal{H}^1(\Phi\Gamma_{LT}^{\bullet}(M))\rightarrow \mathcal{H}^1(\mathbb{I})\\&\rightarrow \mathcal{H}^2(\mathcal{K})\rightarrow \mathcal{H}^2(\Phi\Gamma_{LT}^{\bullet}(M))\rightarrow \mathcal{H}^2(\mathbb{I})\rightarrow \mathcal{H}^3(\mathcal{K})\rightarrow 0 \rightarrow \cdots
		\end{align*}
		Finally, as $\mathcal{H}^0(\mathbb{I})\cong \mathcal{H}^0(\Psi\Gamma_{LT}^{\bullet}(M))$ and $\mathcal{H}^0(\mathcal{K})=0$, the homomorphism $$\mathcal{H}^0(\Phi\Gamma_{LT}^{\bullet}(M))\rightarrow \mathcal{H}^0(\Psi\Gamma_{LT}^{\bullet}(M))$$ is injective.
\end{proof} 
\begin{remark}\label{Herr theorem}
In the classical case when $V$ is a $\mathbb{Z}_p$-representation of {\C $G_{\mathbb{Q}_p}$ with $\mathbb{G}_m$ as the Lubin-Tate module, then }%%%$G_K$ and $K_\infty=K_{cyc}$,  
the action of $\gamma-id$ is bijective on $\Ker\psi=\mathbb{D}(V)^{\psi=0}$, where $\gamma$ is a topological generator of $\Gamma$ \cite[Theorem 3.8]{LH1} and the Herr complex for $\varphi$ and $\psi$ are quasi-isomorphic \cite[Proposition 4.1]{LH1}. {\C In the general Lubin-Tate case, we have the above theorem as a weak result.}
\end{remark}
\subsection{The case of False-Tate type extensions}
\begin{definition}
	Let $M\in \varinjlim {\bf Mod}^{\varphi_q,\Gamma_{LT,FT},\acute{e}t,tor}_{/\mathcal{O}_{\mathcal{F}}}$. Then \emph{the False-Tate type Herr complex $\Psi\Gamma_{LT,FT}^{\bullet}(M)$ corresponding to $\psi_q$} is defined as the total complex of the double complex $\Gamma_{LT,FT}^\bullet(\Psi^\bullet(M))$.
\end{definition}
Let $\mathscr{C}_1 = \Phi^{\bullet}(M), \mathscr{C}_2 = \Psi^{\bullet}(M)$ and $\mathscr{C}_3 = \Gamma_{LT,FT}^{\bullet}(M)$. Then by using Lemma \ref{phi to psi morphism}, we have a morphism
	\begin{equation*}
	\Phi\Gamma_{LT,FT}^\bullet(M)\rightarrow \Psi\Gamma_{LT,FT}^\bullet(M).
	\end{equation*} 
Next, we prove a result in the case of False-Tate type extensions, which is analogous to Theorem \ref{Main5}. Recall that $\Gamma_{LT,FT}$ is topologically generated by $\{\gamma_1,\ldots,\gamma_d,\widetilde{\gamma}\}$, and $a_i = \chi_{LT}'(\gamma_i)$.
\begin{theorem} \label{Theorem False Tate}
	Let $M\in \varinjlim {\bf Mod}^{\varphi_q,\Gamma_{LT,FT},\acute{e}t,tor}_{/\mathcal{O}_{\mathcal{F}}}$. Then the morphism $$\Phi\Gamma_{LT,FT}^{\bullet}(M)\rightarrow \Psi\Gamma_{LT,FT}^{\bullet}(M)$$ induces a well-defined homomorphism $\mathcal{H}^i(\Phi\Gamma_{LT,FT}^{\bullet}(M))\rightarrow\mathcal{H}^i(\Psi\Gamma_{LT,FT}^{\bullet}(M))$ for $i\geq 0$. Moreover, we have $\mathcal{H}^0(\Phi\Gamma_{LT,FT}^{\bullet}(M))\hookrightarrow\mathcal{H}^0(\Psi\Gamma_{LT,FT}^{\bullet}(M))$.
\end{theorem}
\begin{proof}
As before, the fact that $(-\psi_M)(\varphi_M-id) = (\psi_M-\frac{q}{\pi}id)$, and $\psi_M$ commutes with the action of $\Gamma_{LT,FT}$. Then by using Lemma \ref{phi to psi morphism}, we have a morphism $\Phi\Gamma_{LT,FT}^{\bullet}(M) \rightarrow \Psi\Gamma_{LT,FT}^{\bullet}(M)$ of co-chain complexes, which induces a well-defined homomorphism 
	\begin{equation*}
	\mathcal{H}^i(\Phi\Gamma_{LT,FT}^{\bullet}(M))\rightarrow \mathcal{H}^i(\Psi\Gamma_{LT,FT}^{\bullet}(M))\quad \text{for} \ i\geq0.
	\end{equation*}	
Let $\mathcal{K}$ be the kernel and $\mathcal{C}$ the co-kernel of the morphism $\Phi\Gamma_{LT,FT}^{\bullet}(M) \rightarrow \Psi\Gamma_{LT,FT}^{\bullet}(M)$. Then we have an exact sequence 
\begin{equation*}
0\rightarrow \mathcal{K} \rightarrow \Phi\Gamma_{LT,FT}^{\bullet}(M) \rightarrow \Psi\Gamma_{LT,FT}^{\bullet}(M) \rightarrow \mathcal{C}\rightarrow 0.
\end{equation*}
The result then follows by using the same method as in Theorem \ref{Main5}. 
\end{proof}
%123
	\section{Iwasawa Cohomology over Lubin-Tate Extensions}\label{Iwasawa}
In this section, we briefly recall the results of Schneider and Venjakob on their computation of the Iwasawa cohomology. For any $M\in \varinjlim {\bf Mod}^{\varphi_q,\Gamma_{LT},\acute{e}t,tor}_{/\mathcal{O}_{\mathcal{E}}}$, consider the complex
	\begin{equation*}
	\underline{\Psi}^{\bullet}(M): 0\rightarrow M\xrightarrow{\psi_M-id}M\rightarrow 0.
	\end{equation*}
 Let $M\in \varinjlim {\bf Mod}^{\varphi_q,\Gamma_{LT},\acute{e}t,tor}_{/\mathcal{O}_{\mathcal{E}}}$, then $\pi_L^nM =0$ for some $n\geq 1$. Define
	\begin{equation*}
	M^{\vee}:= \text{Hom}_{\mathcal{O}_L}(M,L/\mathcal{O}_L), 
	\end{equation*}
	which can be identified with $\text{Hom}_{\mathcal{O}_{\mathcal{E}}}(M, \mathcal{O}_{\mathcal{E}}/\pi_L^n\mathcal{O}_{\mathcal{E}}(\chi_{LT}))$. For more details see (24) in \cite{SV}. Further, for an \'{e}tale $(\varphi_q,\Gamma_{LT})$-module $M$ such that $\pi_L^nM =0$, for some $n\geq 1$, we have $\Gamma_{LT}$-invariant continuous pairing defined as (\cite[Remark $4.7$]{SV}) 
	\begin{align*}
	\langle\cdot,\cdot\rangle:& M \times M^{\vee} \rightarrow L/\mathcal{O}_L\\
	&(m,F)\mapsto \pi^{-n}\ \Res(F(m)d\log_{LT}(\omega_{LT}))\ \text{mod}\mathcal{O}_L, 
	\end{align*}
	where $\Res$ is the residue map and $\omega_{LT}=\{\iota(v)\}$. This pairing induces the following pairing for any
	$M\in {\bf Mod}^{\varphi_q,\Gamma_{LT},\acute{e}t,tor}_{/\mathcal{O}_{\mathcal{E}}}$:
	\begin{equation*}
	\mathcal{H}^i(\Phi^\bullet(M^\vee))\times \mathcal{H}^{1-i}(\underline{\Psi}^{\bullet}(M))\rightarrow L/\mathcal{O}_L
	\end{equation*}
which is perfect.
\begin{remark}\label{Remark-duality}
Note that the functors $\mathcal{H}^i(\Phi^\bullet(-))$ and $\mathcal{H}^i(\underline{\Psi}^\bullet(-))$ commute with direct limits. Therefore the pairing exists for any $M\in \varinjlim {\bf Mod}^{\varphi_q,\Gamma_{LT},\acute{e}t,tor}_{/\mathcal{O}_{\mathcal{E}}}$. 
\end{remark}
\iffalse{
	This induces a perfect pairing of $\mathcal{O}_K$-modules
	\begin{equation*}
	\mathcal{H}^i(\Phi^\bullet(M^\vee))\times \mathcal{H}^{1-i}(\underline{\Psi}^{\bullet}(M))\rightarrow K/\mathcal{O}_K.
	\end{equation*}
	In other words, we have a canonical isomorphism
	\begin{equation}\label{duality}
	\mathcal{H}^i(\Phi^\bullet(M^\vee))^{\vee} \cong \mathcal{H}^{1-i}(\underline{\Psi}^{\bullet}(M)).
	\end{equation}
}\fi
	Let $V \in {\bf Rep}_{\mathcal{O}_L}(G_K)$, and we define the Iwasawa cohomology over $K_{\infty}$ by
	\begin{equation*}
	H^i_{Iw}(K_{\infty}/K,V):= \varprojlim H^i(F,V),
	\end{equation*}
	where $F$ varies over the finite Galois extensions of $K$ contained in $K_{\infty}$, and the projective limit is taken with respect to the cohomological co-restriction maps.	
We now recall the following theorem of \cite[Theorem 5.13]{SV}, which gives a description of the Iwasawa cohomology groups. {\C We express the theorem in terms of the $\underline{\Psi}^\bullet$-complex. We will use this formulation later in stating Theorem \ref{Main7.3}.} %%%in terms of $\underline{\Psi}^\bullet$ complex. {\C We reprove this theorem as we want to get their result in terms of $\underline{\Psi}^\bullet$-complex.}
%We express the proof of this theorem in terms of complexes.	
\begin{theorem}\!\textup{\cite[Theorem 5.13]{SV}}\label{Iwasawa cohomology}
Let $V \in {\bf Rep}_{\mathcal{O}_L-tor}^{dis}(G_K) $. Then the complex 
		\begin{equation*}
		\underline{\Psi}^{\bullet}(\mathbb{D}_{LT}(V(\chi_{cyc}^{-1}\chi_{LT}))): 0\rightarrow \mathbb{D}_{LT}(V(\chi_{cyc}^{-1}\chi_{LT}))\xrightarrow{\psi-id}\mathbb{D}_{LT}(V(\chi_{cyc}^{-1}\chi_{LT}))\rightarrow 0,
		\end{equation*} 
		where $\psi= \psi_{\mathbb{D}_{LT}(V(\chi_{cyc}^{-1}\chi_{LT}))}$, and $\chi_{cyc}$ is the cyclotomic character, computes the Iwasawa cohomology groups $H^i_{Iw}({K_{\infty}/K},V)$ for all $i\geq 1$. {\C In particular, $H^i_{Iw}({K_{\infty}/K},V)=0$ for $i\geq3$.} %%%%% i.e., $H^i_{Iw}({K_{\infty}/K},V)\cong \mathcal{H}^{i-1}(\underline{\Psi}^{\bullet}(\mathbb{D}_{LT}(V(\chi_{cyc}^{-1}\chi_{LT})))).$ 
	\end{theorem}
	\begin{proof}
		Since $V \in {\bf Rep}_{\mathcal{O}_L-tor}^{dis}(G_K)$, i.e., $V$ is a discrete $\pi_L$-primary representation of $G_K$, we have an isomorphism 
		\begin{equation*}
		H^i_{Iw}({K_{\infty}/K},V)\cong H^{2-i}(\Gal(\bar{K}/K_{\infty}),V^{\vee}(\chi_{cyc}))^{\vee},
		\end{equation*}
		which is induced from the Local Tate duality. For more details see \cite[Remark 5.11]{SV}. Moreover,
	\begin{align*}
	H^{2-i}(\Gal(\bar{K}/K_{\infty}),V^{\vee}(\chi_{cyc}))^{\vee} &= \mathcal{H}^{2-i}(\Phi^{\bullet}(\mathbb{D}_{LT}(V^\vee(\chi_{cyc}))))^{\vee}\\ &=\mathcal{H}^{2-i}(\Phi^\bullet(\mathbb{D}_{LT}(V^\vee)(\chi_{cyc})))^\vee \\&= \mathcal{H}^{2-i}(\Phi^\bullet(\mathbb{D}_{LT}(V)^\vee(\chi_{LT}^{-1}\chi_{cyc})))^\vee\\& = \mathcal{H}^{2-i}(\Phi^\bullet(\mathbb{D}_{LT}(V(\chi_{cyc}^{-1}\chi_{LT}))^\vee))^\vee\\&\cong \mathcal{H}^{i-1}(\underline{\Psi}^{\bullet}(\mathbb{D}_{LT}(V(\chi_{cyc}^{-1}\chi_{LT})))).
	\end{align*} 
	Here the first equality follows from Proposition \ref{H_K-cohomology}. The second and third equality uses Remark 4.6 and Remark 5.6 of \cite{SV}, respectively, while the last isomorphism comes from Remark \ref{Remark-duality}. Hence
	\begin{equation*}
	H^i_{Iw}({K_{\infty}/K},V)\cong \mathcal{H}^{i-1}(\underline{\Psi}^{\bullet}(\mathbb{D}_{LT}(V(\chi_{cyc}^{-1}\chi_{LT})))).
	\end{equation*}
{\C Now the second part follows as the cohomology groups for $\underline{\Psi}^\bullet$-complex are trivial in degree $\geq 2$.} This proves the theorem.
	\end{proof}

{\cv Putting $V_n=V/\pi^nV$ for $V\in{\bf Rep}_{\mathcal{O}_K}(G_K)$, we have $H^i(K,V)=\varprojlim_n H^i(K,V_n)$ and ${\varprojlim_n^1} H^i(K,V_n)=0$. Therefore,
\begin{align*}
		H^i_{Iw}({K_{\infty}/K},V)&=\varprojlim_{K'/K: \mbox{finite}} H^i(K',V)\\
		&=\varprojlim_{K'/K: \mbox{finite}}\varprojlim_n H^i(K',V_n)\\
		&=\varprojlim_n \varprojlim_{K'/K: \mbox{finite}} H^i(K',V_n)\\
		&=\varprojlim_n H^i_{Iw}({K_{\infty}/{K}},V_n).
\end{align*}
}
\begin{corollary}\!\textup{\cite[Theorem 5.13]{SV}}\label{Iwasawa cor.}
		For any $V \in {\bf Rep}_{\mathcal{O}_L}(G_K)$, we have
		\begin{equation*}
		H^i_{Iw}({K_{\infty}/K},V)\cong \mathcal{H}^{i-1}(\underline{\Psi}^{\bullet}(\mathbb{D}_{LT}(V(\chi_{cyc}^{-1}\chi_{LT})))) \quad\text{for}\ i\geq 1.
		\end{equation*}
%%% $$H^i_{Iw}({K_{\infty}/K},V)\cong \mathcal{H}^{i-1}(\underline{\Psi}^{\bullet}(\mathbb{D}_{LT}(V(\chi_{cyc}^{-1}\chi_{LT})))) \quad\text{for}\ i\geq 1.$$
	\end{corollary}
	\begin{proof}
Since the transition maps are surjective in the projective system $(\underline{\Psi}^\bullet(\mathbb{D}_{LT}(V/\pi_L^nV(\chi_{cyc}^{-1}\chi_{LT}))))_n$ of co-chain complexes of abelian groups, {\cv it satisfies the Mittag-Leffler condition, so} the first hyper-cohomology spectral sequence degenerates at $E_2$. {\cv Note that, for all $j\geq0$:
	\begin{equation*}
{\varprojlim_n}^{1} \mathcal{H}^j(\underline{\Psi}^\bullet(\mathbb{D}_{LT}(V/\pi^nV(\chi_{cyc}^{-1}\chi_{LT}))))={\varprojlim_n}^{1}H^{j+1}_{Iw}({K_{\infty}/K},V_n)=0.
	\end{equation*}
Here, $R^1\varprojlim$ of the projective systems $\{H^1_{Iw}({K_{\infty}/K},V_n)\}_n$ vanishes. Indeed, for the open pro-$p$ subgroup $\Gamma_{LT}^\ast$,  $R[[\Gamma_{LT}^\ast]]$ is a complete local noetherian commutative ring, and $\{H^1_{Iw}({K_{\infty}/K},V_n)\}_n$ is a projective system of finitely generated $R[[\Gamma_{LT}^\ast]]$-modules. Hence by \cite[Theorem 8.1]{Jen}, $R^1\varprojlim$ of this projective system vanishes.} 
%%%% Moreover, $\varprojlim\limits_n{}^{1} \mathcal{H}^i(\underline{\Psi}^\bullet(\mathbb{D}_{LT}(V/\pi_L^nV(\chi_{cyc}^{-1}\chi_{LT}))))=0$. 

The second hyper-cohomology spectral sequence 
		\begin{equation*}
		\varprojlim\limits_n{}^{i} \mathcal{H}^j(\underline{\Psi}^\bullet(\mathbb{D}_{LT}(V/\pi_L^nV(\chi_{cyc}^{-1}\chi_{LT})))) \Rightarrow \mathcal{H}^{i+j}(\underline{\Psi}^\bullet(\mathbb{D}_{LT}(V(\chi_{cyc}^{-1}\chi_{LT})))),
		\end{equation*} 
		also degenerates at $E_2$. Therefore $\varprojlim\limits_n \mathcal{H}^i(\underline{\Psi}^\bullet(\mathbb{D}_{LT}(V/\pi_L^nV(\chi_{cyc}^{-1}\chi_{LT}))))= \mathcal{H}^i(\underline{\Psi}^\bullet(\mathbb{D}_{LT}(V(\chi_{cyc}^{-1}\chi_{LT}))))$. Moreover, {\cv as seen above,} it follows from Lemma \ref{commutes inverse} that the functor $H^i_{Iw}({K_{\infty}/K},-)$ commutes with the inverse limits. Now the result follows from Theorem \ref{Iwasawa cohomology} by taking the inverse limits. %%%%% and Lemma \ref{commutes inverse}
	\end{proof}
	Next, we generalize most of our results to the case of any complete local Noetherian ring whose residue field is finite extension of $\mathbb{F}_p$. %%%% For this first, we recall some basic definitions and results from \cite{Dee}.
	
\section{An Equivalence of Categories over the Coefficient Ring}\label{sec6}
\subsection{Background on Coefficient Rings}\label{sec5}
In this section, we recall some basic results on coefficient rings. Recall that a \emph{coefficient ring} $R$ is a complete Noetherian local ring with finite residue field $k_R$ of characteristic $p$, i.e., $k_R$ is a finite extension of $\mathbb{F}_p$. Then $R$ has a natural pro-finite topology with a base of open ideals given by the powers of its maximal ideal $\mathfrak{m}_R$. In other words, $R = \varprojlim_n R/\mathfrak{m}_R^n R$. A \emph{coefficient ring homomorphism} is a continuous homomorphism of coefficient rings $R^{\prime}\rightarrow R$ such that the inverse image of the maximal ideal $\mathfrak{m}_R$ is the maximal ideal $\mathfrak{m}_{R^{\prime}}\subset R^{\prime}$ and the induced homomorphism on residue fields is an isomorphism. 
	
For a fixed prime number $p$, a \emph{$p$-ring} is a complete discrete valuation ring whose valuation ideal is generated by %%%a prime element 
{\C $\pi$, where $\pi$ is a prime lying above $p$}.
	
Let $R$ and $S$ be arbitrary rings, and $I\subset R, J\subset S$ be two ideals. Assume that $R$ and $S$ are both $T$-algebras for some third ring $T$. The \emph{completed tensor product} $R\hat{\otimes}_T S$ is defined as the completion of $R\otimes_T S$ with respect to the $(I\otimes S+R\otimes J)$-adic topology.
	
Let $\mathcal{O}$ be a $p$-ring and $R$ be a coefficient ring. Let $\mathcal{O}_L$ be a finite extension of $\mathbb{Z}_p$. Assume that both $\mathcal{O}$ and $R$ are $\mathcal{O}_L$-algebras, where the maps $\mathcal{O}_L\rightarrow \mathcal{O}$ and $\mathcal{O}_L\rightarrow R$ are local homomorphisms. Define
	\begin{equation*}
	\mathcal{O}_R:= \mathcal{O}\hat{\otimes}_{\mathcal{O}_L} R.
	\end{equation*}
Then by \cite[Proposition 1.2.3]{Dee}, $\mathcal{O}_R$ is a complete Noetherian semi-local ring. %%%%Note that the residue field need not be finite as there is no restriction on the residue field of $\mathcal{O}$.
	\begin{proposition}\label{Artinian}
		Let $A$ be a Noetherian semi-local commutative ring with unity and $\mathfrak{m}_A$ the radical (intersection of all maximal ideals) of $A$. Then $A/\mathfrak{m}_A^n$ is Artinian for all $n\geq 1$. 
	\end{proposition}
	\begin{proof}
		We prove this by using induction on $n$. Let $n=1$. Then by Chinese Remainder theorem, we have
		\begin{equation} \label{eq:11}
		A/\mathfrak{m}_A \cong \bigoplus_{i=1}^{l}A/\mathfrak{m}_i,
		\end{equation}
		where $\mathfrak{m}_i$ is a maximal ideal of $A$ for $0\leq i \leq l$, and the map is a natural projection map. Note that each $A/\mathfrak{m}_i$ is Artinian being a field. Consequently, the right hand side of (\ref{eq:11}) is Artinian as it is a finite direct sum of Artinian rings. Therefore $A/\mathfrak{m}_A$ is Artinian, and the result is true for $n=1$. Suppose that the result is true for $n-1$. For general $n$, the result follows from the exact sequence
		\begin{equation*}
		0\rightarrow \mathfrak{m}_A^{n-1}/\mathfrak{m}_A^n \rightarrow A/\mathfrak{m}_A^n\rightarrow A/\mathfrak{m}_A^{n-1}\rightarrow 0.
		\end{equation*}
		Now $A/\mathfrak{m}_A^{n-1}$ is Artinian by induction hypothesis. Since $A$ is Noetherian, $\mathfrak{m}_A^{n-1}/\mathfrak{m}_A^n$ is a finitely generated module over $A/\mathfrak{m}_A$. Together with the fact that every finitely generated module over an Artinian ring is Artinian, the module $\mathfrak{m}_A^{n-1}/\mathfrak{m}_A^n$ is Artinian.
		Hence $A/\mathfrak{m}_A^n$ is Artinian.
	\end{proof}
\begin{remark}
Since $\mathcal{O}_R$ is complete semi-local Noetherian ring with unity, by Proposition \ref{Artinian}, $\mathcal{O}_R/\mathfrak{m}_R^n\mathcal{O}_R$ is Artinian for all $n\geq 1$.
\end{remark}
Let $R$ and $S$ be two coefficient rings, and $\mathcal{O}$ be a $p$-ring (or indeed any local ring with residue field of characteristic $p$). Let $\theta : R \rightarrow S$ is a ring homomorphism, then it induces $\theta : \mathcal{O}\otimes_{\mathcal{O}_L}R\rightarrow \mathcal{O}\otimes_{\mathcal{O}_L}S$. Assume that $\theta$ is local. Then we have $\theta(\mathcal{O}\otimes\mathfrak{m}_R+\mathfrak{m}_{\mathcal{O}}\otimes R)\subset \mathcal{O}\otimes\mathfrak{m}_S+\mathfrak{m}_{\mathcal{O}}\otimes S$, and $\theta$ is continuous with respect to the obvious topologies. Therefore it induces a semi-local homomorphism 
	\begin{equation*}
	\theta: \mathcal{O}_R\rightarrow \mathcal{O}_S.
	\end{equation*}
%%%%	Next we recall \cite[Proposition $1.2.6$]{Dee}. For more details see \cite[$0.19.7.1.2$]{EGA} and \cite[Theorem $22.3$]{Mat}.
	\begin{proposition}\!\textup{\cite[Proposition 1.2.6]{Dee}}\label{faithfully flat}
		Let $\theta : \mathcal{O}_1\rightarrow\mathcal{O}_2$ be a local homomorphism of $p$-rings and let $R$ be a coefficient ring. If $\theta$ is flat, then it induces a faithfully flat homomorphism 
		\begin{equation*}
		\theta_R : \mathcal{O}_{1,R}\rightarrow \mathcal{O}_{2,R}.
		\end{equation*}  
	%%%%%	is faithfully flat.
	\end{proposition}
	
\subsection{An equivalence over coefficient rings}	
\subsubsection{The characteristic $p$ case} \label{sub6.1}
Let $E$ be a local field of characteristic $p>0$. Then $E\cong k_L((t))$, where $k_L$ is a finite extension of $\mathbb{F}_p$. %%% Assume that $\card(k_L)=q$, where $q=p^r$ for some fixed $r$. 
Let $\mathcal{O}_\mathcal{E}$ be {\C a} %%%the 
Cohen ring of $E$ with uniformizer $\pi_L$. Let $\mathcal{E}$ be the field of fractions of $\mathcal{O}_\mathcal{E}$. 
%%%%%\begin{equation*}
%%%%%%\mathcal{O}_\mathcal{E} = \varprojlim _{n \in\mathbb{N}}\mathcal{O}_\mathcal{E}/\pi^n\mathcal{O}_\mathcal{E},\ \mathcal{O}_\mathcal{E}/\pi\mathcal{O}_\mathcal{E}=E\ \text{and}\ \mathcal{E}=\mathcal{O}_\mathcal{E}[\frac{1}{\pi}].
%%%%\end{equation*}
The field $\mathcal{E}$ is a complete discrete valued field of characteristic $0$, whose residue field is $E$. We fix a choice of $\mathcal{E}$. Let $\mathcal{E}^{ur}$ be the maximal unramified extension of $\mathcal{E}$ with ring of integers $\mathcal{O}_{\mathcal{E}^{ur}}$. Clearly, $\mathcal{E}^{ur}$ is a Galois extension of $\mathcal{E}$, and there is an identification of Galois groups
\begin{equation*}
G_E=\Gal({E^{sep}/E})\xrightarrow{\sim} \Gal({\mathcal{E}^{ur}/\mathcal{E}}),
\end{equation*}
where $E^{sep}$ is the separable closure of $E$. The ring $\mathcal{O}_{\mathcal{E}^{ur}}$ has a valuation induced from $\mathcal{O}_{\mathcal{E}}$ and the valuation ring in the completion $\widehat{\mathcal{E}^{ur}}$ of $\mathcal{E}^{ur}$ {\cv (with respect to the valuation)} is a $p$-ring with residue field $E^{sep}$. We write ${\mathcal{O}}_{\widehat{\mathcal{E}^{ur}}}$ for this ring. The Galois group $G_E$ acts by continuity on $\widehat{\mathcal{E}^{ur}}$.

From now on, $R$ will always denote {\cv a} coefficient ring, unless stated otherwise. Also, we assume that $R$ is always an $\mathcal{O}_L$-algebra such that the map $\mathcal{O}_L\rightarrow R$ is a local ring homomorphism. Here $\mathcal{O}_L$ is the ring of integers of a $p$-adic field $L$ with residue field $k_L$ such that $\card(k_L)=q$, {\cv where $q=p^r$ for some fixed $r$}.\par Define the rings
	\begin{align*}
	&\mathcal{O}_R:= \mathcal{O}_{\mathcal{E}}\hat{\otimes}_{\mathcal{O}_L} R,\\&
	\widehat{\mathcal{O}^{ur}_R}:= \mathcal{O}_{\widehat{\mathcal{E}^{ur}}}\hat{\otimes}_{\mathcal{O}_L} R.
	\end{align*}
Then it follows from Proposition \ref{faithfully flat} that $\widehat{\mathcal{O}^{ur}_R}$ is an $\mathcal{O}_R$-algebra and is faithfully flat over $\mathcal{O}_R$. The action of $G_E$ on $\mathcal{O}_{\mathcal{E}^{ur}}$ induces an action on $\mathcal{O}_{\widehat{\mathcal{E}^{ur}}}$. Now by taking the trivial action of $G_E$ on $R$, it induces a Galois action on $\mathcal{O}_{\widehat{\mathcal{E}^{ur}}}\otimes_{\mathcal{O}_L} R$. Moreover, this action is continuous as $G_E$ acts continuously on $\mathcal{E}^{ur}$. Thus the action of $G_E$ on $\widehat{\mathcal{O}^{ur}_R}$ is continuous with respect to the $\mathfrak{m}_R\widehat{\mathcal{O}^{ur}_R}$-adic topology. %%%%%%and hence on $\mathcal{O}_{\widehat{\mathcal{E}^{ur}}}\otimes_{\mathcal{O}_K} R$, via the trivial action on $R$. Being continuous on $\mathcal{E}^{ur}$, this action is continuous on $\mathcal{O}_{\widehat{\mathcal{E}^{ur}}}\otimes_{\mathcal{O}_K} R$, so it induces a $G_E$ action on $\hat{\mathcal{O}}^{ur}_R$, continuous with respect to the $\mathfrak{m}_R\widehat{\mathcal{O}^{ur}_R}$-adic topology.
	\begin{remark}
It follows from \cite[Proposition 1.2.3]{Dee} that $\mathcal{O}_R$ and $\widehat{\mathcal{O}^{ur}_R}$ are Noetherian semi-local rings, complete with respect to the $\mathfrak{m}_R$-adic topology and that $\mathfrak{m}_R$ generates the radical of these rings. 
	\end{remark}
Let $\varphi_q:=(x\mapsto x^q)$ be the $q$-Frobenius on $E$. Choose a lift of $\varphi_q$ on $\mathcal{E}$ such that it maps $\mathcal{O}_{\mathcal{E}}$ to $\mathcal{O}_{\mathcal{E}}$. Then we have a ring homomorphism $\varphi_q: \mathcal{O}_{\mathcal{E}}\rightarrow \mathcal{O}_{\mathcal{E}}$ such that 
	\begin{equation*}
	\varphi_q(x) \equiv x^q \mod \pi.
	\end{equation*}
	Assume that $\varphi_q$ is flat. Then we have {\cv an} $R$-linear homomorphism
	\begin{equation*}
	\varphi_q:= \varphi_q\otimes id_R : \mathcal{O}_{\mathcal{E}}\otimes_{\mathcal{O}_L} R\rightarrow \mathcal{O}_{\mathcal{E}}\otimes_{\mathcal{O}_L} R.
	\end{equation*}
Since the ideal $\mathfrak{m}_{\mathcal{O}_{\mathcal{E}}}\otimes R+\mathcal{O}_{\mathcal{E}}\otimes \mathfrak{m}_R$ in $\mathcal{O}_{\mathcal{E}}\otimes_{\mathcal{O}_L} R$ is generated by $\mathfrak{m}_R$, it is clear that $\varphi_q$ maps $\mathfrak{m}_{\mathcal{O}_{\mathcal{E}}}\otimes R+\mathcal{O}_{\mathcal{E}}\otimes \mathfrak{m}_R$ to itself. Then we have the following lemma.
	\begin{lemma}\label{flat}
		The homomorphism 
		\begin{equation*}
		\varphi_q : \mathcal{O}_R\rightarrow \mathcal{O}_R
		\end{equation*} 
		is faithfully flat.
	\end{lemma}
	\begin{proof}
		Since $\varphi_q$ is flat, the proof follows from Proposition \ref{faithfully flat}. 
	\end{proof}
As the $q$-Frobenius $\varphi_q$ on $\mathcal{O}_{\mathcal{E}}$ extends uniquely by functoriality and continuity to a $q$-Frobenius on $\mathcal{O}_{\widehat{\mathcal{E}^{ur}}}$, we also have a faithfully flat homomorphism from $\widehat{\mathcal{O}^{ur}_R}$ to $\widehat{\mathcal{O}^{ur}_R}$. %%%%  as in Lemma \ref{flat}.
	\begin{definition}
		An $R$-representation of $G_E$ is a finitely generated $R$-module with a continuous and $R$-linear action of $G_E$.
	\end{definition}
	\begin{definition}
		A $\varphi_q$-module over $\mathcal{O}_R$ is an $\mathcal{O}_R$-module $M$ together with a map  
		%%%%%%%\begin{equation*}
		$\varphi_M : M \rightarrow M$,
	%%%%%%%%	\end{equation*}
		which is semi-linear with respect to $\varphi_q$.
	\end{definition}
	\begin{remark}
		Let $M$ be an $\mathcal{O}_R$-module. Then a semi-linear map $\varphi_M: M\rightarrow M$ is equivalent to an $\mathcal{O}_R$-linear map $\varphi_M^{lin}: M_{\varphi_q}\rightarrow M$, where $M_{\varphi_q}:=M\otimes_{\mathcal{O}_R,\varphi_q}\mathcal{O}_R$ is the base change of $M$ by $\mathcal{O}_R$ via $\varphi_q$.
	\end{remark}
	Let ${\bf Rep}_R(G_E)$ denote the category of $R$-linear representations of $G_E$ and ${\bf Mod}_{/\mathcal{O}_R}^{\varphi_q}$ the category of $\varphi_q$-modules over $\mathcal{O}_R$. The morphisms in ${\bf Mod}_{/\mathcal{O}_R}^{\varphi_q}$ are $\mathcal{O}_R$-linear homomorphisms commuting with $\varphi$. Now we define a functor from ${\bf Rep}_R(G_E)$ to ${\bf Mod}_{/\mathcal{O}_R}^{\varphi_q}$. Let $V$ be an $R$-representation of $G_E$. Define
	\begin{align*}
	\mathbb{D}_R(V):= (\widehat{\mathcal{O}^{ur}_R}\otimes_R V)^{G_E}.
	\end{align*}
Here $G_E$ acts diagonally. Moreover, the multiplication by $\mathcal{O}_R$ on $\hat{\mathcal{O}}^{ur}_R\otimes_R V$ is $G_E$-equivariant, thus $\mathbb{D}_R(V)$ is an $\mathcal{O}_R$-module. %%%%	where $G_E$ acts diagonally. Then $\mathbb{D}_R(V)$ carries the structure of an $\mathcal{O}_R$-module as the multiplication by $\mathcal{O}_R$ on $\hat{\mathcal{O}}^{ur}_R\otimes_R V$ is $G_E$-equivariant. The $q$-Frobenius $\varphi_q$ on $\widehat{\mathcal{O}^{ur}_R}$ acts $G_E$-equivariantly. 
We extend the definition of $q$-Frobenius to $\widehat{\mathcal{O}^{ur}_R}\otimes_R V$ by taking trivial action of $\varphi_q$ on $V$, and then $\varphi_q$ commutes with the action of $G_E$. It induces a $\mathcal{O}_R$-module homomorphism 
	\begin{equation*}
	\varphi_{\mathbb{D}_R(V)} : \mathbb{D}_R(V)\rightarrow \mathbb{D}_R(V),
	\end{equation*} which is semi-linear with respect to $\varphi_q$. Then $V\mapsto\mathbb{D}_R(V)$ is a functor from ${\bf Rep}_R(G_E)$ to ${\bf Mod}_{/\mathcal{O}_R}^{\varphi_q}$. The following lemma shows that the functor $\mathbb{D}_R$ commutes with restriction of scalars.
	\begin{lemma}\label{finite length}
Let $V\in {\bf Rep}_R(G_E) $ such that $\mathfrak{m}_R^nV=0$ for some $n$. Then as $\mathcal{O}_L$-modules, we have
		\begin{equation*}
		\mathbb{D}_{LT}(V) \cong \mathbb{D}_R(V).
		\end{equation*} 
	\end{lemma}
	\begin{proof}
We use induction on $n$. First assume that $\mathfrak{m}_R V=0$. %%%Since $V$ is an $R$-representation of $G_E$, it is finitely generated as an $R$-module. 
Then %%%$\mathfrak{m}_R V = 0$ implies that 
$V$ is finitely generated as an $R/\mathfrak{m}_R$-module. But %%%we know that 
$R/\mathfrak{m}_R=k_R$ is the residue field of $R$ and it is finite. %%% Since $R$ is an $\mathcal{O}_K$-algebra, then it follows that  $k_R$ is a finite extension of $k$, where $k$ is the residue field of $\mathcal{O}_K$. 
So by using Nakayama's lemma (for local rings), $V$ is finitely generated as an $\mathcal{O}_L$-module. Next, suppose that the statement is true for $n-1$, i.e., if $\mathfrak{m}_R^{n-1}W=0$ for any $R$-module $W$, then $W$ is finitely generated as an $\mathcal{O}_L$-module. Now let $\mathfrak{m}_R^nV=0$. Consider the exact sequence
	\begin{equation*}
	0\rightarrow \mathfrak{m}_R^{n-1}V \rightarrow V\rightarrow V/\mathfrak{m}_R^{n-1}V\rightarrow 0.
	\end{equation*}
	Then by using induction hypothesis, $\mathfrak{m}_R^{n-1}V$ and $V/\mathfrak{m}_R^{n-1}V$ are finitely generated as $\mathcal{O}_L$-modules. Thus, $V$ is finitely generated as an $\mathcal{O}_L$-module. Hence 
	\begin{equation*}
	\widehat{\mathcal{O}^{ur}_R}\otimes_R V = (\mathcal{O}_{\widehat{\mathcal{E}^{ur}}}\hat{\otimes}_{\mathcal{O}_L}R)\hat{\otimes}_R V\cong \mathcal{O}_{\widehat{\mathcal{E}^{ur}}}\hat{\otimes}_{\mathcal{O}_L}V=\mathcal{O}_{\widehat{\mathcal{E}^{ur}}}\otimes_{\mathcal{O}_L}V.
	\end{equation*}
	Here the first equality follows from the fact that $\widehat{\mathcal{O}^{ur}_R}$ is complete and $V$ is finitely generated as an $R$-module. The last one uses that %%%%equality follows from the fact that 
	$\mathcal{O}_{\widehat{\mathcal{E}^{ur}}}$ is complete, and $V$ is finitely generated as an $\mathcal{O}_L$-module. Then taking $G_E$-invariants, we get the desired result.
	\end{proof}
Note that using the similar proofs as in \cite{Dee}, it is easy to prove (even if we replace the absolute Frobenius with $q$-Frobenius) that  %%%%using the same lines, we can prove similar results as in section $2.1$ in \cite{Dee} even if we replace the absolute Frobenius with $q$-Frobenius. Thus 
the functor $\mathbb{D}_R$ is an exact and faithful functor, and it commutes with restriction of scalars and inverse limits. Moreover, for any $V \in {\bf Rep}_R(G_E)$, %%%the module 
$\mathbb{D}_R(V)$ is finitely generated as an $\mathcal{O}_R$-module. Also, the canonical $\widehat{\mathcal{O}^{ur}_R}$-linear homomorphism of $G_E$-modules
	\begin{equation*}
	\widehat{\mathcal{O}^{ur}_R}\otimes_{\mathcal{O}_R}\mathbb{D}_R(V)\rightarrow \widehat{\mathcal{O}^{ur}_R}\otimes_R V
	\end{equation*}
	is an isomorphism. 

Next, we define a full subcategory of ${\bf Mod}^{\varphi_q}_{/\mathcal{O}_R}$, which is an essential image of the functor $\mathbb{D}_R$.	
%%%	Now we introduce a category which is a full subcategory of ${\bf Mod}_{/\mathcal{O}_R}^{\varphi_q}$, and we show that this category is the essential image of $\mathbb{D}_R$.
	\begin{definition}
A $\varphi_q$-module $M$ over $\mathcal{O}_R$ is said to be \emph{\'{e}tale} if $\Phi_M^{lin}$ is an isomorphism, and $M$ is finitely generated as an $\mathcal{O}_R$-module.		
	\end{definition} 
Let ${\bf Mod}_{/\mathcal{O}_R}^{\varphi_q,\acute{e}t}$ denote the category of \'{e}tale $\varphi_q$-modules. A morphism of \'{e}tale $\varphi_q$-modules is a morphism of the underlying $\varphi_q$-modules. Then it follows from \cite[Lemma 2.1.16 and 2.1.17]{Dee} that the category ${\bf Mod}_{/\mathcal{O}_R}^{\varphi_q,\acute{e}t}$ is an abelian category. Moreover, it is stable under sub-objects, quotients, tensor products and $\mathbb{D}_R(V) \in {\bf Mod}_{/\mathcal{O}_R}^{\varphi_q,\acute{e}t}$. Now we introduce a functor, which is an inverse functor to $\mathbb{D}_R$. The functor 
	\begin{equation*}
	\mathbb{V}_R : {\bf Mod}_{/\mathcal{O}_R}^{\varphi_q,\acute{e}t} \rightarrow {\bf Rep}_R(G_K)
	\end{equation*} is defined as the following. 
	
	Let $M$ be an \'{e}tale $\varphi_q$-module over $\mathcal{O}_R$. Then view $\widehat{\mathcal{O}^{ur}_R}\otimes_{\mathcal{O}_R} M$ as a $\varphi_q$-module via
	\begin{equation*}
	\varphi_{\widehat{\mathcal{O}^{ur}_R}\otimes_{\mathcal{O}_R} M}(\lambda\otimes m) = \varphi_q(\lambda)\otimes \varphi_M(m) \quad \text{for}\ \lambda \in \widehat{\mathcal{O}^{ur}_R}, m \in M.
	\end{equation*}
For simplicity,	we write $\varphi_q\otimes\varphi_M$ rather than $\varphi_{\widehat{\mathcal{O}^{ur}_R}\otimes_{\mathcal{O}_R} M}$. The Galois group $G_E$-acts on $\widehat{\mathcal{O}^{ur}_R}\otimes_{\mathcal{O}_R} M$ via its action on $\widehat{\mathcal{O}^{ur}_R}$ and the group action commutes with the action of $\varphi_q$. Define
	\begin{equation*}
	\mathbb{V}_R(M) := (\widehat{\mathcal{O}^{ur}_R}\otimes_{\mathcal{O}_R} M)^{\varphi_q\otimes\varphi_M=id},
	\end{equation*}
	which is a sub $R$-module stable under the action of $G_E$. Thus $M \mapsto \mathbb{V}_R(M)$ is a functor from ${\bf Mod}_{/\mathcal{O}_R}^{\varphi_q,\acute{e}t}$ to the category ${\bf Rep}_R(G_E)$.
	
	Next, without any extra work using \cite[Proposition $2.1.21.$]{Dee}, we can easily show that the functor $\mathbb{V}_R$ commutes with the inverse limits. {\C Let $V\in {\bf Rep}_R(G_E)$ and $\mathfrak{m}_R^nV$ is a sub-module of $V$ generated by elements of the form $mv$ for $m\in \mathfrak{m}_R^n$ and $v\in V$. Define $V_n:= V/\mathfrak{m}_R^nV$.}
	\begin{lemma}
		\label{surjective V}
		Let $V$ be an $R$-representation of $G_E$. Then $\varphi_q\otimes id_V-id$ is a surjective homomorphism of abelian groups acting on $\widehat{\mathcal{O}^{ur}_R}\otimes_R V$.
	\end{lemma}
	\begin{proof}
		Suppose 
		$\mathfrak{m}_R V= 0$. Then the map $\varphi_q-id : E^{sep}\rightarrow E^{sep}$ is surjective, as the polynomial $x^q-x-\lambda$ is separable, for all $\lambda \in E^{sep}$. As $k_R$ is finite extension of $\mathbb{F}_p$ and $\varphi_q$ acts trivially on $k_R$, the map
		\begin{equation*}
		\varphi_q-id : E^{sep}\otimes_{\mathbb{F}_p}k_R\rightarrow E^{sep}\otimes_{\mathbb{F}_p}k_R
		\end{equation*}
		is also surjective. Also, $\varphi_q-id$ is continuous, thus the map
		\begin{equation*}
		\varphi_q-id: k_{\widehat{\mathcal{O}^{ur}_R}}=E^{sep}\hat{\otimes}_{\mathbb{F}_p}k_R\rightarrow E^{sep}\hat{\otimes}_{\mathbb{F}_p}k_R
		\end{equation*}
		is surjective. %%% where $k_{\widehat{\mathcal{O}^{ur}_R}}$ is the residue field of $\widehat{\mathcal{O}^{ur}_R}$. 
	Note that $V_1=V/\mathfrak{m}_RV$ is free over $k_R$ and $\varphi_q$ acts on $k_{\widehat{\mathcal{O}^{ur}_R}}\otimes_{k_R}V_1$ via its action on $k_{\widehat{\mathcal{O}^{ur}_R}}$, it follows that $\varphi_q\otimes id_{V_1}-id$ is surjective on $k_{\widehat{\mathcal{O}^{ur}_R}}\otimes_{k_R}V_1$. Then by d\'{e}vissage, 
		\begin{equation*}
		\varphi_q\otimes id_{V_n}-id: \widehat{\mathcal{O}^{ur}_R}/\mathfrak{m}_R^n\otimes_{R}V_n \rightarrow \widehat{\mathcal{O}^{ur}_R}/\mathfrak{m}_R^n\otimes_{R}V_n
		\end{equation*}
		is surjective.
		Here $\widehat{\mathcal{O}^{ur}_R}/\mathfrak{m}_R^n$ is Artinian and $V_n$ has finite length, so the Mittag-Leffler condition holds for $\widehat{\mathcal{O}^{ur}_R}/\mathfrak{m}_R^n\otimes_{R}V_n$. Then by passage to the inverse limits, the result holds for general $V$.
	\end{proof}
	To prove the next proposition, we need the following analogous result of Lemma \ref{finite length}, which can be easily proved.
	\begin{lemma}\label{finite-length M}
		If $\mathfrak{m}_R^n M = 0$, then as $\mathcal{O}_L$-modules
		\begin{equation*}
		\mathbb{V}_{LT}(M) = \mathbb{V}_R(M).
		\end{equation*}
	\end{lemma}
{\C Define $M_n:= M/\mathfrak{m}_R^nM$, where $\mathfrak{m}_R^nM$ is a sub-module of $M$.}
%%%	The above lemma shows that the functor $\mathbb{V}_R$ commutes with the restriction of scalars.
	\begin{proposition} \label{surjective}
		Let $M$ be an \'{e}tale $\varphi_q$-module. Then $\varphi_q\otimes \varphi_M-id$ is a surjective homomorphism of abelian groups on $\widehat{\mathcal{O}^{ur}_R}\otimes_{\mathcal{O}_R} M$.
	\end{proposition}
	\begin{proof}
If $\mathfrak{m}_R^nM=0$, then by Lemma \ref{finite-length M}, we have 
	%	\begin{equation*}
		$\mathbb{V}_{LT}(M) = \mathbb{V}_R(M)$ as an $\mathcal{O}_L$-module.
	%	\end{equation*} 
Now by using \cite{KR}, it follows that 
		\begin{equation*}
		\widehat{\mathcal{O}^{ur}_R}\otimes_R\mathbb{V}_R(M_n)\rightarrow \widehat{\mathcal{O}^{ur}_R}\otimes_{\mathcal{O}_R} M_n
		\end{equation*}
		is an isomorphism. Moreover, this isomorphism respects the action of $\varphi_q\otimes\varphi_M$. Then by using Lemma \ref{surjective V}, the map $\varphi_q\otimes\varphi_M-id$ is surjective on $\widehat{\mathcal{O}^{ur}_R}\otimes_{\mathcal{O}_R} M_n$, and the general case follows by passing to the inverse limits.
	\end{proof}
	\begin{proposition}\label{exact V}
		The functor $\mathbb{V}_R$ is an exact functor.
	\end{proposition}
	\begin{proof}
		Let $0\rightarrow M \rightarrow M^{\prime} \rightarrow M^{\prime\prime} \rightarrow 0$ be an exact sequence of \'{e}tale $\varphi_q$-modules. Then we have the following commutative diagram with the exact rows,	
		\begin{center}
			\begin{tikzcd}
			0\arrow{r} & \widehat{\mathcal{O}^{ur}_R}\otimes_{\mathcal{O}_R}M \arrow{r} \arrow[swap]{d}{\varphi_q\otimes\varphi_M-id} &  \widehat{\mathcal{O}^{ur}_R}\otimes_{\mathcal{O}_R}M^{\prime} \arrow{r} \arrow[swap]{d}{\varphi_q\otimes\varphi_{M^\prime}-id} &  \widehat{\mathcal{O}^{ur}_R}\otimes_{\mathcal{O}_R}M^{\prime\prime} \arrow{r} \arrow[swap]{d}{\varphi_q\otimes\varphi_{M^{\prime \prime}}-id} & 0\\
			0 \arrow{r} & \widehat{\mathcal{O}^{ur}_R}\otimes_{\mathcal{O}_R}M  \arrow{r} & \widehat{\mathcal{O}^{ur}_R}\otimes_{\mathcal{O}_R}M^{\prime} \arrow{r} & \widehat{\mathcal{O}^{ur}_R}\otimes_{\mathcal{O}_R}M^{\prime\prime} \arrow{r} & 0.
			\end{tikzcd}
		\end{center}
		Now by applying the Snake lemma, we get an exact sequence
		\begin{equation*}
		0 \rightarrow \mathbb{V}_R(M)\rightarrow \mathbb{V}_R(M^{\prime})\rightarrow \mathbb{V}_R(M^{\prime\prime})\rightarrow \widehat{\mathcal{O}^{ur}_R}\otimes_{\mathcal{O}_R}M/(\varphi_q\otimes\varphi_M-id)\rightarrow \cdots. 
		\end{equation*}
		By Lemma \ref{surjective}, we know that the map $\varphi_q\otimes\varphi_M-id$ is a surjective homomorphism acting on $\widehat{\mathcal{O}^{ur}_R}\otimes_{\mathcal{O}_R}M$, so the last term is zero, and the sequence 
		\begin{equation*}
		0 \rightarrow \mathbb{V}_R(M)\rightarrow \mathbb{V}_R(M^\prime)\rightarrow \mathbb{V}_R(M^{\prime\prime})\rightarrow 0
		\end{equation*} is an exact sequence. Hence the functor $\mathbb{V}_R$ is exact.  
	\end{proof}
	Moreover, for an \'{e}tale $\varphi_q$-module $M$, the module $\mathbb{V}_R(M)$ is finitely generated over $R$, and the homomorphism of $\widehat{\mathcal{O}^{ur}_R}$-modules
	\begin{equation*}
	\widehat{\mathcal{O}^{ur}_R}\otimes_R \mathbb{V}_R(M)\rightarrow \widehat{\mathcal{O}^{ur}_R}\otimes_{\mathcal{O}_R} M
	\end{equation*}
	is an isomorphism. The proof is similar to \cite[Proposition 2.1.26]{Dee}. Next, the following
%%%	The next 
theorem gives the equivalence of categories between ${\bf Rep}_R(G_E)$ and ${\bf Mod}_{/\mathcal{O}_R}^{\varphi_q,\acute{e}t}.$
	\begin{theorem}\label{Main6.1}
		The functor
		\begin{equation*}
		\mathbb{D}_R: {\bf Rep}_R(G_E) \rightarrow {\bf Mod}_{/\mathcal{O}_R}^{\varphi_q,\acute{e}t}
		\end{equation*}
		is an equivalence of categories with quasi-inverse functor 
		\begin{equation*}
		\mathbb{V}_R: {\bf Mod}_{/\mathcal{O}_R}^{\varphi_q,\acute{e}t}\rightarrow {\bf Rep}_R(G_E). 
		\end{equation*}
	\end{theorem}
	\begin{proof}
		It is enough to construct functorial isomorphisms
		\begin{equation*}
		\mathbb{V}_R(\mathbb{D}_R(V))\xrightarrow{\sim} V \quad \text{and} \quad\mathbb{D}_R(\mathbb{V}_R(M))\xrightarrow{\sim} M
		\end{equation*}
		for an $R$-representation $V$ of $G_E$ and an \'{e}tale $\varphi_q$-module $M$ over $\mathcal{O}_R$, respectively.
		Consider the isomorphism of $G_E$-modules
		\begin{equation*}
		\widehat{\mathcal{O}^{ur}_R}\otimes_{\mathcal{O}_R}\mathbb{D}_R(V)\rightarrow \widehat{\mathcal{O}^{ur}_R}\otimes_R V.
		\end{equation*} Taking $\varphi_q\otimes\varphi_M$-invariants, we obtain an isomorphism
		\begin{equation*}
		\mathbb{V}_R(\mathbb{D}_R(V))\rightarrow (\widehat{\mathcal{O}^{ur}_R}\otimes_R V)^{\varphi_q\otimes\varphi_M=id}.
		\end{equation*}
	Note that $V$ has trivial action of $\varphi_q\otimes\varphi_M$, so there is a map
		\begin{equation*}
		V\rightarrow (\widehat{\mathcal{O}^{ur}_R}\otimes_R V)^{\varphi_q\otimes\varphi_M=id}.
		\end{equation*}
	Now for finite length modules, the above map is an isomorphism by using Theorem \ref{Kisin Ren}. By taking the inverse limits, the map is an isomorphism for the general $V$. Hence
		\begin{equation*}
		\mathbb{V}_R(\mathbb{D}_R(V))\rightarrow V.
		\end{equation*}
		Similarly, the map
		\begin{equation*}
		M\rightarrow (\widehat{\mathcal{O}^{ur}_R}\otimes_{\mathcal{O}_R} M)^{G_E}
		\end{equation*}
		is an isomorphism. Moreover, we have an isomorphism 
		\begin{equation*}
		\widehat{\mathcal{O}^{ur}_R}\otimes_R \mathbb{V}_R(M)\rightarrow \widehat{\mathcal{O}^{ur}_R}\otimes_{\mathcal{O}_R} M.
		\end{equation*}
		Then taking $G_E$-invariant we have
		\begin{equation*}
		\mathbb{D}_R(\mathbb{V}_R(M))=(\widehat{\mathcal{O}^{ur}_R}\otimes_R \mathbb{V}_R(M))^{G_E}\xrightarrow{\sim} (\widehat{\mathcal{O}^{ur}_R}\otimes_{\mathcal{O}_R} M)^{G_E}.
		\end{equation*}
		Therefore 
		\begin{equation*}
		\mathbb{D}_R(\mathbb{V}_R(M))\rightarrow M
		\end{equation*}
		is an isomorphism, and this proves the theorem.
	\end{proof}
	\begin{remark}
		The functors $\mathbb{D}_R$ and $\mathbb{V}_R$  are compatible with the tensor product.
	\end{remark}
	\subsubsection{The characteristic zero case}\label{sub6.2}
	Let $K$ be a local field of characteristic $0$. {\col For a finite extension $L$ of $\mathbb{Q}_p$ contained in $K$,} recall that the ring $\mathcal{O}_{\mathcal{E}}$ is the $p$-adic completion of $W_L[[X]][\frac{1}{X}]$, and $\mathcal{O}_{\mathcal{E}^{ur}}$ is the maximal integral unramified extension of $\mathcal{O}_{\mathcal{E}}$. The ring $\mathcal{O}_{\widehat{\mathcal{E}^{ur}}}$ is the $p$-adic completion of  $\mathcal{O}_{\mathcal{E}^{ur}}$. Also, the Galois group $H_K=\Gal(\bar{K}/K_\infty)$ is identified with $G_E$, where $E$ is the field of norms of the extension $K_\infty/K$ and it is a field of characteristic $p$.
	
	{\C Note that the action of $G_K$ on $\mathcal{E}^{ur}$ induces a $G_K$-action on $\mathcal{O}_{\widehat{\mathcal{E}^{ur}}}$. Now by taking the trivial action of $G_K$ on $R$, it induces a $G_K$-action on $\mathcal{O}_{\widehat{\mathcal{E}^{ur}}}\otimes_{\mathcal{O}_L} R$. Since $G_K$ acts continuously on $\mathcal{E}^{ur}$, the action is continuous on $\mathcal{O}_{\widehat{\mathcal{E}^{ur}}}\otimes_{\mathcal{O}_L} R$. This induces a $G_K$-action on $\widehat{\mathcal{O}^{ur}_R}$, which is continuous with respect to the $\mathfrak{m}_R\widehat{\mathcal{O}^{ur}_R}$-adic topology.}
	
	Let $V$ be an $R$-representation of $G_K$. Then 
	\begin{equation*}
	\mathbb{D}_R(V):=(\widehat{\mathcal{O}^{ur}_R}\otimes_R V)^{H_K} = (\widehat{\mathcal{O}^{ur}_R}\otimes_{R}V)^{G_E}
	\end{equation*}
	is a $\varphi_q$-module over $\mathcal{O}_R$. The $G_K$-action on $\widehat{\mathcal{O}^{ur}_R}\otimes_{R}V$ induces a semi-linear action of $G_K/H_K = \Gamma_{LT} = \Gal(K_\infty/K)$ on $\mathbb{D}_R(V)$. 
	%%%Now we introduce the category of $(\varphi_q,\Gamma_{LT})$-modules over $\mathcal{O}_R$. 
	\begin{definition}
	A \emph{$(\varphi_q,\Gamma_{LT})$-module} $M$ over $\mathcal{O}_R$ is a $\varphi_q$-module over $\mathcal{O}_R$ equipped with a continuous semi-linear action of $\Gamma_{LT}$, which commutes with the endomorphism $\varphi_M$ of $M$, and a $(\varphi_q,\Gamma_{LT})$-module is \emph{\'{e}tale} if its underlying $\varphi_q$-module is \'{e}tale.
	\end{definition} 
We write ${\bf Mod}_{/\mathcal{O}_R}^{\varphi_q,\Gamma_{LT},\acute{e}t}$ for the category of \'{e}tale $(\varphi_q,\Gamma_{LT})$-modules over $\mathcal{O}_R$. Then $\mathbb{D}_R$ is a functor from the category of $R$-linear representations of $G_K$ to the category of \'{e}tale $(\varphi_q,\Gamma_{LT})$-modules over $\mathcal{O}_R$.
	
	If $M$ is an \'{e}tale $(\varphi_q,\Gamma_{LT})$-module over $\mathcal{O}_R$, then
	\begin{equation*}
	\mathbb{V}_R(M) = (\widehat{\mathcal{O}^{ur}_R}\otimes_{\mathcal{O}_R}M)^{\varphi_q\otimes \varphi_M=id}
	\end{equation*}
is an $R$-representation of $G_K$. The group $G_K$ acts on $\widehat{\mathcal{O}^{ur}_R}$ as before and acts via $\Gamma_{LT}$ on $M$. The $G_K$ action on $\widehat{\mathcal{O}^{ur}_R}\otimes_{\mathcal{O}_R}M$ is $\varphi_q\otimes\varphi_M$-equivariant, and this induces a $G_K$ action on $\mathbb{V}_R(M)$.
	
For any $V \in {\bf Rep}_R(G_K)$, there is a canonical $R$-linear homomorphism
	\begin{equation*}
	V\rightarrow \mathbb{V}_R(\mathbb{D}_R(V)).
	\end{equation*} 
	of representations of $G_K$. By Theorem \ref{Main6.1}, this is an isomorphism when restricted to $H_K$, so it must be an isomorphism of $G_K$-representations. Similarly, for an \'{e}tale $(\varphi_q,\Gamma_{LT})$-module $M$, the canonical homomorphism of \'{e}tale $(\varphi_q,\Gamma_{LT})$-modules
	\begin{equation*}
	M\rightarrow \mathbb{D}_R(\mathbb{V}_R(M))
	\end{equation*}
	is an isomorphism. Moreover, by using Theorem \ref{Main6.1}, the underlying map of $\varphi_q$-modules is an isomorphism, and this proves the following theorem.
	\begin{theorem} \label{Main6.2}
		The functor $\mathbb{D}_R$ is an equivalence of categories between the category ${\bf Rep}_R(G_K)$ %%%%the category 
		of $R$-linear representations of $G_K$ and the category ${\bf Mod}_{/\mathcal{O}_R}^{\varphi_q,\Gamma_{LT},\acute{e}t}$ %%%%the category 
		of \'{e}tale $(\varphi_q,\Gamma_{LT})$-modules with quasi-inverse functor $\mathbb{V}_R$.
	\end{theorem}
Next, we extend the functor $\mathbb{D}_R$ to the category ${\bf Rep}_{\mathfrak{m}_R-tor}^{dis}(G_K)$ of discrete $\mathfrak{m}_R$-primary {\C $\mathcal{O}_R$-modules} %%abelian groups 
with a continuous and linear action of $G_K$. Any object in ${\bf Rep}_{\mathfrak{m}_R-tor}^{dis}(G_K)$ is the filtered direct limit of $\mathfrak{m}_R$-power torsion objects in ${\bf Rep}_R(G_K)$. For any $V\in {\bf Rep}_{\mathfrak{m}_R-tor}^{dis}(G_K)$, define
\begin{equation*}
\mathbb{D}_R(V)=(\widehat{\mathcal{O}^{ur}_R}\otimes_R V)^{H_K}.
\end{equation*}
Note that the functor $\mathbb{D}_R$ commutes with the direct limits as the tensor product and taking $H_K$-invariants commute with the direct limits. Then $\mathbb{D}_R(V)$ is an object into the category $\varinjlim {\bf Mod}_{/\mathcal{O}_R}^{\varphi_q,\Gamma_{LT},\acute{e}t,tor}$ of injective limits of $\mathfrak{m}_R$-power torsion objects in ${\bf Mod}_{/\mathcal{O}_R}^{\varphi_q,\Gamma_{LT},\acute{e}t}$. For any $M\in \varinjlim {\bf Mod}_{/\mathcal{O}_R}^{\varphi_q,\Gamma_{LT},\acute{e}t,tor}$, put
\begin{equation*}
\mathbb{V}_R(M) = (\widehat{\mathcal{O}^{ur}_R}\otimes_{\mathcal{O}_R}M)^{\varphi_q\otimes \varphi_M=id}.
\end{equation*} 
Then the functor $\mathbb{V}_R$ also commutes with the direct limits, and we have the following result.
\begin{proposition}\label{7.19}
	The functors $\mathbb{D}_R$ and $\mathbb{V}_R$ are quasi-inverse equivalences of categories between the category ${\bf Rep}_{\mathfrak{m}_R-tor}^{dis}(G_K)$ and $\varinjlim {\bf Mod}_{/\mathcal{O}_R}^{\varphi_q,\Gamma_{LT},\acute{e}t,tor}$.
\end{proposition}
\begin{proof}
	Since the functors $\mathbb{D}_R$ and $\mathbb{V}_R$ commute with the direct limits, the proposition follows from Theorem \ref{Main6.2} by taking the direct limits.
\end{proof}
	\section{Galois Cohomology over the Coefficient Ring} \label{sec7}
By Proposition \ref{7.19}, the functor $\mathbb{D}_R$ is a quasi-inverse equivalence of categories between the category ${\bf Rep}_R(G_K)\ (\text{resp.,}\ {\bf Rep}_{\mathfrak{m}_R-tor}^{dis}(G_K))$ and ${\bf Mod}_{/\mathcal{O}_R}^{\varphi_q,\Gamma_{LT},\acute{e}t}$ $(\text{resp.,}\ \varinjlim {\bf Mod}_{/\mathcal{O}_R}^{\varphi_q,\Gamma_{LT},\acute{e}t,tor})$. The following theorem is a generalization of Theorem \ref{lattices} over the coefficient rings.
	\begin{theorem} \label{Main7}
	Let $V$ be an $R$-representation of $G_K$. Then there is a natural isomorphism 
		\begin{equation*}
		H^i(G_K,V)\cong \mathcal{H}^i(\Phi\Gamma_{LT}^{\bullet}(\mathbb{D}_{R}(V)))\quad \text{for} \ i\geq 0.
		\end{equation*}
	\end{theorem}
	\begin{proof}
First assume that $\mathfrak{m}_R^n V= 0$ for some $n\in \mathbb{N}$. Then by Lemma \ref{finite length} and Theorem \ref{lattices}, we have 
		\begin{equation*}
		H^i(G_K,V)\cong \mathcal{H}^i(\Phi\Gamma_{LT}^{\bullet}(\mathbb{D}_{R}(V))).
		\end{equation*}
 Next, it follows from \cite[Theorem 2.1]{Ta1} and \cite[Corollary 2.2]{Tate} that the functor $H^i(G_K,-)$ commutes with the inverse limits. Moreover, we have
%\begin{equation*}
$\mathbb{D}_R(V)\xrightarrow{\sim} \varprojlim \mathbb{D}_R(V_n)$.
%\end{equation*}
Observe that the modules $\mathbb{D}_R(V_n)$ are finitely generated over the Artinian ring $\mathcal{O}_R/\mathfrak{m}_R^n\mathcal{O}_R$, so the inverse limit functor is an exact functor on the category of $\mathfrak{m}_R$-power torsion \'{e}tale $(\varphi_q,\Gamma_{LT})$-modules over $\mathcal{O}_R$. Then we have 
\begin{equation*}
\mathcal{H}^i(\Phi\Gamma_{LT}^{\bullet}(\mathbb{D}_{R}(V)))=\varprojlim \mathcal{H}^i(\Phi\Gamma_{LT}^{\bullet}(\mathbb{D}_{R}(V_n))).
\end{equation*} 
Hence the general case follows by passing to the inverse limits.		
	\end{proof}
Next, in order to generalize the Theorem \ref{Main5} to the case of coefficient rings; first, we extend the operator $\psi:= \psi_{\mathbb{D}_{LT}(V)}$ to $\mathbb{D}_R(V)$.
	
	As $\psi_q$ maps $\mathcal{O}_{\widehat{\mathcal{E}^{ur}}}$ to $\mathcal{O}_{\widehat{\mathcal{E}^{ur}}}$, we extend $\psi_q$ to $\mathcal{O}_{\widehat{\mathcal{E}^{ur}}}\otimes_{\mathcal{O}_L} R$ by making it trivially act on $R$. Moreover, it maps $\mathfrak{m}_{\mathcal{O}_{\widehat{\mathcal{E}^{ur}}}}\otimes R+ \mathcal{O}_{\widehat{\mathcal{E}^{ur}}}\otimes \mathfrak{m}_R$ to itself, inducing an $R$-linear map
	\begin{equation*}
	\psi_q: \widehat{\mathcal{O}^{ur}_R}\rightarrow \widehat{\mathcal{O}^{ur}_R}.
	\end{equation*} 
	Moreover, as $\psi_q$ acts Galois equivariantly, making it act on $\widehat{\mathcal{O}^{ur}_R}\otimes_{\mathcal{O}_L} V$ via its action on $\widehat{\mathcal{O}^{ur}_R}$, we have an operator $\psi_{\mathbb{D}_R(V)}$ on $\mathbb{D}_R(V)$. %Then we have the following theorem.
\begin{theorem}\label{Main7.2}
	Let $V \in {\bf Rep}_{\mathfrak{m}_R-tor}^{dis}(G_K)$. Then we have a well-defined homomorphism 
	\begin{equation*}
		\mathcal{H}^i(\Phi\Gamma_{LT}^{\bullet}(\mathbb{D}_R(V)))\rightarrow \mathcal{H}^i(\Psi\Gamma_{LT}^{\bullet}(\mathbb{D}_R(V)))\quad \text{for}\ i\geq 0.
		\end{equation*}
	Moreover, the homomorphism $\mathcal{H}^0(\Phi\Gamma_{LT}^{\bullet}(\mathbb{D}_R(V)))\rightarrow \mathcal{H}^0(\Psi\Gamma_{LT}^{\bullet}(\mathbb{D}_R(V))$ is injective.\end{theorem}
\begin{proof}
If $V$ is a finite abelian $\mathfrak{m}_R$ group, then the theorem follows from Lemma \ref{finite length} and Theorem \ref{Main5}. Also, the functors $\mathcal{H}^i(\Phi\Gamma_{LT}^{\bullet}(\mathbb{D}_R(-)))$ and $\mathcal{H}^i(\Psi\Gamma_{LT}^{\bullet}(\mathbb{D}_R(-)))$ commute with the inverse limits. Hence the result follows for general $V$ by passing to the inverse limits.
\end{proof}	
The following theorem and its proof is a generalization of \cite[Theorem 5.13]{SV} to the case of coefficient rings. It is possible that this approach leads to the construction of a Perrin-Riou homomorphism for Galois representation defined over the coefficient ring $R$.
	\begin{theorem}\label{Main7.3}
		Let $V \in {\bf Rep}_R(G_K)$. Then we have 
		\begin{equation*}
		H^i_{Iw}({K_{\infty}/K},V)\cong \mathcal{H}^{i-1}(\underline{\Psi}^{\bullet}(\mathbb{D}_{R}(V(\chi_{cyc}^{-1}\chi_{LT})))) \quad \text{for}\ i\geq 1.
		\end{equation*}
In particular, $H^i_{Iw}({K_{\infty}/K},V)=0$ for $i\geq3$.
	\end{theorem}
	\begin{proof}
Suppose that $V$ has a finite length. Then 
%%	\begin{equation*}
$\mathbb{D}_R(V) = \mathbb{D}_{LT}(V)$
%%%	\end{equation*}	
as an $\mathcal{O}_L$-module. Thus $\psi_{\mathbb{D}_R(V)}$ agrees with the $\psi_{\mathbb{D}_{LT}(V)}$. Now by Corollary \ref{Iwasawa cor.}, we have
\begin{equation*}
H^i_{Iw}({K_{\infty}/K},V)\cong \mathcal{H}^{i-1}(\underline{\Psi}^{\bullet}(\mathbb{D}_{R}(V(\chi_{cyc}^{-1}\chi_{LT})))) \quad \text{for}\ i\geq1.
\end{equation*}
Moreover, it follows from Lemma \ref{commutes inverse} that the functor $H^i_{Iw}(K_{\infty}/K,-)$ also commutes with the inverse limits. {\cv Now let $V \in {\bf Rep}_R(G_K)$. Then consider the exact sequence of projective systems for $V_n:= V/\mathfrak{m}_R^nV$;
\begin{equation*}
0\rightarrow H^1_{Iw}({K_{\infty}/K},V_n)\rightarrow \mathbb{D}_{R}(V_n(\chi_{cyc}^{-1}\chi_{LT}))\xrightarrow{\psi_{\mathbb{D}_{R}(V_n(\chi_{cyc}^{-1}\chi_{LT}))}-1}\mathbb{D}_{R}(V_n(\chi_{cyc}^{-1}\chi_{LT}))\rightarrow H^2_{Iw}({K_{\infty}/K},V_n)\rightarrow 0.
\end{equation*}
We now show that this exact sequence remains exact on taking the projective limit. For this, it is enough to show that $R^1\varprojlim$ of the two projective systems 
\begin{equation*}
\{{H^1_{Iw}({K_{\infty}/K},V_n)}\}_n\ \mbox{and}\ \{{(\psi_{\mathbb{D}_{R}(V(\chi_{cyc}^{-1}\chi_{LT}))}-1)\mathbb{D}_{R}(V_n(\chi_{cyc}^{-1}\chi_{LT}))}\}_n
\end{equation*}
%%% $\{{H^1_{Iw}({K_{\infty}/K},V_n)}_n\}$ and $\{{(\psi_{\mathbb{D}_{R}(V(\chi_{cyc}^{-1}\chi_{LT}))}-1)\mathbb{D}_{R}(V_n(\chi_{cyc}^{-1}\chi_{LT}))}_n\}$
vanishes. Consider the open pro-$p$ subgroup $\Gamma_{LT}^*$, then $R[[\Gamma_{LT}^*]]$ is a complete local noetherian commutative ring, and $\{{H^1_{Iw}(K_\infty/K,V_n)}_n\}$ is a projective system of finitely generated $R[[\Gamma_{LT}^*]]$-modules. Hence by \cite[Theorem 8.1]{Jen}, $R^1\varprojlim$ of this projective system vanishes. Since 
\begin{equation*}
\mathbb{D}_{R}(V_n(\chi_{cyc}^{-1}\chi_{LT})) = \mathbb{D}_{R}(V(\chi_{cyc}^{-1}\chi_{LT}))/\mathfrak{m}_R^n\mathbb{D}_{R}(V(\chi_{cyc}^{-1}\chi_{LT})),
\end{equation*}
 therefore for the projective system $\{{(\psi_{\mathbb{D}_{R}(V(\chi_{cyc}^{-1}\chi_{LT}))}-1)\mathbb{D}_{R}(V_n(\chi_{cyc}^{-1}\chi_{LT}))}\}_n$, the transition maps are surjective %%as $\mathbb{D}_{R}(V_n(\chi_{cyc}^{-1}\chi_{LT})) = \mathbb{D}_{R}(V(\chi_{cyc}^{-1}\chi_{LT}))/\mathfrak{m}_R^n\mathbb{D}_{R}(V(\chi_{cyc}^{-1}\chi_{LT}))$, 
and hence the $R^1\varprojlim$ of this projective system vanishes. The exactness on the right follows.}

%%%	Remark 8.4. It is possible to extend Theorem 4.2 to the case of coeffici} Then by passing to the inverse limits, we deduce the theorem for general $V$.	
	\end{proof}
	\begin{remark}
	It is possible to extend Theorem \ref{False-Tate equivalence} to the case of coefficient rings, and using that we can prove that for any $V \in {\bf Rep}_R(G_K)$,
		\begin{equation*}
		H^i(G_K,V)\cong \mathcal{H}^i(\Phi\Gamma_{LT,FT}^{\bullet}(\mathbb{D}_{R}(V))).
		\end{equation*}
		This gives a  generalization of Theorem \ref{Main4} over the coefficient rings. We can also generalize Theorem \ref{Theorem False Tate} to the case of coefficient rings.
	\end{remark}
By Theorem \ref{Iwasawa cohomology}, we have
	\begin{equation*}
	H^1_{Iw}({K_{\infty}/K},\mathcal{O}_L(\chi_{cyc}\chi_{LT}^{-1}))\cong \mathbb{D}_{LT}(\mathcal{O}_L)^{\psi_{\mathbb{D}_{LT}(\mathcal{O}_L)}=id}.
	\end{equation*}
	The map
\begin{equation*}
\Exp^*:H^1_{Iw}({K_{\infty}/K},\mathcal{O}_L(\chi_{cyc}\chi_{LT}^{-1}))\rightarrow \mathbb{D}_{LT}(\mathcal{O}_L)^{\psi_{\mathbb{D}_{LT}(\mathcal{O}_L)}=id}
\end{equation*}
is called the \emph{dual exponential map}. These dual exponential maps occur in the construction of the Coates-Wiles homomorphisms. For more details about this dual exponential map, see \cite{SV}. We generalize the dual exponential map over the coefficient ring to check if one can extend the Coates-Wiles homomorphisms to Galois representations defined over $R$. 
\begin{theorem}\label{commutative}
Let $V\in {\bf Rep}_{R}(G_K)$. {\C Consider a map of local rings $R\rightarrow\mathcal{O}_L$. Then} %%Then {\C for any $\mathcal{O}_K$-algebra homomorphism $R\rightarrow \mathcal{O}_K$,} 
we have the following commutative diagram
\begin{center}
\begin{tikzcd}
H^i_{Iw}({K_{\infty}/K},V) \arrow{r}{\cong} \arrow{d} & \mathcal{H}^{i-1}(\underline{\Psi}^{\bullet}(\mathbb{D}_{R}(V(\chi_{cyc}^{-1}\chi_{LT})))) \arrow{d} \\
H^i_{Iw}({K_{\infty}/K},V\otimes_{R}\mathcal{O}_L)\arrow{r}[swap]{\cong}& \mathcal{H}^{i-1}(\underline{\Psi}^{\bullet}(\mathbb{D}_{LT}(V\otimes_{R}\mathcal{O}_L(\chi_{cyc}^{-1}\chi_{LT})))).
\end{tikzcd}
\end{center}
\end{theorem}
\begin{proof}
Note that the ring homomorphism $R\rightarrow \mathcal{O}_L$ induces a map from $V\rightarrow V\otimes_{R}\mathcal{O}_L$. Moreover, $H^i_{Iw}(K_\infty/K,V)= H^i(G_K,R[[\Gamma_{LT}]]\otimes_{R}V)$. Therefore, we have a well defined map $H^i_{Iw}(K_\infty/K, V)\rightarrow H^i_{Iw}({K_{\infty}/K},V\otimes_{R}\mathcal{O}_L)$. Similarly, the map $V\rightarrow V\otimes_{R}\mathcal{O}_L$ defines a map from $\mathbb{D}_R(V)\rightarrow \mathbb{D}_{LT}(V\otimes_{R}\mathcal{O}_L)$, and this induces a well-defined map $\mathcal{H}^{i-1}(\underline{\Psi}^{\bullet}(\mathbb{D}_{R}(V(\chi_{cyc}^{-1}\chi_{LT}))))\rightarrow\mathcal{H}^{i-1}(\underline{\Psi}^{\bullet}(\mathbb{D}_{LT}(V\otimes_{R}\mathcal{O}_L(\chi_{cyc}^{-1}\chi_{LT})))).$ Then the result follows from Theorem \ref{Iwasawa cohomology} and Theorem \ref{Main7.3}.
\end{proof}
 Next, we generalize the dual exponential map over coefficient rings.
\begin{corollary}\label{dual exp.}
There is a dual exponential map $$\Exp_R^*:H^1_{Iw}({K_{\infty}/K},R(\chi_{cyc}\chi_{LT}^{-1}))\xrightarrow{\sim}\mathcal{O}_R^{\psi_R=id}$$ over $R$, and {\C for any map $R\rightarrow\mathcal{O}_L$ of local rings}, the diagram
\begin{center}
\begin{tikzcd}
H^1_{Iw}({K_{\infty}/K},R(\chi_{cyc}\chi_{LT}^{-1})) \arrow{r}{\Exp_R^*} \arrow{d} & \mathcal{O}_R^{\psi_R=id} \arrow{d} \\
	H^1_{Iw}({K_{\infty}/K},\mathcal{O}_L(\chi_{cyc}\chi_{LT}^{-1}))\arrow{r}[swap]{\Exp^*}& \mathcal{O}_\mathcal{E}^{\psi=id},
	\end{tikzcd}
\end{center}
where $\psi_R= \psi_{\mathbb{D}_R(R)}$ and $\psi= \psi_{\mathbb{D}_{LT}(\mathcal{O}_L)}$, is commutative.
\end{corollary}
\begin{proof}
Since $R(\chi_{cyc}\chi_{LT}^{-1})$ is an $R$-representation of $G_K$, by Theorem \ref{Main7.3}, we have
\begin{align*}
H^1_{Iw}({K_{\infty}/K},R(\chi_{cyc}\chi_{LT}^{-1}))&\cong \mathcal{H}^0(\underline{\Psi}^{\bullet}(\mathbb{D}_{R}(R(\chi_{cyc}\chi_{LT}^{-1})(\chi_{cyc}^{-1}\chi_{LT})))\\&\cong\mathcal{O}_R^{\psi_R=id}.
\end{align*} 
Also, by Theorem \ref{Iwasawa cohomology}, we have
\begin{align*}  
 H^1_{Iw}({K_{\infty}/K},R(\chi_{cyc}\chi_{LT}^{-1})\otimes_{R}\mathcal{O}_L)&\cong H^1_{Iw}({K_{\infty}/K},\mathcal{O}_L(\chi_{cyc}\chi_{LT}^{-1}))\\& \cong \mathcal{H}^0(\underline{\Psi}^{\bullet}(\mathbb{D}_{LT}(\mathcal{O}_L(\chi_{cyc}\chi_{LT}^{-1})(\chi_{cyc}^{-1}\chi_{LT})))\\&\cong\mathcal{O}_\mathcal{E}^{\psi=id}.
 \end{align*}
  Now, the result follows from Theorem \ref{commutative} by putting $i=1$ and $V= R(\chi_{cyc}\chi_{LT}^{-1})$.
\end{proof}

\bibliography{mybib}{}
\bibliographystyle{acm}

\subsection*{Data Availability} Data sharing not applicable to this article as no datasets were generated or analyzed during the current study.

\subsection*{Conflict of interest} The authors declare that they have no conflicts of interest.

\end{document}